\documentclass[10pt]{article}
\usepackage[utf8]{inputenc} 
\usepackage{amsmath, amssymb, amsthm,graphicx}
\usepackage{natbib}
\newtheorem{rema}{Remark}

\newtheorem{lemma}{Lemma}
\newtheorem{corollary}{Corollary}
\newtheorem{prop}{Proposition}
\newtheorem{thm}{Theorem}
\newtheorem{defin}{Definition}

\def \n{\Vert}

\newcommand{\E}{\mathop{\mathbb{E}}}

\renewcommand{\P}{\mathop{\mathbb{P}}}

\def\R{{\mathbb{R}}}
\def\N{{\mathbb{N}}}

\def\F{{\cal{F}}}

\def\|{\,|\,}
\def\bn#1\en{\begin{align*}#1\end{align*}}
\def\bnn#1\enn{\begin{align}#1\end{align}}

\title{Worst-risk minimization in generalized structural equation models}

\begin{document} 
\date{}
\author{Philip Kennerberg and Ernst C. Wit}
\maketitle


\begin{abstract}
We consider extended structural equation models (SEMs), whereby a target of interest and its covariates are considered in several shifted environments. Given $k\in\N$ shift environments we consider the collection of all shifts that are at most $\gamma$-times as strong as a given weighted linear combination of these $k$ shifts together with its associated worst (quadratic) risk. This worst risk has a convenient decomposition with an explicit population minimizer. We consider its corresponding plug-in estimator. We show that this plug-in estimator is (almost surely) consistent, and satisfies a concentration in measure result. The solution to the worst risk minimizer is rather reminiscent of the corresponding ordinary least squares solution in that it is product of a vector and an inverse of a Grammian matrix. Due to this, the central moments of the plug-in estimator is not well-defined in general, but we instead consider these moments conditioned on the Grammian inverse being bounded by some given constant. We also study conditional variance of the estimator with respect to a natural filtration for the incoming data. Similarly we consider the conditional covariance matrix with respect to this filtration and prove a bound for the determinant of this matrix. This SEM model generalizes the linear models that have been studied previously for instance in the setting of casual inference or anchor regression but the concentration in measure result and the moment bounds are new even in the linear setting.
\end{abstract}

\section{Introduction}
Simple linear structural equation models, whereby a target is expressed via a linear relationship to its covariates, have previously been studied in various  casual inference settings. These models have a structure of the form
\begin{equation}\label{Geneq}
	\begin{bmatrix}
		Y\\
		X
	\end{bmatrix} = 
	B\cdot\begin{bmatrix}
		Y\\
		X
	\end{bmatrix}
	+\epsilon
	+M,
\end{equation}
where $\epsilon$ is some noise term, $M$ is some random (or deterministic) vector and $B$ is a fixed deterministic matrix. \cite{kania2022causal} consider two systems of the type \eqref{Geneq}, one where $M=0$, called an observational environment and one with $M\not =0$ called a shifted environment. This setup was also used in anchor regression \citep{Rot}.

In this paper, we consider a number of extensions of the setting described above. First, instead of a simple linear SEM, we allow the matrix $B$ to be random. From a strictly causal perspective, this may seem unusual. However, from a risk minimization perspective this extends the range of scenarios dramatically. We consider an enormous range of systems, where in some way the covariates and target are related to each other. It allows for uncertainty in the structure of the causal system, whereby the causal structure can change randomly. Secondly, we consider any finite number of shifted environments. This can be a more realistic representation of available data in many situations. Thirdly, whereas \cite{kania2022causal} and \cite{Rot} do not allow for shifts on the target, this paper explicitly allows for environments with such shifts. 

The paper starts by formally introducing the model in section~\ref{sec:model}. We assume that our target and $p$ covariates are related by an extended SEM in an observational and $k\in\N$ different shifted environments. We introduce the weighted space $C^\gamma_w$ that consists of all shifts in $L^2(\P)$ that are no more than $\gamma$ times stronger than a weighted mean of the given shifts (weighted by $w\in\R^p$). We show that the supremum over all quadratic risks, $R_{\tilde{A}\in C^\gamma_w}(\beta)$, in this space has a certain worst-risk decomposition. We prove that under a weak assumption there exists a unique $\arg\min$ solution over all $\beta$s, depending on $\gamma$. This is defined as the worst risk minimizer, $\beta_\gamma$. We show that this $\beta_\gamma$ has a closed form solution. 

In section~\ref{sec:estimation} we define the plug-in estimator of $\beta_\gamma$, $\hat{\beta}_\gamma$ under the assumptions that samples arrive through independent channels across the different environments. This is a flexible framework that allows for very asymmetric dataflow. Theorem~\ref{converge} in section~\ref{sec:consistency} shows that the plug-in estimator is consistent in the almost sure sense. Furthermore, we establish a concentration in measure result for the estimator as well. The concentration in measure result is then given a few applications. In section~\ref{sec:variance} we study conditional absolute central moments for the estimator. Due to the structure of $\hat{\beta}_\gamma$ involving a matrix inversion, the unconditional moments may not exist. By conditioning on the fact that this inverse matrix is norm-bounded by some given constant, we establish a bound for the conditional $q^{\mbox{\scriptsize th}}$ absolute central moment in Theorem~\ref{varthm1}. It shows that under sufficient moment conditions on the target and covariates, the rate of decay is $n^{-q/2}$, where $n$ denotes the minimum samples across the different environments. In section~\ref{sec:stopping} we consider the empirical data arrival process as a process in discrete time and consider filtrations with respect to this time index. We consider both filtrations based on the target and covariate samples and filtrations based only on the covariates. By considering certain stopping times for these filtrations,  Theorem~\ref{aopt} finds almost sure bounds for the conditional variance for our estimator, whereas Theorem~\ref{covar} gives an explicit expression of the conditional covariance matrix and an a.s. bound of its determinant.

We conclude in section~\ref{sec:conclusion} with a discussion where we highlight the fact that many of our results can easily be extended to much more general estimators. Several proofs have been delegated to the appendix.

\section{Risk minimization in a multi-environment setting}
\label{sec:model}
Assume we are given a probability space $\left(\Omega,\mathcal{F},\P \right)$. For $1\le i\le k$, $k\in\N$, let $Y^{A_i}\in\R$, $X^{A_i}\in\R^p$ be random variables and vectors respectively on this space that are solutions to the following structural equations, 
\begin{equation}\label{SEMA}
	\begin{bmatrix}
		Y^{A_i}\\
		X^{A_i}
	\end{bmatrix} = 
	B(\omega)\cdot\begin{bmatrix}
		
		Y^{A_i}\\
		X^{A_i}
	\end{bmatrix}
	+\epsilon_{A_i}
	+
	A_i
\end{equation}
where $B(\omega)$ is a random real-valued $(p+1)\times(p+1)$ matrix such that $I-B$ is full rank a.s., the components of $A_i\in\R^{p+1}$, $\epsilon_{A_i}\in\R^{p+1}$ are in $L^2(\P)$ for $1\le i\le k$. Note that $B$ is the same random matrix for all equation systems. We will refer to these $k$ equation systems as environments. Stochastically, the roles of $X^{.}$ and $Y^{.}$ are completely identical, but our prediction focus is on the \emph{target} $Y^.$. The random vector $A_i\in \R^{p+1}$ is called the shift corresponding to environment $i$. Since $I-B$ has full rank a.s., $X^{A_i}$ and $Y^{A_i}$ have unique solutions. We also consider the observational (shift free) environment 
\begin{equation}\label{SEMO}
	\begin{bmatrix}
		Y^{O}\\
		X^{O}
	\end{bmatrix} = 
	B(\omega)\cdot\begin{bmatrix}
		Y^{O}\\
		X^{O}
	\end{bmatrix}
	+\epsilon_{O}.
\end{equation}
We assume that $\epsilon_O$ and $B$ have the same joint law as $\epsilon_{A_i}$ and $B$ for all $1\le i\le k$. Denote $\sigma(B)$ as the sigma algebra generated by the entries of $B$ and let $\E\left[.\|B\right]$ denote conditional expectation with respect to $\sigma(B)$. We will assume that the noise terms are uncorrelated with the shifts given $B$ i.e. $\E\left[\epsilon_{A_i}A_i^T\| B\right]=0$ a.s..
\\
Let $\n.\n$ be the Euclidean norm and define the set of weights, $\mathcal{W}=\left\{w\in\R^k: \n w\n=1\right\}$ (to be clear, these weights are deterministic and moreover this choice of $\mathcal{W}$ is in some sense arbitrary, any compact set in $\R^k$ works). For a vector $w\in\mathcal{W}$ we define $A_w=\sum_{i=1}^kw_iA_i$. Given a joint distribution $F_A$ on $\R^{p+1}$ whose marginals have finite second moments, let $(\Omega_A,\mathcal{F}_A,\P_A)$ denote an extension of the original probability space that also supports a random vector $A\in\R^{p+1}$ distributed according to $F_A$, independent from $B$, and a random vector $\epsilon_A$ with the same distribution as $\epsilon_O$ and that is also independent from both $A$ and $B$. On this space we may define $Y^{A}\in\R$ and $X^{A}\in\R^p$ through,
\begin{equation}\label{A}
	\begin{bmatrix}
		Y^{A}\\
		X^{A}
	\end{bmatrix} := 
	(I-B)^{-1}(\epsilon_{A}
	+
	A),
\end{equation}
which is the solution to the structural equation system,
\begin{equation*}
	\begin{bmatrix}
		Y^{A}\\
		X^{A}
	\end{bmatrix} = 
	B(\omega)\cdot\begin{bmatrix}
		
		Y^{A}\\
		X^{A}
	\end{bmatrix}
	+\epsilon_{A}
	+
	A.
\end{equation*}
Let $\mathcal{P}$ denote some set of distributions on $\R^{p+1}$ whose marginals have finite second moments and such that $F_{A_w}\in\mathcal{P}$ for any $w\in\mathcal{W}$. We define the shiftspace of distributions,
\begin{align}\label{shiftspace}
	C^\gamma_w(B)=\left\{F_A\in \mathcal{P}:\E_A\left[AA^T\|B\right]\preccurlyeq \gamma \E_A\left[A_wA_w^T\|B\right], \P_A\textsf{- a.s. }\right\},
\end{align} 
where $\preccurlyeq$ denotes less or equal under the Loewner order. The following proposition gives some important examples of how the shift-space is affected by certain properties of $B$.
\begin{prop}
	\begin{itemize}
		\item[1)] if $B$ is any non-zero deterministic matrix then \eqref{SEMA} and \eqref{SEMO} become linear SEMs and \eqref{shiftspace} reduces to
		$$C^\gamma_w(B)=\left\{F_A\in \mathcal{P}:\E_A\left[AA^T\right]\preccurlyeq \gamma \E\left[A_wA_w^T\right]\right\}.$$ 
		\item[2)] If $B$ is a simple map taking (deterministic) values $\{B_1,...,B_m\}$ then \eqref{SEMA} and \eqref{SEMO} become piecewise linear SEMs, while \eqref{shiftspace} reduces to 
		$$C^\gamma_w(B)=\left\{F_A\in \mathcal{P}:\E_A\left[AA^T1_{B=B_l}\right]\preccurlyeq \gamma \E\left[A_wA_w^T1_{B=B_l}\right], 1\le l\le m\right\}.$$
		\item[3)] If $B$ is independent of $A_1,...,A_k$ then \eqref{shiftspace} reduces to 
		$$C^\gamma_w(B)=\left\{F_A\in \mathcal{P}:\sigma(A)\perp\sigma(B),\E_A\left[AA^T\right]\preccurlyeq \gamma \E\left[A_wA_w^T\right]\right\}.$$ 
		and the condition $\E_A\left[\epsilon A\|B\right]=0$ is equivalent to $\epsilon\E_A\left[A\right]=0$ a.s.. This means that all shifts have zero mean unless there is no noise a.s..
	\end{itemize}
\end{prop}
\begin{proof}
Since 1) is a special case of both 2) and 3) it suffices to show these last two statements. For 2), we first note that $\E_A\left[AA^T\|B\right]\preccurlyeq \gamma \E_A\left[A_wA_w^T\|B\right]$ a.s. is equivalent to $\E_A\left[AA^T1_D\right]\preccurlyeq \gamma \E\left[A_wA_w^T1_D\right]$ (the subscript of $A$ can be dropped in the last expectation since $A_w$ is measurable with respect to $\mathcal{F}$) for all $D\in\sigma(B)$. Let $C_l=\{B=B_l\}$, for $1\le l\le m$. Since $\sigma(B)=\{\emptyset,C_1,...,C_m\}$, this inequality needs only to be verified for these sets, the inequality restricted to the empty set is tautological and is therefore not included in the condition. As for 3), the first statement follows directly from the independence of the shifts from $B$. For the second statement, note that if we define $f$ by 
	$$\epsilon_A A=\sum_{l=1}^{p+1}\epsilon_A(l)A(l)=f(\epsilon_A(1),...,\epsilon_A(p+1),A(1),...,A(p+1))$$ 
	then if we let 
	$$h(a_1,...,a_{p+1})=\E_A\left[f(a_1,...,a_{p+1},A(1),...,A(p+1))\right]=\sum_{l=1}^{p+1}a_l\E_A[A(l)]$$ 
	it follows that (see for instance section 9.10 in \cite{williams1991probability})
	\begin{align*}
		\E_A\left[\epsilon_A A\|B\right] &= \E_A\left[f(\epsilon_A(1),...,\epsilon_A(p+1),A(1),...,A(p+1))\|B\right]
		\\
		&=h(\epsilon_A(1),...,\epsilon_A(p+1))
		=\sum_{l=1}^{p+1}\epsilon_A(l)\E_A[A(l)]
		=\epsilon_A\E_A[A] \hspace{2mm} a.s.
	\end{align*}
\end{proof}
\subsubsection*{A Short discussion regarding the implications of a random transfer matrix}
A few words about what the implications are of having a random transfer matrix $B$ are in order here. One obvious benefit is that since $B$ is random, this allows for extra randomness not captured by some corresponding linear model. It will allow for hidden confounding that enters multiplicatively in the same manner across all environments (including the observational). It may also help to introduce certain kinds of non-linearity. We can simultaneously fit any $k\in\N$ number of non-linear environments to a system of the type \eqref{SEMA}-\eqref{SEMO} (albeit with different shifts and noise) on the same probability space. Consider first $k\le p+1$ non-linear systems of the form
\begin{equation}\label{SSA}
	\begin{bmatrix}
		Y^{\tilde{A}_i}\\
		X^{\tilde{A}_i}
	\end{bmatrix} = 
	f\left(\begin{bmatrix}
		Y^{\tilde{A}_i}\\
		X^{\tilde{A}_i}
	\end{bmatrix} + \tilde{A}_i\right)
	+\eta_{\tilde{A}_i},
\end{equation}
for $0\le i\le p$ ($i=0$ will correspond to the observational environment where $\tilde{A}_0=0$), where $f:\mathbb{R}^{p+1}\rightarrow\mathbb{R}^{p+1}$ is a measurable function and $A_i,\eta_{A_i}\in\R^{p+1}$ are random vectors. We are now tasked with finding $\epsilon_O,\epsilon_{A_1},...,\epsilon_{A_p},A_1,...,A_p$ and $B$ such that the conditional dependence assumptions of the systems in \eqref{SEMA}-\eqref{SEMO} are met Let $\epsilon_O=\epsilon_{A_1}=...=\epsilon_{A_p}:=\epsilon=(1,0,...,0)$ and $A_i(l)=\delta_{i,l+1}$. Denote the $(p+1)\times(p+1)$ matrices
\begin{equation}\label{CD}
C=\begin{bmatrix}
		\epsilon;\epsilon+A_1;...;\epsilon+A_p
	\end{bmatrix} , 
	\\
	D(\omega)= \begin{bmatrix}
	f\left(\begin{bmatrix}
		Y^{\tilde{O}}\\
		X^{\tilde{O}}
	\end{bmatrix}\right)
	+\eta_{\tilde{A}_0};
	f\left(\begin{bmatrix}
		Y^{\tilde{A}_1}\\
		X^{\tilde{A}_1}
	\end{bmatrix} + \tilde{A}_1\right)
	+\eta_{\tilde{A}_1};...;
		f\left(\begin{bmatrix}
		Y^{\tilde{A}_p}\\
		X^{\tilde{A}_p}
	\end{bmatrix} + \tilde{A}_p\right)
	+\eta_{\tilde{A}_p}\end{bmatrix}.
\end{equation}
We can fit $B$ to the $p+1$  non-linear environments if we can solve the matrix equation $(I-B)^{-1}C =D$, which is equivalent (if $D$ is full rank) $B=I-D^{-1}C$ while still having $I-B$ being full rank (a.s.). This is possible if and only if $D$ and $C$ are full rank a.s., and $C$ was already chosen to have full rank. So with $D$ being full rank we can thus find a $B(\omega)$ for the systems \eqref{SEMA}-\eqref{SEMO} while also solving the corresponding non-linear systems in \eqref{SSA}, i.e.,
\begin{equation}\label{kpplus1}
	\begin{bmatrix}
		Y^{\tilde{O}}\\
		X^{\tilde{O}}\\
		Y^{\tilde{A}_1}\\
		X^{\tilde{A}_1}\\
		\vdots\\
		Y^{\tilde{A}_k}\\
		X^{\tilde{A}_k}
	\end{bmatrix} =
	\begin{bmatrix} 
	f\left(\begin{bmatrix}
		Y^{\tilde{O}}\\
		X^{\tilde{O}}
	\end{bmatrix} \right)\\
	f\left(\begin{bmatrix}
		Y^{\tilde{A}_1}\\
		X^{\tilde{A}_1}
	\end{bmatrix}+\tilde{A}_1 \right)\\
	\vdots\\
	f\left(\begin{bmatrix}
		Y^{\tilde{A}_k}\\
		X^{\tilde{A}_k}
	\end{bmatrix}+\tilde{A}_k \right)
	\end{bmatrix}
	+\begin{bmatrix}
	\eta_{\tilde{O}}\\
	\eta_{\tilde{A}_1}\\
	\vdots\\
	\eta_{\tilde{A}_k}
	\end{bmatrix}
	=
	\left((I-B)^{-1}
	\begin{bmatrix}
	\epsilon;
	\epsilon+A_1;
	...
	\epsilon+A_k;
	\end{bmatrix}\right)^T
	=
	\begin{bmatrix}
	Y^{O}\\
	X^{O}\\
	Y^{A_1}\\
	X^{A_1}\\
	\vdots\\
	Y^{A_p}\\
	X^{A_p}
	\end{bmatrix}
\end{equation}
Let us now deal with the case when $k>p+1$. First extend $f$ to $\tilde{f}:\R^k\to\R^k$, as $\tilde{f}=(f,0..,0)$. Similarly extend $X^{\tilde{A_i}}$ to $X'^{\tilde{A_i}}\in\R^{k-1}$, with  $X'^{\tilde{A_i}}(l)=\eta_{\tilde{A_i}}(l)$ for $l> p+1$ and for $l\le  p+1$, $X'^{\tilde{A_i}}(l)=X^{\tilde{A_i}}(l)$. We extend the shifts, $\tilde{A_i}$ to $\tilde{A_i}'\in\R^{k}$, with $\tilde{A_i}'(l)=\tilde{A_i}(l)$ for $l\le p+1$ and $\tilde{A_i}'(l)=0$ for $p+1<l\le k$. Finally we extend the noise, to $\eta_{A_i}'\in\R^k$ by $\eta_{A_i}'(l)=\eta_{A_i}(l)$ for $1\le l\le p+1$ and $\eta_{A_i}'(l)=Z_{i,l}$ for $p+1<l\le k$, where $\{Z_{i,l}\}_{i,l}$ is some set of absolutely continuous random variables that are independent of the rest of the system and amongst each other. In our new construction we have $p'+1=k$ (where $p'$ is the number of covariates in the extended system) so we can now apply to former construction in \eqref{CD} to get the corresponding solution in \eqref{kpplus1}. The first $p+1$ rows in
\begin{equation*}\label{SSAprime}
	\begin{bmatrix}
		Y^{\tilde{A}'_i}\\
		X'^{\tilde{A}'_i}
	\end{bmatrix} = 
	\tilde{f}\left(\begin{bmatrix}
		Y^{\tilde{A}'_i}\\
		X'^{\tilde{A}'_i}
	\end{bmatrix} + \tilde{A}'_i\right)
	+\eta'_{\tilde{A}'_i}
	=
	\begin{bmatrix}
		Y^{A_i}\\
		X^{A_i}
	\end{bmatrix},
\end{equation*}
are exactly the same as those in \eqref{SSA}, which means that our desired system of the form \eqref{SEMA}-\eqref{SEMO} is given by
\begin{equation*}
\begin{bmatrix}
		Y^{A_i}\\
		X^{A_i}(1:p)
	\end{bmatrix},
\end{equation*}
where
\begin{equation*}
\begin{bmatrix}
		Y^{A_i}\\
		X^{A_i}
	\end{bmatrix}=(I-B')^{-1}C',
\end{equation*}
$B'=I-D'^{-1}C$, $C$ is defined as before and 
\begin{equation*}
	D'(\omega)= \begin{bmatrix}
	\tilde{f}\left(\begin{bmatrix}
		Y'^{\tilde{O}}\\
		X'^{\tilde{O}}
	\end{bmatrix}\right)
	+\eta'_{\tilde{A}_0};
	\tilde{f}\left(\begin{bmatrix}
		Y'^{\tilde{A}_1}\\
		X'^{\tilde{A}_1}
	\end{bmatrix} + \tilde{A}'_1\right)
	+\eta_{\tilde{A}_1};...;
		\tilde{f}\left(\begin{bmatrix}
		Y'^{\tilde{A}_p}\\
		X'^{\tilde{A}_p}
	\end{bmatrix} + \tilde{A}'_p\right)
	+\eta'_{\tilde{A}_p}\end{bmatrix}.
\end{equation*}
\subsection{The worst-risk decomposition}
For any $F_A\in\mathcal{P}$, denote $R_{A}(\beta)=\E_A\left[\left(Y^A-\beta X^{A}\right)^2\right]$ and $R_{A_i}(\beta)=\E\left[\left(Y^{A_i}-\beta X^{A_i}\right)^2\right]$ for $0\le i\le k$ (with $R_O=R_{A_0}$). We also define
\begin{itemize}
\item[] $R^w_+(\beta)=\sum_{i=1}^kw_i^2R_{A_i}+R_O(\beta)$ and
\item[] $R^w_\Delta(\beta)=\sum_{i=1}^kw_i^2R_{A_i}-R_O(\beta)$
\end{itemize}
\begin{prop}\label{suppropfixedw}
Let $\tau\ge -\frac 12$ and $w\in\mathcal{W}$ then 
$$\sup_{F_A\in C^{1+\tau}_w(B)}R_{A}(\beta)=\frac{1}{2}R^{w}_+(\beta)+\frac{1+2\tau}{2}R^{w}_\Delta(\beta).$$
\end{prop}
\begin{proof}
Recall that all the noise has the same distribution across all environments, we shall denote $\epsilon$ as generic random vector with such a distribution. Due to \eqref{A}, $Y^{A}=((I-B)^{-1})_{1,\textbf{.}}(A+\epsilon_A)$ and $X^{A}=((I-B)^{-1})_{2:p+1,\textbf{.}}(A+\epsilon_A)$. Since the entries of $(I-B)^{-1}$ are $\sigma(B)$-measurable it follows that if we define $v=\beta (I-B)^{-1}_{2:p+1,\textbf{.}}-(I-B)^{-1}_{1,\textbf{.}}$ (implying $v(\epsilon_A+A)=Y^A-\beta X^A$) then $v$ is also $\sigma(B)$-measurable. With this notation
$$\sup_{F_A\in C^{1+\tau}_w(B)}R_{A}(\beta) =\sup_{F_A\in C^{1+\tau}_w(B)}\E_A\left[ v(A+\epsilon_{A})(A+\epsilon_{A})^T v^T\right].$$
Since $R_O(\beta)<\infty$ for any $\beta\in\R^p$ it follows that $\E\left[v\epsilon\epsilon^Tv^T\right]<\infty$. Similarly, since $R_{A_i}(\beta)<\infty$ it follows that $\E\left[ v(A_i+\epsilon)(A_i+\epsilon)^T v^T\right]<\infty$. We also have $\E\left[ vA_i\epsilon^T v^T\right]=\E\left[ v\E\left[A_i\epsilon^T\|B\right] v^T\right]=0$ and hence $\E\left[vA_iA_i^Tv^T\right]<\infty$ for all $1\le i\le k$. This leads to
\begin{align}\label{supeq1}
\sup_{F_A\in C^{1+\tau}_w(B)}R_{A}(\beta) &=\E\left[ v\epsilon\epsilon^T v^T\right]+\sup_{F_A\in C^{1+\tau}_w(B)}\left(2\E_A\left[ v\epsilon A^T v^T\right]+\E_A\left[ vAA^T v^T\right]\right)
\nonumber
\\
&=\E\left[ v\epsilon\epsilon^T v^T\right]+\sup_{F_A\in C^{1+\tau}_w(B)}\E_A\left[ vAA^T v^T\right].
\end{align}
 If $F_A\in C_w^{1+\tau}(B)$ then
\begin{align*}
\E_A\left[vAA^Tv^T\right]
&=
\E_A\left[\E_A\left[vAA^Tv^T\|B\right]\right]
=
\E_A\left[v\E_A\left[AA^T\|B\right]v^T\right]
\le
(1+\tau)\E_A\left[v\E_A\left[A_wA_w^T\|B\right]v^T\right]
\\
&=(1+\tau)\E\left[\E\left[vA_wA_w^Tv^T\|B\right]\right]
=\E\left[(1+\tau)vA_wA_w^Tv^T\right],
\end{align*}
i.e. $\sup_{F_A\in C^{1+\tau}_w(B)} \E_A\left[vAA^Tv^T\right]\le \E\left[(1+\tau)vA_wA_w^Tv^T\right]$. Since $F_{\sqrt{(1+\tau)}A_w}\in C_w^{1+\tau}(B)$ it follows that 
$$\sup_{F_A\in C^{1+\tau}_w(B)} \E_A\left[vAA^Tv^T\right]=\E\left[(1+\tau)vA_wA_w^Tv^T\right].$$
Going back to \eqref{supeq1} we find
\begin{align}\label{supeq}
\sup_{F_A\in C^{1+\tau}_w(B)}R_{A}(\beta)
&=
\E\left[v\epsilon\epsilon^Tv^T\right]+(1+\tau)\E\left[vA_w A_w^Tv^T\right]
\nonumber
\\
&=R_O(\beta)
+
(1+\tau)\E\left[v(A_w+\epsilon) (A_w+\epsilon)^Tv^T\right]
-
(1+\tau)\E\left[v\epsilon \epsilon^Tv^T\right]
\nonumber
\\
&=R_{O}(\beta)+(1+\tau)R_{A_w}(\beta)-(1+\tau)R_{O}(\beta)
\nonumber
\\
&=
(1+\tau)R_{A_w}(\beta)-\tau R_{O}(\beta)
\end{align}
Since
\begin{align*}
&R_{A_{w}}(\beta)=\E\left[\left(X^{A_{w}}\beta-Y^{A_{w}}\right)^2\right]=\E\left[v\left(\epsilon+\sum_{i=1}^kw_iA_i\right)\left(\epsilon+\sum_{i=1}^kw_iA_i\right)^Tv^T\right]
=\sum_{i=1}^k(w_i)^2R_{A_i}(\beta),
\end{align*}
we may plug this back into \eqref{supeq} and get
$$ \sup_{F_A\in C^{1+\tau}_w(B)}R_{A}(\beta)=(1+\tau)\sum_{i=1}^k(w_i)^2R_{A_i}(\beta)-\tau R_{O}(\beta)=\frac{1}{2}R_+(\beta)+\frac{1+2\tau}{2}R_\Delta(\beta).$$
\end{proof}

\begin{defin}[Worst risk minimization]
With the above proposition in mind we now define the worst risk minimizer with parameter $\gamma\in [0,\infty]$, as the solution to
$$\beta_\gamma=\arg\min R_+^w(\beta)+\gamma R^w_\Delta(\beta).$$
\end{defin}
We also define the associated minimal risk.
\begin{defin}[Minimal risk]
$$R=\inf_{\beta\in\R^p}\sup_{\tilde{A}\in C^{1+\tau}_w(B)}R_{\tilde{A}}(\beta). $$
\end{defin}
Given the quadratic nature of the risk, it is possible to determine the worst risk minimizer explicitly. Let 
\begin{itemize}
\item[] $G_+=\E\left[(X^O)^TX^O+\sum_{i=1}^kw_i^2(X^{A_i})^TX^{A_i}\right]$,
\item[]$G_\Delta=\E\left[\sum_{i=1}^kw_i^2(X^{A_i})^TX^{A_i}-(X^O)^TX^O\right]$,
\item[] $Z_+=\E\left[(X^O)^TY^O+\sum_{i=1}^kw_i^2(X^{A_i})^TY^{A_i}\right]$ and $Z_\Delta=\E\left[\sum_{i=1}^kw_i^2(X^{A_i})^TY^{A_i}-(X^O)^TY^O\right]$
\end{itemize}

\begin{prop}\label{sol}
If $G_++\gamma G_\Delta$ is of full rank then
$$ \beta_\gamma=\left(G_++\gamma G_\Delta\right)^{-1}\left(Z_++\gamma Z_\Delta\right) $$
and
$$R=R_+^w(\beta_\gamma)+\gamma R^w_\Delta(\beta_\gamma) $$
\end{prop}
\begin{proof}
Let $f(\beta)=R_+(\beta)+\gamma R_\Delta(\beta)$ for $\beta\in\R^p$. To find an explicit solution of $\beta_\gamma$ we solve the equation $\nabla_\beta  f(\beta)=0$. To that end we compute,
\begin{align*}
&\nabla_\beta R_{A_i}(\beta)=2\E\left[(X^{A_i})^T\left(X^{A_i}\beta-Y^{A_i} \right)\right].
\end{align*}
Plugging this into $\nabla_\beta  f(\beta)=0$ leads to,
\begin{align*}
&\E\left[2(X^O)^TX^O+\sum_{i=1}^k 2 w_i^2 (X^{A_i})^TX^{A_i}\right]\beta+\gamma\E\left[\sum_{i=1}^k2 w_i^2 (X^{A_i})^TX^{A_i}-2(X^O)^TX^O\right]\beta
-
\\
&\E\left[2(X^O)^TY^O+\sum_{i=1}^k2 w_i^2 (X^{A_i})^TY^{A_i}\right]-\gamma\E\left[\sum_{i=1}^k2 w_i^2 (X^{A_i})^TY^{A_i}-2(X^O)^TY^O\right]=0,
\end{align*}
which leads to
$$ \beta_\gamma=\left(G_++\gamma G_\Delta\right)^{-1}\left(Z_++\gamma Z_\Delta\right). $$
The final claim follows directly from Proposition \ref{suppropfixedw}.
\end{proof}

\section{Estimation of the minimizer and minimal risk}
\label{sec:estimation}
We now turn to the problem of estimating $\beta_\gamma$ in an empirical setting with data. We assume all of our empirical data lives on some fixed probability space $(\Omega,\F,\P)$, although we denote it (for the sake of exposition) the same as the space for population case, these are not assumed to be the same probability spaces. We now set up a framework for how we handle samples from multiple different environments. We assume that for each environment $i$ we have an i.i.d. sequence $\{(Y_u^{A_i},X_u^{A_i})\}_{u=1}^\infty$ (strictly speaking, we only need infinite sequences to establish consistency, all of our bounds are with regards to a given finite number of samples) where $(Y_u^{A_i},X_u^{A_i})$ is distributed according to the given SEM for environment $i$ and we assume that the sequences $\{(Y_u^{A_i},X_u^{A_i})\}_{u=1}^\infty$ are mutually independent across the different $i$'s.
Suppose $\textbf{n}=\left\{n_{A_0},...,n_{A_k}\right\}$, let $\mathbb{X}^{A_i}(\textbf{n})$ be the $n_{A_i}\times p$ matrix whose rows are $X^{A_i}_1,...,X^{A_i}_{n_{A_i}}$ (from top to bottom) and similarly let $\mathbb{Y}^{A_i}(\textbf{n})$ be the $n_{A_i}\times 1$ column vectors whose entries are $Y^{A_i}_1,...,Y^{A_i}_{n_{A_i}}$ (from top to bottom).
Denote 
$$\hat{G}_+(\textbf{n})=\sum_{i=0}^k\frac{w_i^2}{n_{A_i}}(\mathbb{X}^{A_i})^T\mathbb{X}^{A_i}+\frac{1}{n_O} (\mathbb{X}^{O})^T\mathbb{X}^{O},$$ 
$$\hat{Z}_+(\textbf{n})=\sum_{i=0}^k\frac{w_i^2}{n_{A_i}}(\mathbb{X}^{A_i})^T\mathbb{Y}^{A_i}+\frac{1}{n_O} (\mathbb{X}^{O})^T\mathbb{Y}^{O},$$
$$\hat{R}_{A_i}(\beta)=\frac{1}{n_{A_i}}\sum_{u=1}^{n_{A_i}}\left(Y^{A_i}_u -\langle\beta,X^{A_i}_u\rangle\right)^2 \textsf{, for } \beta\in\R^p ,$$
$$\hat{R}_{+}(\beta)=\sum_{i=0}^{k}\hat{R}_{A_i}(\beta) \textsf{, for } \beta\in\R^p ,$$
$$\hat{R}_{+}(\beta)=\sum_{i=1}^{k}\hat{R}_{A_i}(\beta)- \hat{R}_{O}(\beta)\textsf{, for } \beta\in\R^p $$
$$\hat{G}_\Delta(\textbf{n})=\sum_{i=0}^k\frac{w_i^2}{n_{A_i}}(\mathbb{X}^{A_i})^T\mathbb{X}^{A_i}-\frac{1}{n_O} (\mathbb{X}^{O})^T\mathbb{X}^{O}$$ and 
$$\hat{Z}_\Delta(\textbf{n})=\sum_{i=0}^k\frac{w_i^2}{n_{A_i}}(\mathbb{X}^{A_i})^T\mathbb{Y}^{A_i}-\frac{1}{n_O} (\mathbb{X}^{O})^T\mathbb{Y}^{O}.$$ With this notation we denote the plug-in estimator $\hat{\beta}_\gamma(\textbf{n})=\left(\hat{G}_++\gamma\hat{G}_\Delta \right)^{-1}\left(\hat{Z}_+-\gamma\hat{Z}_w \right)$ of the minimizer and denote $ \hat{R}(\textbf{n})=\hat{R}_+^w(\beta_\gamma(\textbf{n}))+\gamma \hat{R}^w_\Delta(\beta_\gamma(\textbf{n}))$, the plug-in estimator of the minimal risk. 

\subsection{Consistency and concentration in measure}
\label{sec:consistency}
From here on, given any vector norm $\n.\n$ and any matrix we will let $\n A\n$ denote the operator norm induced by the vector norm $\n.\n$. The next Theorem shows that the plug-in estimator is consistent in the almost sure sense and provides a concentration in measure result. Showing the consistency is straight-forward and the proof is included in the main body of the text. The second and third results regarding concentration in measure are more technical and the proof is deferred to the supplementary material. Our first bound is valid with a large enough number of samples from each environment assuming only that the covariates and target have finite second moments. With the marginally stronger assumption that the covariates and the target are in weak $L^\zeta$ for some $\zeta>2$ we get a bound valid for all $n_{A_i}\ge 1$. 
\begin{thm}\label{converge}
Suppose $G_\Delta$ is non-singular then 
\begin{itemize}
\item $\tilde{\beta}_\gamma(\textbf{n})\xrightarrow{a.s.}\beta_\gamma,$
as $n_{A_0}\wedge...\wedge n_{A_k}\to \infty$ where $\tilde{\beta}_\gamma(n)=\left(\hat{G}_++\gamma\hat{G}_\Delta \right)^{-1}_g\left(\hat{Z}_++\gamma\hat{Z}_\Delta \right)$ and $()^{-1}_g$ denotes any map that coincides with matrix inversion on the space of full rank $p\times p$ matrices.
\item For any $0<\delta<1$ and $0<\alpha< \frac 14$, there exists $N_{A_0},...,N_{A_k}$ such that if $n_{A_i}\ge N_{A_i}$ then
\begin{align}\label{Pref}
&\P\left( \n \tilde{\beta}_\gamma(\textbf{n}) -\beta_\gamma\n_{\mathit{l}^2}\ge c \right)\nonumber
\\
&\le
(4p+2)\sum_{i=0}^ke^{-\left(\delta\wedge c\right)^2 E^2 n_{A_i}^{1-4\alpha}}
+5\left(1-\prod_{i=0}^k\left(F_{\left|X^{A_i}(1)\right|,...,\left|X^{A_i}(p)\right|,|Y^{A_i}|}(n_{A_i}^\alpha,...,n_{A_i}^\alpha)\right)^{n_{A_i}}\right),
\end{align}
where
\tiny
\begin{align*}
&E=\frac{1}{2\sqrt{2}p(k+1)}
\left(\n \left(G_++\gamma G_\Delta\right)^{-1}\n\wedge\frac{1}{6\n \left(G_++\gamma G_\Delta\right)^{-1}\n(1+\delta)\left(1\vee\left(
\n \left(G_++\gamma G_\Delta\right)^{-1} \n\n\left( Z_++\gamma Z_\Delta \right)\n\right)\right)}\right).
\end{align*}
\normalsize
\item Suppose $ X^{A_i}(l), Y^{A_i}\in L^{\zeta,w}(\P)$, where $L^{\zeta,w}(\P)$ denotes the weak $L(\P)^{\zeta}$-space, for some $\zeta>2$. Then for any $0<\delta<1$, $0<\alpha< \frac 14$ and all $n_{A_i}\ge 1$,
\begin{align*}
\P\left(\n \tilde{\beta}_\gamma(\textbf{n}) -  \beta_\gamma\n_{\mathit{l}^2}\ge c \right)
&\le 
V\sum_{i=0}^ke^{-\frac{(c\wedge \delta)^2\left(\n \left(G_++\gamma G_\Delta\right)^{-1}\n_{\mathit{l}^2}\wedge \n G_++\gamma G_\Delta\n_{\mathit{l}^2}\right)^6}{6p^2(k+1)^2(1+\delta)\left(1\vee
\n\left( Z_++\gamma Z_\Delta \right)\n_{\mathit{l}^2}\right)^2} n_{A_i}^{1-4\alpha}}
\nonumber
\\
&+5\left(1-\prod_{i=0}^k\left(F_{\left|X^{A_i}(1)\right|,...,\left|X^{A_i}(p)\right|,|Y^{A_i}|}(n_{A_i}^\alpha,...,n_{A_i}^\alpha)\right)^{n_{A_i}}\right).
\end{align*}
where 
\scriptsize
\begin{align*}
&V=10p\exp\left( \frac{(c\vee \delta)^2\left(\n \left(G_++\gamma G_\Delta\right)^{-1}\n_{\mathit{l}^2}\vee \n G_++\gamma G_\Delta\n_{\mathit{l}^2}\right)^6}{6p^2(k+1)^2\left(1\wedge\n\left( Z_++\gamma Z_\Delta \right)\n_{\mathit{l}^2}\right)^2}\right.
\\
&\left.\times
\left(1+2M_w'(\zeta)^{\frac{2+\zeta}{\alpha\zeta(\zeta'-2)}}(p+1)^{\frac{2\zeta}{\alpha\zeta(\zeta-2)}}(1\vee c^{-1}\vee\delta^{-1})6(1+\gamma)\left(1\vee \n \left(G_++\gamma G_\Delta\right)^{-1}\n_{\mathit{l}^2}\right)^3\left(1\vee \n Z_++\gamma Z_\Delta \n_{\mathit{l}^2}\right)\right)\right)
\end{align*}
\normalsize
and
$M_w'(\zeta)=\max_{i,l}\left(\n X^{A_i}(l)\n^\zeta_{L^{\zeta,w}(\P)}\bigvee \n Y^{A_i}\n^\zeta_{L^{\zeta,w}(\P)}\right)$
\end{itemize}
\end{thm}
\begin{proof}[Proof of the first claim]
We begin with the first claim. Since $G_\Delta$ is positive semi-definite and full rank it must in fact be positive definite and since $G_\Delta\preccurlyeq G_+$, $G_+$ must also be positive definite. Since $\gamma\ge 0$ it follows that $G_++\gamma G_\Delta$ is positive definite and hence invertible. By the law of large numbers, applied component-wise we have that $\frac{1}{n_{A_i}}(\mathbb{X}^{A_i}(\textbf{n}))^T\mathbb{X}^{A_i}(\textbf{n})\xrightarrow{a.s.}\E\left[ (X^{A_i})^TX^{A_i} \right]$ and $\frac{1}{n_{A_i}}(\mathbb{X}^{A_i}(\textbf{n}))^T\mathbb{Y}^{A_i}(\textbf{n})\xrightarrow{a.s.}\E\left[ (X^{A_i})^TY^{A_i} \right]$ as $n_{A_i}\to\infty$. By linearity,  $\hat{G}_++\gamma\hat{G}_\Delta\xrightarrow{a.s.} G_++\gamma G_\Delta$ as well as $\hat{Z}_+-\gamma\hat{Z}_w \xrightarrow{a.s.}Z_++\gamma Z_\Delta$ as $n_{A_0}\wedge...\wedge n_{A_k}\to\infty$. The set of non-singular $p\times p$ matrices form an open subset of $\R^{p^2}$ and the operation of matrix inversion on this set is continuous. For any $\textbf{n}$ such that $\n  G_++\gamma G_\Delta-\left(\hat{G}_+(\textbf{n})+\gamma\hat{G}_\Delta(\textbf{n})\right) \n< \n \left( G_++\gamma G_\Delta \right)^{-1}\n$, $\left(\hat{G}_+(\textbf{n})+\gamma\hat{G}_\Delta(\textbf{n})\right)$ will have an inverse, by Lemma 1 in the Appendix, which will then coincide with the Penrose inverse, i.e. $\hat{\beta}_\gamma(\textbf{n})=\tilde{\beta}_\gamma(\textbf{n})$ and moreover the inversion is continuous in the open ball centred in $ G_++\gamma G_\Delta$ with radius $\n \left( G_++\gamma G_\Delta\right)^{-1}\n$. Therefore it follows from continuous mapping that $\left(\hat{G}_++\gamma\hat{G}_\Delta \right)^{-1}\xrightarrow{a.s.} \left( G_++\gamma G_\Delta \right)^{-1}$ and the first claim then follows.
\end{proof}
We also have corresponding results regarding the minimal risk. Again, we include the proof of consistency here and leave the proof of the concentration in measure to the supplementary material.
\begin{thm}\label{Riskthm}
Suppose $G_\Delta$ is non-singular then 
\begin{itemize}
\item $\hat{R}(\textbf{n})\xrightarrow{a.s.}R,$
as $n_{A_0}\wedge...\wedge n_{A_k}\to \infty$.
\item For any $0<\delta<1$ and $0<\alpha< \frac 14$, there exists $N_{A_0}'...,N_{A_k}'$ such that if $n_{A_i}\ge N_{A_i}'$ then
\begin{align*}
&\P\left( \left|\hat{R}(\textbf{n}) -R\right|\ge c \right)
\le
(4k+9)\left(1-\prod_{i=0}^k\left(F_{\left|X^{A_i}(1)\right|,...,\left|X^{A_i}(p)\right|,|Y^{A_i}|}(n_{A_i}^\alpha,...,n_{A_i}^\alpha)\right)^{n_{A_i}}\right)
\\
&+2p(k+1)(3+4p)e^{-(\delta\wedge c)\left(E\wedge \frac{c}{\n\beta_\gamma\n_{\mathit{l}^2} 48(\gamma+1)(k+1)p}\right)^2},
\end{align*}
where $E$ is as in Theorem \ref{converge} and we use the convention that if $\n\beta_\gamma\n_{\mathit{l}^2}=0$ then $\frac{1}{\n\beta_\gamma\n_{\mathit{l}^2}}=+\infty$.
\end{itemize}
\end{thm}
\begin{proof}[Proof of the first claim]
On one hand we have that in the population case
$$R_{A_i}(\beta_\gamma)=\E\left[(Y^{A_i})^2\right]+2\sum_{l=1}^p\beta_\gamma(l)\E\left[Y^{A_i}X^{A_i}(l)\right]+\sum_{l=1}^p\beta_\gamma(l)^2\E\left[(X^{A_i}(l))^2\right]$$
while in the empirical setting
$$\hat{R}_{A_i}(\hat{\beta}_\gamma)=\frac{1}{n_{A_i}}\sum_{u=1}^{n_{A_i}}(Y^{A_i}_u)^2+2\sum_{l=1}^p\hat{\beta}_\gamma(l)\frac{1}{n_{A_i}}\sum_{u=1}^{n_{A_i}}Y^{A_i}_uX^{A_i}_u(l)+\sum_{l=1}^p\hat{\beta}_\gamma(l)^2\frac{1}{n_{A_i}}\sum_{u=1}^{n_{A_i}}(X^{A_i}_u(l))^2.$$
By the law of large numbers $\frac{1}{n_{A_i}}\sum_{u=1}^{n_{A_i}}(Y^{A_i}_u)^2\xrightarrow{a.s.}\E\left[(Y^{A_i})^2\right]$, \\$\frac{1}{n_{A_i}}\sum_{u=1}^{n_{A_i}}Y^{A_i}_uX^{A_i}_u(l)\xrightarrow{a.s.}\E\left[Y^{A_i}X^{A_i}(l)\right]$ and $\frac{1}{n_{A_i}}\sum_{u=1}^{n_{A_i}}(X^{A_i}_u(l))^2\xrightarrow{a.s.}\E\left[(X^{A_i}(l))^2\right]$,  for $1\le l\le p$ as $n_{A_i}\to\infty$. This combined with the fact that $\hat{\beta}_\gamma\xrightarrow{a.s.}\beta_\gamma$ shows that $\hat{R}_{A_i}(\hat{\beta}_\gamma)\xrightarrow{a.s.}R_{A_i}(\beta_\gamma)$
\end{proof}
The following two Corollaries illustrates a straight-forward application of the above concentration in measure result for the minimizer. 
\begin{corollary}\label{cor}
Under the hypothesis of Theorem \ref{converge}, if $1-F_{\left|X^{A_i}(1)\right|,...,\left|X^{A_i}(p)\right|,|Y^{A_i}|}(x,...,x) \le e^{-d_ix^{\psi_{i}}}$, for some $d_i,\psi_i>0$, whenever $x$ is large enough, then 
$$ \P\left( \n \tilde{\beta}_\gamma(\textbf{n}) -\beta_\gamma\n\ge c \right) \sim  O\left( 
\sum_{i=0}^k n_{A_i}e^{-d_{i}n_{A_i}^{\gamma_{i}\alpha}}
\vee
\sum_{i=0}^ke^{-r(c)^2n_{A_i}^{1-4\alpha}}
\right), $$
where $r(c)$ is as in Theorem \ref{converge} and $\n.\n$ is any vector norm on $\R^p$.
\end{corollary}
\begin{proof}
By writing
\small
\begin{align*}
&\prod_{i=0}^k\left(F_{\left|X^{A_i}(1)\right|,...,\left|X^{A_i}(p)\right|,|Y^{A_i}|}(n_{A_i}^\alpha,...,n_{A_i}^\alpha)\right)^{n_{A_i}}
\\
=
&\exp\left(\sum_{i=0}^k n_{A_i}\log\left(1- F_{\left|X^{A_i}(1)\right|,...,\left|X^{A_i}(p)\right|,|Y^{A_i}|}(n_{A_i}^\alpha,...,n_{A_i}^\alpha)\right) \right)
\\
=&\exp\left(-\sum_{i=0}^k  n_{A_i}\left(F_{\left|X^{A_i}(1)\right|,...,\left|X^{A_i}(p)\right|,|Y^{A_i}|}(n_{A_i}^\alpha,...,n_{A_i}^\alpha)
+
O\left( F_{\left|X^{A_i}(1)\right|,...,\left|X^{A_i}(p)\right|,|Y^{A_i}|}(n_{A_i}^\alpha,...,n_{A_i}^\alpha)^2\right) \right)\right)
\\
=&1-\sum_{i=0}^k n_{A_i}F_{\left|X^{A_i}(1)\right|,...,\left|X^{A_i}(p)\right|,|Y^{A_i}|}(n_{A_i}^\alpha,...,n_{A_i}^\alpha)
+
O\left( n_{A_i}^2F_{\left|X^{A_i}(1)\right|,...,\left|X^{A_i}(p)\right|,|Y^{A_i}|}(n_{A_i}^\alpha,...,n_{A_i}^\alpha)^2\right),
\end{align*}
\normalsize
we see that
\small
\begin{align*}
&1-\prod_{i=0}^k\left(F_{\left|X^{A_i}(1)\right|,...,\left|X^{A_i}(p)\right|,|Y^{A_i}|}(n_{A_i}^\alpha,...,n_{A_i}^\alpha)(n_{A_i}^\alpha)\right)^{n_{A_i}}
\\
&=\sum_{i=0}^k \left(n_{A_i}F_{\left|X^{A_i}(1)\right|,...,\left|X^{A_i}(p)\right|,|Y^{A_i}|}(n_{A_i}^\alpha,...,n_{A_i}^\alpha)
+
O\left( n_{A_i}^2F_{\left|X^{A_i}(1)\right|,...,\left|X^{A_i}(p)\right|,|Y^{A_i}|}(n_{A_i}^\alpha,...,n_{A_i}^\alpha)^2\right)\right) 
\\
&\le\sum_{i=0}^k n_{A_i}e^{-d_{i}n_{A_i}^{\gamma_{i}\alpha}}+O\left(\sum_{i=0}^k n_{A_i}^2e^{-2d_{i}n_{A_i}^{\psi_{i}\alpha}}\right).
\end{align*}
\normalsize
By plugging this into the second term on the right-hand side of \eqref{Pref} (and the equivalence of vector norms on $\R^p$) it now follows that
$$ \P\left( \n \tilde{\beta}_\gamma(\textbf{n}) -\beta_\gamma\n\ge c \right) 
\sim 
O\left( \sum_{i=0}^k n_{A_i}e^{-d_{i}n_{A_i}^{\psi_{i}\alpha}}
\vee
\sum_{i=0}^ke^{-r(c)n_{A_i}^{1-4\alpha}}
\right)  $$
\end{proof}
The next corollary gives convergence rates based on the number of finite moments
\begin{corollary}
Suppose $M(\eta)<\infty$ for some $\eta> 4$. Then for any $0<\delta<1$ and $\frac{1}{\eta}<\alpha<\frac 14 $, there exists $N_{A_0},...,N_{A_k}$ such that if $n_{A_i}\ge N_{A_i}$ then
\begin{align*}
&\P\left( \n \tilde{\beta}_\gamma(\textbf{n}) -\beta_\gamma\n_{\mathit{l}^2}\ge c \right)\le
(4p+2)\sum_{i=0}^ke^{-\left(\delta\wedge c\right)^2 E^2 n_{A_i}^{1-4\alpha}}
+5\frac{(p+1)(k+1) M_{X,Y}(\eta)}{\left(n_{A_0}\wedge...\wedge n_{A_k}\right)^{\alpha\eta-1}},
\end{align*}
with $E$ as in Theorem \ref{converge} and $M_{X,Y}(\eta)=\max_{i,l}\left(\E\left[\left|X^{A_i}(l)\right|^\eta\right]\right)\vee\max_{i}\E\left[\left|Y^{A_i}\right|^\eta\right]$.
\end{corollary}
\begin{proof}
By the Markov inequality we have for $x\ge 0$
\begin{align*}
1-F_{\left|X^{A_i}(1)\right|,...,\left|X^{A_i}(p)\right|,|Y^{A_i}|}(x,...,x)
&=\P\left(\left|X^{A_i}(1)\right|^\eta\vee,...,\vee\left|X^{A_i}(p)\right|^\eta\vee|Y^{A_i}|^\eta> x^\eta\right)
\nonumber
\\
&\le  \P\left(\sum_{l=1}^p\left|X^{A_i}(l)\right|^\eta+\left|Y^{A_i}\right|^\eta\ge x^\eta\right)
\le \frac{(p+1)M_{X,Y}(\eta)}{x^\eta}.
\end{align*}
Combining this with Theorem \ref{converge} and an inequality analogous to (1.45) in the supplementary material gives the result.
\end{proof}
When the target and covariates have bounded support, the  the tail term \\$1-F_{\left|X^{A_i}(1)\right|,...,\left|X^{A_i}(p)\right|,|Y^{A_i}|}(x,...,x)$ will vanish for large $x$ which implies that
$$\P\left( \n \tilde{\beta}_\gamma(\textbf{n}) -\beta_\gamma\n\ge c \right) \sim  O\left( 
\sum_{i=0}^ke^{-\left(\delta\wedge c\right)^2 R^2n_{A_i}^{1-4\alpha}}
\right). $$
Next, consider the case when $A_i(1),...,A_i(p),\mathbb{\epsilon}_0,...,\mathbb{\epsilon}_p,\mathbb{\epsilon}_Y$ are jointly normal for each $1\le i\le k$. Since 
$$\textbf{N}=\left(X^{A_0}(1),...,X^{A_0}(p),X^{A_i}(1),...,X^{A_i}(p),Y^{A_i}\right)$$
is a linear transform of  $\left(A_i(1),...,A_i(p),\mathbb{\epsilon}_0,...,\mathbb{\epsilon}_p,\mathbb{\epsilon}_Y\right)$, $\textbf{N}$ must also be normal. Let $\mu$ and $\Sigma$ denote the mean vector and covariance matrix of $\textbf{N}$. Note that 
\begin{align*}
F_{\left|X^{A_i}(1)\right|,...,\left|X^{A_i}(p)\right|,|Y^{A_i}|}(x,...,x)
&\ge \int_{\n (\textbf{x},y) \n_2\le x}\frac{\exp\left(-\frac{1}{2}\left((\textbf{x},y)-\mu\right)^T\Sigma^{-1}\left((\textbf{x},y)-\mu\right)\right)}{(2\pi)^{\frac{p+1}{2}}\sqrt{\det\left(\Sigma\right)}} d\textbf{x}dy 
\\
&\ge 1-\exp(-cx^2)
\end{align*}
for some $c>0$. From Corollary \ref{cor},
$$ \P\left( \n \tilde{\beta}_\gamma(\textbf{n}) -\beta_\gamma\n\ge c \right) \sim  O\left( 
\sum_{i=0}^k n_{A_i}e^{-cn_{A_i}^{2\alpha}}
\vee
\sum_{i=0}^ke^{-r(c)^2n_{A_i}^{1-4\alpha}}
\right), $$
\subsection{Conditional q-variance, conditioning on a bounded inverse}
\label{sec:variance}
In this section we shall study bounds for moments of the expected deviation of $\hat{\beta}(\textbf{n})$ from its expected value under certain conditioning. There is one big problem with just studying these moments as-is, i.e., without conditioning: they simply do not exist in general due to the fact that $\left(G_+(\textbf{n})+\gamma G_\Delta(\textbf{n})\right)^{-1}$ is not  integrable. What \textit{is} possible is to somewhat circumvent this issue (or to do the the ``next best thing'') by instead looking at these moments given that $\left(G_+(\textbf{n})+\gamma G_\Delta(\textbf{n})\right)^{-1}$ is bounded by some given constant. This is  what we do in the following results. For this we define the conditional $q$-variances,
\begin{defin}[Conditional q-variance]
Let $X$ be a $p$-dimensional random vector such that $X(l)1_A\in L^q(\P)$ for $2\le q<\infty$ for some event $A$ with $\P\left(A\right)>0$, then we define
$$Var_q\left( X\| A\right) = \frac{\E\left[ \n X1_A -\E\left[X1_A\right] \n_{\mathit{l}^2}^q \right]}{\P(A)}.$$
\end{defin}
Our first result shows that if $M(\zeta)<\infty$ for large enough $\zeta$ then conditioning on a bounded inverse, the q-variance for $\hat{\beta}_\gamma(\textbf{n})$ is of order $\left(n_{A_0}\wedge...\wedge n_{A_k}\right)^{q/2}$.
\begin{thm}\label{varthm1}
Define $M(\zeta)=\max_{i,l}\E\left[ |X^{A_i}(l)|^{\zeta} \right]$, $M_w(\zeta')=\max_{i,l}\n X^{A_i}(l)\n^\zeta_{L^{\zeta,w}(\P)}$ (where $\n .\n^\zeta_{L^{\zeta,w}(\P)}$ denotes the weak $L^\zeta$-norm), $\tilde{M}(\zeta)=\max_{i,l}\E\left[ |X^{A_i}(l)Y^{A_i}|^{\zeta} \right]$ and
$$C_\textbf{n}=\left\{\n \left(\hat{G}_+(\textbf{n})+\gamma\hat{G}_\Delta(\textbf{n}) \right)^{-1} \n_{\mathit{l}^2} \le C\right\} $$ 
\begin{itemize}
\item[(1)]If $\tilde{M}(\zeta)<\infty$ for $\zeta>q$, $M_w(\zeta')<\infty$ for some $\zeta'>4$ and \\$C>\n \left( G_+ +\gamma G_\Delta\right)^{-1} \n_{\mathit{l}^2}+ \n \left( G_+ +\gamma G_\Delta\right)^{-1} \n_{\mathit{l}^2}^3$. Then
$$Var_q\left(\hat{\beta}_\gamma(\textbf{n}) \|C_\textbf{n}\right)
\le
\frac{D}{\left(n_{A_0}\wedge...\wedge n_{A_k}\right)^{\left(\left(\alpha\zeta'-1\right)\frac{\zeta-2q}{\zeta}\right)\wedge \frac{q}{2}}}.$$
for any $0<\alpha<\frac 14$, $n_{A_i}\ge N$, where $N$ depend on $p,\zeta,M(\zeta),q,\gamma,k$ and $C$. In terms of $M,M_w$ and $\tilde{M}$ we have that $D$ is proportional to $(1\vee M(2q\wedge 4(\alpha\zeta'-1))\vee\tilde{M}(\zeta)\vee M_w(\zeta'))^{1+\frac{2}{\zeta-1}}$
\item[(2)] If we assume $M(\zeta)<\infty$, $\max_i\E\left[\left|Y^{A_i}\right|^{\zeta}\right]<\infty$ for $\zeta>2q$, and $C$ is as in (1) then we have an analogous bound to the one in (1), albeit with a different constant $\tilde{D}$.
\end{itemize}
 The explicit expressions for $D$ and $N$ can be found at the end of the proof. 
\end{thm}
\begin{rema}\label{rateremark}
In particular, if $M(\zeta)<\infty$ for all $\zeta\ge\frac{2+(4\alpha+1)q}{4\alpha}+ \sqrt{\left(\frac{2+(4\alpha+1)q}{4\alpha}\right)^2-\frac{2q}{\alpha}}$ then \\$\left(\alpha\zeta-1\right)\frac{\zeta-2q}{\zeta}\ge \frac{q}{2}$ which implies that the rate of convergence is is $\frac{1}{\left(n_{A_0}\wedge...\wedge n_{A_k}\right)^{q/2}}$.
\end{rema}

\subsection{Conditional variance, conditioning on a stopped sigma-algebra}
\label{sec:stopping}

Consider a sequence of $k+1$-tuples, with entries in $\N$, $\{\textbf{n}_l\}_l$ such that if $l_1< l_2$ then $\textbf{n}_{l_1}(j)\le \textbf{n}_{l_2}(j)$, $\forall 0\le j\le k$ (we use a slight abuse of notation here, allowing for zero indexation since we want to associate $\textbf{n}_l(0)$ with the observational environment) and let $\mathcal{F}_l$ be the sigma algebra generated by the random variables $\mathbb{X}^{A_i}(\textbf{n}_l)_{u,v}$, for $1\le u\le \textbf{n}_l(i)$, $1\le v\le p$, and $\mathbb{Y}^{A_i}_{u}(\textbf{n}_l)$ for $1\le u\le \textbf{n}_l(i)$. Then it is clear that $\{\mathcal{F}_l\}_{l\ge 1}$ forms a filtration.
Given $\delta\in (0,1)$, let
\begin{align*}
\tau_{\delta,1}
=
\inf\left\{
l\ge 1: 
\n G_++\gamma G_\Delta+\left(\hat{G}_+(\textbf{n}_l)+\gamma\hat{G}_\Delta(\textbf{n}_l)\right)\n_{\mathit{l}^2}
<
\delta \n \left(G_++\gamma G_\Delta\right)^{-1}\n_{\mathit{l}^2} 
\right\}
\end{align*}
and
\begin{align*}
&\tau_{\delta,m}
=
\inf\left\{l> \tau_{\delta,m-1}: 
\n G_++\gamma G_\Delta-\left(\hat{G}_+(\textbf{n}_l)+\gamma\hat{G}_\Delta(\textbf{n}_l)\right)\n_{\mathit{l}^2}
<
\delta \n \left(G_++\gamma G_\Delta\right)^{-1}\n_{\mathit{l}^2}  
\right\}
\end{align*}
for $m>1$. Let us make two elementary observations about these stopping times. The first observation to make is that all of these stopping times are a.s.~finite. This fact follows since $\hat{G}_+(\textbf{n}_l)-\gamma\hat{G}_w(\textbf{n}_l)\xrightarrow{a.s.} G_++\gamma G_\Delta$ due to the law of large numbers applied entry-wise to the matrices $\{\hat{G}_+(\textbf{n}_l)-\gamma\hat{G}_w(\textbf{n}_l)\}_l$. A second observation is that $\P\left(\bigcup_{l=1}^\infty\bigcap_{m=l}^\infty\left\{\tau_{\delta,m+1}=\tau_{\delta,m}+1\right\}\right)=1$, i.e.~with probability 1, after some finite (but random) time there are no gaps between the stopping times. This means that $\{\hat{G}_+(\textbf{n}_{l})-\gamma\hat{G}_w(\textbf{n}_{l})\}_l$ never leaves the ball $B_{\delta \n \left(G_++\gamma G_\Delta\right)^{-1}\n_{\mathit{l}^2}}\left(G_++\gamma G_\Delta\right)$ after this time.
Let $\mathcal{X}_l=\sigma\left(X^{A_0}_1,...,X^{A_0}_{\textbf{n}_l(1)},...,X^{A_k}_1,...,X^{A_k}_{\textbf{n}_l(k+1)}\right)$ then $\{\mathcal{X}_l\}_l$ also forms a filtration and $\{\tau_{\delta,m}\}_m$ are also stopping times for this filtration. With this framework in mind we will now study conditional variance for the estimator with respect to this filtration.
\begin{thm}\label{aopt}
\begin{align*}
&Var\left(\hat{\beta}_\gamma\left(\textbf{n}_{\tau_{\delta,m}}\right) \| \mathcal{X}_{\tau_{\delta,m}}\right)
\\
&= 
\E\left[ \hat{\beta}_\gamma^T\left(\textbf{n}_{\tau_{\delta,m}}\right)\hat{\beta}_\gamma\left(\textbf{n}_{\tau_{\delta,m}}\right) \| \mathcal{X}_{\tau_{\delta,m}}\right]
-
\E\left[ \hat{\beta}_\gamma^T\left(\textbf{n}_{\tau_{\delta,m}}\right) \| \mathcal{X}_{\tau_{\delta,m}}\right]
\E\left[ \hat{\beta}_\gamma\left(\textbf{n}_{\tau_{\delta,m}}\right) \| \mathcal{X}_{\tau_{\delta,m}}\right]
\end{align*} is well-defined.
If we suppose that $(Y^{A_i},X^{A_i})$ has a joint density and that \\$\sup_{x\in \R^p} Var\left(Y^{A_i}\|X^{A_i}=x\right)<\infty$ then  
$$Var\left(\hat{\beta}_\gamma\left(\textbf{n}_{\tau_{\delta,m}}\right) \| \mathcal{X}_{\tau_{\delta,m}}\right)
\le
C_m\sum_{i=0}^k\frac{1}{\textbf{n}_{\tau_{\delta,m}}(i)} \hspace{3mm}a.s.,$$
where $\{C_m\}_m\in O_P(1)$.
\end{thm}
\begin{proof}
\textbf{Step 1}: The conditional variance is well-defined\\
Fix $m\ge 1$ and let, for the sake of brevity $\tau:=\tau_{\delta,m}$. We have that

\small
\begin{align}\label{intbound}
&|\hat{\beta}_\gamma^T(\textbf{n}_{\tau})\hat{\beta}_\gamma(\textbf{n}_{\tau})|
\nonumber
\\
&=\left| \sum_{i,j\not=0}^k\frac{(1+\gamma)w_i(1+\gamma)w_j}{\textbf{n}_{\tau}(i)\textbf{n}_{\tau}(j)}((\mathbb{Y}^{A_i})^T(\textbf{n}_{\tau})\mathbb{X}^{A_i} (\textbf{n}_{\tau})\left(\left(\hat{G}_+(\textbf{n}_{\tau})+\gamma\hat{G}_\Delta(\textbf{n}_{\tau}) \right)^{-1}\right)^T\left(\hat{G}_+(\textbf{n}_{\tau})+\gamma\hat{G}_\Delta (\textbf{n}_{\tau})\right)^{-1}(\mathbb{X}^{A_j}(\textbf{n}_{\tau}))^T\mathbb{Y}^{A_j}(\textbf{n}_{\tau})
\right.\nonumber
\\
&\left. +\sum_{i=1}^k\frac{(1+\gamma)w_i(1-\gamma)}{\textbf{n}_{\tau}(i)\textbf{n}_{\tau}(0)}((\mathbb{Y}^{A_i})^T(\textbf{n}_{\tau})\mathbb{X}^{A_i} (\textbf{n}_{\tau})\left(\left(\hat{G}_+(\textbf{n}_{\tau})+\gamma\hat{G}_\Delta(\textbf{n}_{\tau}) \right)^{-1}\right)^T\left(\hat{G}_+(\textbf{n}_{\tau})+\gamma\hat{G}_\Delta (\textbf{n}_{\tau})\right)^{-1}(\mathbb{X}^{A_0}(\textbf{n}_{\tau}))^T\mathbb{Y}^{A_0}(\textbf{n}_{\tau})
\right.\nonumber
\\
&\left. +\sum_{j=1}^k\frac{(1-\gamma)(1+\gamma)w_j}{\textbf{n}_{\tau}(0)\textbf{n}_{\tau}(j)}((\mathbb{Y}^{A_0})^T(\textbf{n}_{\tau})\mathbb{X}^{A_0} (\textbf{n}_{\tau})\left(\left(\hat{G}_+(\textbf{n}_{\tau})+\gamma\hat{G}_\Delta(\textbf{n}_{\tau}) \right)^{-1}\right)^T\left(\hat{G}_+(\textbf{n}_{\tau})+\gamma\hat{G}_\Delta (\textbf{n}_{\tau})\right)^{-1}(\mathbb{X}^{A_j}(\textbf{n}_{\tau}))^T\mathbb{Y}^{A_j}(\textbf{n}_{\tau}) \right|\nonumber
\\
&\le\sum_{i,j\not=0}^k\left|\frac{(1+\gamma)w_i(1+\gamma)w_j}{\textbf{n}_{\tau}(i)\textbf{n}_{\tau}(j)}\right|
\left|\langle\left(\hat{G}_+(\textbf{n}_{\tau})+\gamma\hat{G}_\Delta(\textbf{n}_{\tau}) \right)^{-1}(\mathbb{X}^{A_i}(\textbf{n}_{\tau}))^T\mathbb{Y}^{A_i}(\textbf{n}_{\tau}),\left(\hat{G}_+(\textbf{n}_{\tau})+\gamma\hat{G}_\Delta(\textbf{n}_{\tau}) \right)^{-1}(\mathbb{X}^{A_j}(\textbf{n}_{\tau}))^T\mathbb{Y}^{A_j}(\textbf{n}_{\tau}) \rangle_{\mathit{l}^2}\right|\nonumber
\\ 
&+2\sum_{i=1}^k\left|\frac{(1+\gamma)w_i(1-\gamma)}{\textbf{n}_{\tau}(i)\textbf{n}_{\tau}(0)}\right|
\left|\langle \left(\hat{G}_+(\textbf{n}_{\tau})+\gamma\hat{G}_\Delta(\textbf{n}_{\tau}) \right)^{-1}(\mathbb{X}^{A_i}(\textbf{n}_{\tau}))^T\mathbb{Y}^{A_i}(\textbf{n}_{\tau}),\left(\hat{G}_+(\textbf{n}_{\tau})+\gamma\hat{G}_\Delta(\textbf{n}_{\tau}) \right)^{-1}(\mathbb{X}^{A_0}(\textbf{n}_{\tau}))^T\mathbb{Y}^{A_0}(\textbf{n}_{\tau}) \rangle_{\mathit{l}^2}\right|\nonumber
\\
&+\left|\frac{(1-\gamma)(1-\gamma )}{\textbf{n}_{\tau}(0)\textbf{n}_{\tau}(0)}\right|
\left|\langle \left(\hat{G}_+(\textbf{n}_{\tau})+\gamma\hat{G}_\Delta(\textbf{n}_{\tau}) \right)^{-1}(\mathbb{X}^{A_0}(\textbf{n}_{\tau}))^T\mathbb{Y}^{A_0}(\textbf{n}_{\tau}),\left(\hat{G}_+(\textbf{n}_{\tau})+\gamma\hat{G}_\Delta(\textbf{n}_{\tau}) \right)^{-1}(\mathbb{X}^{A_0}(\textbf{n}_{\tau}))^T\mathbb{Y}^{A_0}(\textbf{n}_{\tau}) \rangle_{\mathit{l}^2}\right|\nonumber
\\ 
&\le\sum_{i,j=0}^k\left|\frac{(1+\gamma)w_i(1+\gamma)w_j}{\textbf{n}_{\tau}(i)\textbf{n}_{\tau}(j)}\right|\n 
\left(\hat{G}_+(\textbf{n}_{\tau})+\gamma\hat{G}_\Delta(\textbf{n}_{\tau}) \right)^{-1}(\mathbb{X}^{A_i}(\textbf{n}_{\tau}))^T\mathbb{Y}^{A_i}(\textbf{n}_{\tau})\n_{\mathit{l}^2} \cdot\nonumber
\\
&\n\left(\hat{G}_+(\textbf{n}_{\tau})+\gamma\hat{G}_\Delta(\textbf{n}_{\tau}) \right)^{-1}(\mathbb{X}^{A_j}(\textbf{n}_{\tau}))^T\mathbb{Y}^{A_j} (\textbf{n}_{\tau})\n_{\mathit{l}^2}\le\nonumber
\\
&\sum_{i,j=0}^k\left|\frac{(1+\gamma)w_i(1+\gamma)w_j}{\textbf{n}_{\tau}(i)\textbf{n}_{\tau}(j)}\right|\n \left(\hat{G}_+(\textbf{n}_{\tau})+\gamma\hat{G}_\Delta(\textbf{n}_{\tau}) \right)^{-1}\n_{\mathit{l}^2}^2 \n(\mathbb{X}^{A_i}(\textbf{n}_{\tau}))^T\mathbb{Y}^{A_i}(\textbf{n}_{\tau})\n_{\mathit{l}^2} 
\n(\mathbb{X}^{A_j}(\textbf{n}_{\tau}))^T\mathbb{Y}^{A_j} (\textbf{n}_{\tau})\n_{\mathit{l}^2}.
\end{align}
\normalsize
In order to bound the right-most expression in \eqref{intbound} we consider the following inequalities   
\begin{align*}
&\n \left(\hat{G}_+(\textbf{n}_{\tau})+\gamma\hat{G}_\Delta(\textbf{n}_{\tau}) \right)^{-1}\n_{\mathit{l}^2}^2
\le 
2\n \left(G_++\gamma G_\Delta\right)^{-1}-\left(\hat{G}_+(\textbf{n}_{\tau})+\gamma\hat{G}_\Delta(\textbf{n}_{\tau})\right)^{-1} \n_{\mathit{l}^2}^2
+
2\n \left(G_++\gamma G_\Delta\right)^{-1} \n_{\mathit{l}^2}^2
\\
&\le 2(1+\delta) \n \left(G_++\gamma G_\Delta\right)^{-1} \n^2\n G_++\gamma G_\Delta-\left(\hat{G}_+(\textbf{n}_\tau)+\gamma\hat{G}_\Delta(\textbf{n}_\tau)\right)\n_{\mathit{l}^2}
+
2\n \left(G_++\gamma G_\Delta\right)^{-1} \n_{\mathit{l}^2}^2
\\
&\le 2(1+\delta)\delta \n \left(G_++\gamma G_\Delta\right)^{-1} \n_{\mathit{l}^2}^3+2\n \left(G_++\gamma G_\Delta\right)^{-1} \n_{\mathit{l}^2}^2,
\end{align*}
where we used inequality (1.2) from the Appendix (which applies by definition of the stopping time) in the first inequality and the definition of the stopping time in the second inequality. Therefore

\begin{align*}
&\E\left[|\hat{\beta}_\gamma^T(\textbf{n}_{\tau})\hat{\beta}_\gamma(\textbf{n}_{\tau})|\right]
\le
\left(2(1+\delta)\delta \n \left(G_++\gamma G_\Delta\right)^{-1} \n_{\mathit{l}^2}^3+2\n \left(G_++\gamma G_\Delta\right)^{-1} \n_{\mathit{l}^2}^2\right)\cdot
\\
&\left(\sum_{i,j=0, i\not=j}^k\left|(1+\gamma)w_i(1+\gamma)w_j\right|\E\left[\frac{\n(\mathbb{X}^{A_i}(\textbf{n}_{\tau}))^T\mathbb{Y}^{A_i} (\textbf{n}_{\tau})\n_{\mathit{l}^2}}{\textbf{n}_\tau(i)}\right]
\E\left[\frac{\n(\mathbb{X}^{A_j}(\textbf{n}_{\tau}))^T\mathbb{Y}^{A_j} (\textbf{n}_{\tau})\n_{\mathit{l}^2}}{\textbf{n}_\tau(j)}\right]
+
\right.
\\
&\left.\sum_{i=1}^k(1+\gamma)w_i^2\E\left[\frac{\n(\mathbb{X}^{A_i}(\textbf{n}_{\tau}))^T\mathbb{Y}^{A_i} (\textbf{n}_{\tau})\n_{\mathit{l}^2}^2}{\textbf{n}_{\tau}(i)^2}\right]\right),
\end{align*}
where we used the independence property for the environmental samples. By the Jensen inequality it suffices to show that $\E\left[\frac{\n(\mathbb{X}^{A_i}(\textbf{n}_{\tau}))^T\mathbb{Y}^{A_i} (\textbf{n}_{\tau})\n_{\mathit{l}^2}^2}{\textbf{n}_{\tau}(i)^2}\right]<\infty$ for every $i$. Expanding the $\mathit{l}^2$-norms above and using the definitions of $\mathbb{X}^{A_i}$ and $\mathbb{Y}^{A_i}$ we find that,
\begin{align*}
\E\left[\frac{\n(\mathbb{X}^{A_i}(\textbf{n}_{\tau}))^T\mathbb{Y}^{A_i} (\textbf{n}_{\tau})\n_{\mathit{l}^2}^2}{\textbf{n}_{\tau}(i)^2}\right]
&=
\sum_{u=1}^p\E\left[\left( \frac{\sum_{v=1}^{\textbf{n}_{\tau}(i)}\mathbb{X}_{v,u}^{A_i}(\textbf{n}_\tau(i))\mathbb{Y}_{v}^{A_i}(\textbf{n}_\tau(i)) }{\textbf{n}_\tau(i)}\right)^2\right]
\\
&=\sum_{u=1}^p\E\left[\left( \frac{\sum_{v=1}^{\textbf{n}_{\tau}(i)} X_{v}^{A_i}(u)Y_{v}^{A_i}}{\textbf{n}_\tau(i)}\right)^2\right].
\end{align*}
We must show that all $p$ terms in sum on right-hand side above is finite. Fix $i$ and $1\le u\le p$, let $H_v=X_{v}^{A_i}(u)Y_{v}^{A_i}$, then $\{H_v\}_v$ is i.i.d. distributed according to $X^{A_i}(u)Y^{A_i}$. From the SEM it follows that $X^{A_i}(u)=\left(I-B\right)^{-1}_{u,.}\cdot\left((0,A)+\epsilon\right)$ and $Y^{A_i}=\left(I-B\right)^{-1}_{1,.}\cdot\left((0,A)+\epsilon\right)$, therefore $\{H_v\}_v$ is an $L^2$ sequence. Note that if we let $M_m=\sum_{v=1}^m \left(H_v-\E\left[H_v\right]\right)$ then $\{M_m\}_m$ is an $L^2$ martingale (and hence also a u.i. martingale) with respect to the filtration $\{\mathcal{F}_m\}_m$. Since $\textbf{n}_\tau(i)$ is a stopping time with respect to the filtration $\{\mathcal{X}_m\}_m$ and $\mathcal{X}_m \subset \F_m$, it must also be stopping time with respect to $\{\mathcal{F}_m\}_m$. It follows that $\{M_{m\wedge\textbf{n}_\tau(i)}\}_m$ is an $L^2$ martingale (and again, also a u.i. martingale) with respect to the filtration $\{\mathcal{F}_{m\wedge\textbf{n}_\tau(i)}\}_m$ and consequently converges both a.s. and in $L^2$. Since $M_{m\wedge\textbf{n}_\tau(i)}\xrightarrow{a.s.}\sum_{v=1}^{\textbf{n}_\tau(i)} H_v-\textbf{n}_\tau(i)\E\left[H_1\right]$ and since $\textbf{n}_\tau(i)\ge 1$ we obviously have that 
$$\frac{1}{\textbf{n}_\tau(i)}\left| \sum_{v=1}^{\textbf{n}_\tau(i)} H_v-\textbf{n}_\tau(i)\E\left[H_1\right] \right|= \left|\frac{1}{\textbf{n}_\tau(i)} \sum_{v=1}^{\textbf{n}_\tau(i)} H_v-\E\left[H_1\right] \right|\in L^2.$$
Since $\left|\E\left[H_1\right]\right|<\infty$ we conclude that $\frac{1}{\textbf{n}_\tau(i)} \sum_{v=1}^{\textbf{n}_\tau(i)} H_v\in L^2$ and this implies that \\$\E\left[\frac{\n(\mathbb{X}^{A_i}(\textbf{n}_{\tau}))^T\mathbb{Y}^{A_i} (\textbf{n}_{\tau})\n_{\mathit{l}^2}^2}{\textbf{n}_{\tau}(i)^2}\right]<\infty$.
We conclude that $E\left[ \hat{\beta}_\gamma^T(\textbf{n}_\tau)\hat{\beta}_\gamma(\textbf{n}_\tau)\|\mathcal{X}_\tau\right]$ is well-defined. For $1\le u\le p$, 
$$\E\left[ \left|\hat{\beta}_\gamma^T(\textbf{n}_\tau)_u\right|\right]\le \sqrt{\E\left[ \left|\hat{\beta}_\gamma^T(\textbf{n}_\tau)_u^2\right|\right]}\le\sqrt{\E\left[ \n \hat{\beta}_\gamma^T(\textbf{n}_\tau) \n_{\mathit{l}^2}^2 \right]},$$ 
which we have already shown is finite. We may therefore conclude that the conditional variance is well-defined.
\\
\\
\textbf{Step 2}: computing the bound\\
We wish to compute
$$E\left[ \hat{\beta}_\gamma^T(\textbf{n}_\tau)\hat{\beta}_\gamma(\textbf{n}_\tau)\|\mathcal{X}_\tau\right]-\E\left[ \hat{\beta}_\gamma^T(\textbf{n}_\tau)\|\mathcal{X}_\tau\right]\E\left[ \hat{\beta}_\gamma(\textbf{n}_\tau)\|\mathcal{X}_\tau\right].$$
To this end we first compute (where we drop the indexation in the sample matrices for brevity)
\small
\begin{align*}
&\E\left[ \hat{\beta}_\gamma^T(\textbf{n}_\tau)\hat{\beta}_\gamma(\textbf{n}_\tau)\|\mathcal{X}_\tau\right]   
=
\E\left[ \left(\hat{Z}_{+}^T+\gamma\hat{Z}_{\Delta}^T\right)\left(\left(\hat{G}_++\gamma\hat{G}_\Delta \right)^{-1}\right)^T\left(\hat{G}_++\gamma\hat{G}_\Delta \right)^{-1}\left(\hat{Z}_{+}+\gamma\hat{Z}_{\Delta}\right)\|\mathcal{X}_\tau\right]   
\\
=&\E\left[ 
\sum_{i=1}^k\frac{1}{\textbf{n}_\tau(i)}(1+\gamma)w_i(\mathbb{Y}^{A_i})^T\mathbb{X}^{A_i}\left(\left(\hat{G}_++\gamma\hat{G}_\Delta \right)^{-1}\right)^T\left(\hat{G}_++\gamma\hat{G}_\Delta \right)^{-1}	
\sum_{i=1}^k\frac{1}{\textbf{n}_\tau(i)}(1+\gamma)w_i(\mathbb{X}^{A_i})^T\mathbb{Y}^{A_i}\|\mathcal{X}_\tau\right]
\\
+&2\E\left[ 
\frac{1}{\textbf{n}_\tau(0)}(1-\gamma)(\mathbb{Y}^{A_0})^T\mathbb{X}^{A_0}\left(\left(\hat{G}_++\gamma\hat{G}_\Delta \right)^{-1}\right)^T\left(\hat{G}_++\gamma\hat{G}_\Delta \right)^{-1}	
\sum_{i=1}^k\frac{1}{\textbf{n}_\tau(i)}(1+\gamma)w_i(\mathbb{X}^{A_i})^T\mathbb{Y}^{A_i}\|\mathcal{X}_\tau\right]
\\
+&\E\left[ 
\frac{1}{\textbf{n}_\tau(0)}(1-\gamma)(\mathbb{Y}^{A_0})^T\mathbb{X}^{A_0}\left(\left(\hat{G}_++\gamma\hat{G}_\Delta \right)^{-1}\right)^T\left(\hat{G}_++\gamma\hat{G}_\Delta \right)^{-1}	
\frac{1}{\textbf{n}_\tau(0)}(1-\gamma)(\mathbb{X}^{A_0})^T\mathbb{Y}^{A_0}\|\mathcal{X}_\tau\right]
\\
=&\sum_{i,j=1}^k\frac{1}{\textbf{n}_\tau(i)\textbf{n}_\tau(j)}(1+\gamma)w_i(1+\gamma)w_j\E\left[(\mathbb{Y}^{A_i})^T\mathbb{X}^{A_i} \left(\left(\hat{G}_++\gamma\hat{G}_\Delta \right)^{-1}\right)^T\left(\hat{G}_++\gamma\hat{G}_\Delta \right)^{-1}(\mathbb{X}^{A_j})^T\mathbb{Y}^{A_j} \|\mathcal{X}_\tau\right]
\\
+&2\sum_{i=1}^k\frac{1}{\textbf{n}_\tau(i)\textbf{n}_\tau(0)}(1+\gamma)w_i(1-\gamma )\E\left[(\mathbb{Y}^{A_i})^T\mathbb{X}^{A_i} \left(\left(\hat{G}_++\gamma\hat{G}_\Delta \right)^{-1}\right)^T\left(\hat{G}_++\gamma\hat{G}_\Delta \right)^{-1}(\mathbb{X}^{A_0})^T\mathbb{Y}^{A_0} \|\mathcal{X}_\tau\right]
\\
+&\frac{1}{\textbf{n}_\tau(0)\textbf{n}_\tau(0)}(1-\gamma )(1-\gamma )\E\left[(\mathbb{Y}^{A_0})^T\mathbb{X}^{A_0} \left(\left(\hat{G}_++\gamma\hat{G}_\Delta \right)^{-1}\right)^T\left(\hat{G}_++\gamma\hat{G}_\Delta \right)^{-1}(\mathbb{X}^{A_0})^T\mathbb{Y}^{A_0} \|\mathcal{X}_\tau\right].
\end{align*}
\normalsize
If we now denote $M^{(i,j)}=\mathbb{X}^{A_i}\left(\left(\hat{G}_++\gamma\hat{G}_\Delta \right)^{-1}\right)^T\left(\hat{G}_++\gamma\hat{G}_\Delta \right)^{-1}(\mathbb{X}^{A_j})^T$ then
\small
\begin{align*}
&\E\left[(\mathbb{Y}^{A_i})^T\mathbb{X}^{A_i}\left(\left(\hat{G}_++\gamma\hat{G}_\Delta \right)^{-1}\right)^T\left(\hat{G}_++\gamma\hat{G}_\Delta \right)^{-1}  (\mathbb{X}^{A_j})^T\mathbb{Y}^{A_j}\|\mathcal{X}_\tau\right]=
\sum_{u=1}^{\textbf{n}_\tau(i)}\sum_{v=1}^{\textbf{n}_\tau(j)}M^{(i,j)}_{uv}\E\left[ Y_u^{A_i}Y_v^{A_j} \|\mathcal{X}_\tau\right].
\end{align*}
\normalsize
When $u\not=v$ or $i\not=j$, 
$$\E\left[ Y_u^{A_i}Y_v^{A_j} \|\mathcal{X}_\tau\right]I_{\tau=m}=\E\left[ Y_u^{A_i}Y_v^{A_j} \|\mathcal{X}_m\right]I_{\tau=m}=\E\left[ Y_u^{A_i}Y_v^{A_j} \|X_u^{A_i},X_v^{A_i}\right]I_{\tau=m},$$ 
where we used the "irrelevance of independent information" property of conditional expectation in the last equality. Furthermore 
$$\E\left[ Y_u^{A_i}Y_v^{A_j} \|X_u^{A_i},X_v^{A_j}\right]=\phi_1\left(X_u^{A_i},X_v^{A_j}\right)$$ 
where
\begin{align*}
\phi_1\left(x^{a_i}_u,x^{a_j}_v\right)&=
\E\left[ Y_u^{A_i}Y_v^{A_j} \|X^{A_i}_u=x^{a_i}_u,X^{A_j}_v=x^{a_j}_v\right]
\\
&=
\int y_1y_2\frac{f_{Y^{A_i}_u,Y^{A_j}_v,X^{A_i}_u,X^{A_j}_v}(y_1,y_2,x^{a_i}_u,x^{a_j}_v)}{f_{X^{A_i}_u,X^{A_j}_v}(y_1,y_2,x^{a_i}_u,x^{a_j}_v)}dy_1dy_2 
\\
&=\frac{1}{f_{X^{A_i}}(x^{a_i}_u)f_{X^{A_j}}(x^{a_j}_v)}\int yf_{Y^{A_i},X^{A_i}}(y,x^{a_i}_u)dy\int yf_{Y^{A_j},X^{A_j}}(y,x^{a_j}_v)dy 
\\
&=\E\left[ Y_u^{A_i} \|X^{A_i}_u=x^{a_i}_u\right]\E\left[Y_v^{A_j} \|X^{A_j}_v=x^{a_j}_v\right]
=\E\left[ Y^{A_i} \|X^{A_i}=x^{a_i}_u\right]\E\left[Y^{A_j} \|X^{A_j}=x^{a_j}_v\right]
\end{align*}
hence 
\begin{align*}
\E\left[ Y_u^{A_i}Y_v^{A_j} \|\mathcal{X}_\tau\right]
&=\sum_{m=1}^\infty\E\left[ Y_u^{A_i}Y_v^{A_j} \|\mathcal{X}_\tau\right]I_{\tau=m}=\sum_{m=1}^\infty\E\left[ Y_u^{A_i}Y_v^{A_j} \|\mathcal{X}_m\right]I_{\tau=m}
\\
&=\sum_{m=1}^\infty\E\left[ Y^{A_i} \|X^{A_i}=X^{A_i}_u\right]\E\left[Y^{A_j} \|X^{A_j}=X^{A_j}_v\right] I_{\tau=m}
\\
&=
\E\left[ Y^{A_i} \|X^{A_i}=X^{A_i}_u\right]\E\left[Y^{A_j} \|X^{A_j}=X^{A_j}_v\right].
\end{align*}
When $u=v$ and $i=j$, 
$$\E\left[ (Y_u^{A_i})^2 \|\mathcal{X}_m\right]=\phi_2\left(X_u^{A_i}\right)$$ 
where
\begin{align*}
\phi_2\left(x_u^{a_i}\right)&=
\E\left[ (Y_u^{A_i})^2 \|X_u^{A_i}=x_u^{a_i}\right]=
\frac{1}{f_{X^{A_i}}(x^{a_i}_u)^2}\int y^2f_{Y^{A_i},X^{A_i}}(y,x^{a_i}_u)dy
\\
&=\E\left[ (Y^{A_i})^2 \|X^{A_i}=x_u^{a_i}\right].
\end{align*}
Therefore, by analogous computations to the previous case we have that when $u=v$ and $i=j$,
$$\E\left[ (Y_u^{A_i})^2 \|\mathcal{X}_\tau\right]=\E\left[ (Y^{A_i})^2 \|X^{A_i}=X_u^{A_i}\right]. $$
If we denote 
$$\tilde{Y}_i=\left(\E\left[ Y^{A_i}\|X^{A_i}=X_1^{A_i}\right],...,\E\left[ Y^{A_i}\|X^{A_i}=X_{\textbf{n}_{\tau(i)}}^{A_i}\right]\right)^T$$
then
$$ \tilde{Y}_i M^{(i,j)}\tilde{Y}_j
=
\sum_{u=1}^{\textbf{n}_\tau(i)}\sum_{v=1}^{\textbf{n}_\tau(j)}M^{(i,j)}_{uv}\E\left[ Y^{A_i} \|X^{A_i}=X^{A_i}_u\right]\E\left[Y^{A_j} \|X^{A_j}=X^{A_j}_v\right].$$
So when $i\not=j$
$$
\E\left[(\mathbb{Y}^{A_i})^T\mathbb{X}^{A_i}\left(\left(\hat{G}_++\gamma\hat{G}_\Delta \right)^{-1}\right)^T\left(\hat{G}_++\gamma\hat{G}_\Delta \right)^{-1}  (\mathbb{X}^{A_j})^T\mathbb{Y}^{A_j}\|\mathcal{X}_\tau\right]= \tilde{Y}_i M^{(i,j)}\tilde{Y}_j
$$
and when $i=j$
\begin{align*}
&\E\left[(\mathbb{Y}^{A_i})^T\mathbb{X}^{A_i}\left(\left(\hat{G}_++\gamma\hat{G}_\Delta \right)^{-1}\right)^T\left(\hat{G}_++\gamma\hat{G}_\Delta \right)^{-1}  (\mathbb{X}^{A_i})^T\mathbb{Y}^{A_i}\|\mathcal{X}_\tau\right]=\tilde{Y}_i M^{(i,i)}\tilde{Y}_i
\\
&+\sum_{u=1}^{\textbf{n}_\tau(i)}M^{(i,i)}_{uu}\left(\E\left[ (Y^{A_i})^2 \|X^{A_i}=X_u^{A_i}\right]-\E\left[ Y^{A_i} \|X^{A_i}=X_u^{A_i}\right]^2\right)
\\
&=\tilde{Y}_i M^{(i,i)}\tilde{Y}_i+\sum_{u=1}^{\textbf{n}_\tau(i)}M^{(i,i)}_{uu}V\left(Y^{A_i}|X^{A_i}=X^{A_i}_u\right).
\end{align*}
Thus
\begin{align*}
\E\left[ \hat{\beta}_\gamma^T(\textbf{n}_\tau)\hat{\beta}_\gamma(\textbf{n}_\tau)\|\mathcal{X}_\tau\right]
&=\sum_{i,j=1, i\not=j}^k\frac{1}{\textbf{n}_\tau(i)\textbf{n}_\tau(j)}(1+\gamma)^2w_iw_j(\tilde{Y}_i)^TM^{(i,j)}\tilde{Y}_j
\\
&+\sum_{i=1}^k\frac{1}{\textbf{n}_\tau(i)^2}(1+\gamma)^2w_i^2\sum_{u=1}^{\textbf{n}_\tau(i)}M^{(i,i)}_{uu}V\left(Y^{A_i}|X^{A_i}=X^{A_i}_u\right)
\\
&+
\sum_{i=1}^k\frac{1}{\textbf{n}_\tau(i)^2}(1+\gamma)^2w_i^2\tilde{Y}_i M^{(i,i)}\tilde{Y}_i
\\
&+\frac{(1-\gamma)^2}{\textbf{n}_\tau(0)^2}\sum_{u=1}^{\textbf{n}_\tau(i)}M^{(0,0)}_{uu}V\left(Y^{A_0}|X^{A_0}=X^{A_0}_u\right)
+
\frac{(1-\gamma)^2}{\textbf{n}_\tau(0)^2}\tilde{Y}_0 M^{(0,0)}\tilde{Y}_0
\\
&+
2\sum_{i=1}^k\frac{1}{\textbf{n}_\tau(i)\textbf{n}_\tau(0)}(1+\gamma)w_i(1-\gamma )(\tilde{Y}_i)^TM^{(i,0)}\tilde{Y}_0
\end{align*}
For the second term of the variance expression,
\small
\begin{align*}
&\E\left[ \hat{\beta}_\gamma(\textbf{n}_\tau(i))\|\mathcal{X}_\tau\right]^T\E\left[ \hat{\beta}_\gamma(\textbf{n}_\tau(i))\|\mathcal{X}_\tau\right]
\\&=\E\left[ \left(\hat{Z}_{+}^T+\gamma\hat{Z}_{\Delta}^T\right)\left(\left(\hat{G}_++\gamma\hat{G}_\Delta \right)^{-1}\right)^T\|\mathcal{X}_\tau\right] 		\E\left[ \left(\hat{G}_++\gamma\hat{G}_\Delta \right)^{-1}\left(\hat{Z}_{+}+\gamma\hat{Z}_{\Delta}\right)\|\mathcal{X}_\tau\right]
\\
&=\sum_{i,j=1}^k\frac{(1+\gamma)w_i(1+\gamma)w_j}{\textbf{n}_\tau(i)\textbf{n}_\tau(j)}
\E\left[ (\mathbb{Y}^{A_i})^T\mathbb{X}^{A_i} \left(\left(\hat{G}_++\gamma\hat{G}_\Delta \right)^{-1}\right)^T\|\mathcal{X}_\tau\right] 
\E\left[ \left(\hat{G}_++\gamma\hat{G}_\Delta \right)^{-1}(\mathbb{X}^{A_j})^T\mathbb{Y}^{A_j}\|\mathcal{X}_\tau\right]
\\
&+2\sum_{i=1}^k\frac{(1+\gamma)w_i(1-\gamma)}{\textbf{n}_\tau(i)\textbf{n}_\tau(0)}
\E\left[ (\mathbb{Y}^{A_i})^T\mathbb{X}^{A_i} \left(\left(\hat{G}_++\gamma\hat{G}_\Delta \right)^{-1}\right)^T\|\mathcal{X}_\tau\right] 
\E\left[ \left(\hat{G}_++\gamma\hat{G}_\Delta \right)^{-1}(\mathbb{X}^{A_0})^T\mathbb{Y}^{A_0}\|\mathcal{X}_\tau\right]
\\
&+\frac{(1-\gamma)^2}{\textbf{n}_\tau(0)\textbf{n}_\tau(0)}
\E\left[ (\mathbb{Y}^{A_0})^T\mathbb{X}^{A_0} \left(\left(\hat{G}_++\gamma\hat{G}_\Delta \right)^{-1}\right)^T\|\mathcal{X}_\tau\right] 
\E\left[ \left(\hat{G}_++\gamma\hat{G}_\Delta \right)^{-1}(\mathbb{X}^{A_0})^T\mathbb{Y}^{A_0}\|\mathcal{X}_\tau\right]
\\
&=\sum_{i,j=1}^k\frac{(1+\gamma)^2w_iw_j}{\textbf{n}_\tau(i)\textbf{n}_\tau(j)} \tilde{Y}_iM^{(i,j)}\tilde{Y}_j+
2\sum_{i=1}^k\frac{(1+\gamma)w_i(1-\gamma)}{\textbf{n}_\tau(i)\textbf{n}_\tau(0)} \tilde{Y}_iM^{(i,0)}\tilde{Y}_0
+
\frac{(1-\gamma)^2}{\textbf{n}_\tau(0)^2} \tilde{Y}_0M^{(0,0)}\tilde{Y}_0.
\end{align*}
\normalsize
We see that all terms in $\E\left[ \hat{\beta}_\gamma(\textbf{n}_\tau(i))\|\mathcal{X}_\tau\right]^T\E\left[ \hat{\beta}_\gamma(\textbf{n}_\tau(i))\|\mathcal{X}_\tau\right]$ cancel terms in \\$\E\left[ \hat{\beta}_\gamma(\textbf{n}_\tau(i))^T\hat{\beta}_\gamma(\textbf{n}_\tau(i))\|\mathcal{X}_\tau\right]$ and we are left with
\begin{align}\label{condvareq}
&Var(\hat{\beta}_\gamma(\textbf{n}_{\tau})\|\mathcal{X}_\tau)\nonumber
\\
&=\sum_{i=1}^k\frac{(1+\gamma)^2w_i^2}{\textbf{n}_\tau(i)^2}\sum_{u=1}^{\textbf{n}_\tau(i)}M^{(i,i)}_{uu}V\left(Y^{A_i}|X^{A_i}=X^{A_i}_u\right)
+
\frac{(1-\gamma)^2}{\textbf{n}_\tau(0)^2}\sum_{u=1}^{\textbf{n}_\tau(0)}M^{(0,0)}_{uu}V\left(Y^{A_0}|X^{A_0}=X^{A_0}_u\right).
\end{align}
To bound the above expression we compute,
\begin{align*}
&\left|\sum_{u=1}^{\textbf{n}_\tau(i)}M^{(i,i)}_{uu}V\left(Y^{A_i}|X^{A_i}=X^{A_i}_u\right)\right|
\\
&=
\left|\sum_{u=1}^{\textbf{n}_\tau(i)}V\left(Y^{A_i}|X^{A_i}=X^{A_i}_u\right)\left(\mathbb{X}^{A_i}\left(\left(\left(\hat{G}_++\gamma \hat{G}_\Delta \right)^{-1}\right)^T\left(\hat{G}_++\gamma \hat{G}_\Delta \right)^{-1}\right)(\mathbb{X}^{A_i})^T\right)_{u,u}\right|
\\
&\le 
\sup_x V\left(Y^{A_i}|X^{A_i}=x\right)\cdot\left| trace \left(\mathbb{X}^{A_i}\left(\left(\hat{G}_++\gamma \hat{G}_\Delta \right)^{-1}\right)^T\left(\hat{G}_++\gamma \hat{G}_\Delta \right)^{-1}(\mathbb{X}^{A_i})^T\right)\right|
\\
&=\sup_x V\left(Y^{A_i}|X^{A_i}=x\right)trace\left((\mathbb{X}^{A_i})^T\mathbb{X}^{A_i}\right)\cdot
trace\left(\left(\left(\hat{G}_++\gamma \hat{G}_\Delta \right)^{-1}\right)^T\left(\hat{G}_++\gamma \hat{G}_\Delta \right)^{-1}\right).
\end{align*}
If we let $\n . \n_F$ denote the Frobenius norm then
\begin{align*}
&trace\left(\left(\left(\hat{G}_++\gamma \hat{G}_\Delta \right)^{-1}\right)^T\left(\hat{G}_++\gamma \hat{G}_\Delta \right)^{-1}\right)
=\n \left(\hat{G}_++\gamma \hat{G}_\Delta \right)^{-1} \n_F
\\
&\le\n \left(\hat{G}_++\gamma \hat{G}_\Delta \right)^{-1} - \left(G_++\gamma G_\Delta \right)^{-1} \n_F +\n  \left(G_++\gamma G_\Delta \right)^{-1}  \n_F
\\
&\le\sqrt{p}\left(\n \left(\hat{G}_++\gamma \hat{G}_\Delta \right)^{-1}-
\left(G_++\gamma G_\Delta \right)^{-1}\n_{\mathit{l}^2}
+
\n
\left(G_++\gamma G_\Delta \right)^{-1}\n_{\mathit{l}^2}
\right) 
\\
&\le\sqrt{p}(1+\delta) \n \left(G_++\gamma G_\Delta\right)^{-1} \n_{\mathit{l}^2}\n G_++\gamma G_\Delta-\left(\hat{G}_+(\textbf{n}_\tau)+\gamma\hat{G}_\Delta(\textbf{n}_\tau)\right)\n_{\mathit{l}^2}
\\
&+
\sqrt{p}\n \left(G_++\gamma G_\Delta\right)^{-1} \n_{\mathit{l}^2}^2
\\
&\le\sqrt{p}\n \left(G_++\gamma G_\Delta\right)^{-1} \n_{\mathit{l}^2}^2\left(1+\delta(1+\delta)\right),
\end{align*}
where we used the definition of the stopping time and the fact that $\n .\n_F \le \sqrt{p}\n .\n_{\mathit{l}^2}$, for $p\times p$ matrices. Since $trace\left((\mathbb{X}^{A_i}(\textbf{n}))^T\mathbb{X}^{A_i}(\textbf{n})\right)=\sum_{l=1}^p\sum_{m=1}^{\textbf{n}_\tau(i)}(X^{A_i}_m(l))^2$ it follows from the law of large numbers that 
$$\frac{1}{n_{A_i}} trace\left((\mathbb{X}^{A_i}(\textbf{n}))^T\mathbb{X}^{A_i}(\textbf{n})\right)\xrightarrow{a.s.}\sum_{l=1}^p\E[X^{A_i}(l)^2],$$ 
as $n_{A_i}\to\infty$, and therefore 
$$\frac{1}{\textbf{n}_{\tau_{\delta,m}}(i)} trace\left((\mathbb{X}^{A_i}(\textbf{n}_{\tau_{\delta,m}}))^T\mathbb{X}^{A_i}(\textbf{n}_{\tau_{\delta,m}})\right)\xrightarrow{a.s.}\sum_{l=1}^p\E[X^{A_i}(l)^2],$$ 
as $m\to\infty$, which implies that $\left\{\frac{1}{\textbf{n}_{\tau_{\delta,m}}(i)} trace\left((\mathbb{X}^{A_i}(\textbf{n}_{\tau_{\delta,m}}))^T\mathbb{X}^{A_i}(\textbf{n}_{\tau_{\delta,m}})\right)\right\}_m$ is bounded in probability. Since
\begin{align*}
&V(\hat{\beta}_\gamma(\textbf{n}_{\tau})\|\mathcal{X}_\tau)
\\
&\le 
\sqrt{p}\n \left(G_++\gamma G_\Delta\right)^{-1} \n_{\mathit{l}^2}^2\left(1+\delta(1+\delta)\right)(1+\gamma)
\sum_{i=0}^k\sup_x V\left(Y^{A_i}|X^{A_i}=x\right)\frac{1}{\textbf{n}_\tau(i)^2}
\sum_{l=1}^p\sum_{m=1}^{\textbf{n}_\tau(i)}(X^{A_i}_m(l))^2,
\end{align*}
the result follows by letting
\footnotesize
$$C_m=\sqrt{p}\n \left(G_++\gamma G_\Delta\right)^{-1} \n_{\mathit{l}^2}^2\left(1+\delta(1+\delta)\right)(1+\gamma)\max_{0\le i\le k}\left(\sup_x V\left(Y^{A_i}|X^{A_i}=x\right)\left(\frac{1}{\textbf{n}_{\tau_{\delta,m}(i)}}\sum_{l=1}^p\sum_{m=1}^{\textbf{n}_{\tau_{\delta,m}(i)}}(X^{A_i}_m(l))^2\right)\right).$$
\normalsize
\end{proof}
\begin{rema}
Note that \eqref{condvareq} gives us a means to directly compute the conditional variance directly from our samples. This could then be used for bootstrapping purposes for choosing the weights.
\end{rema}

We now consider the conditional covariance matrix and bound its determinant. One can view this as related to d-optimality (although this is conditional) and use this optimize the weights by bootstrapping.

\begin{thm}\label{covar}
Let 
$$\mathbb{C}_{\tau_{\delta,m}}=\E\left[ \hat{\beta}(\textbf{n}_{\tau_{\delta,m}})\hat{\beta}_\gamma(\textbf{n}_{\tau_{\delta,m}})^T\|\mathcal{X}_{\tau_{\delta,m}}\right]-\E\left[ \hat{\beta}_\gamma(\textbf{n}_{\tau_{\delta,m}})\|\mathcal{X}_{\tau_{\delta,m}}\right]\E\left[ \hat{\beta}_\gamma(\textbf{n}_{\tau_{\delta,m}})\|\mathcal{X}_{\tau_{\delta,m}}\right]^T$$ 
be the conditional covariance matrix of $\hat{\beta}_\gamma$ at time $\tau_{\delta,m}$ for any $m\in \N$ and $\delta\in(0,1)$. If we suppose that $(Y^{A_i},X^{A_i})$ has a joint density and that $\sup_{x\in \R^p} V\left(Y^{A_i}\|X^{A_i}=x\right)<\infty$ then $\mathbb{C}$ is well-defined and 
$$\left|\det \left(\mathbb{C}_{\tau_{\delta,m}}\right)\right|\le C_m\left(\sum_{i=0}^k\frac{1}{\textbf{n}_{\tau_{\delta,m}}(i)}\right)^p\hspace{3mm}a.s.$$
where $\{C_m\}_m\in \mathcal{O}_P(1)$
and $D^i$ is a $\textbf{n}_{\tau_{\delta,m}}(i)\times\textbf{n}_{\tau_{\delta,m}}(i)$ diagonal matrix with 
$$ D^i_{(u,u)}=V\left(Y^{A_i}\|X^{A_i}=X^{A_i}_u\right), $$
for $1\le u\le \textbf{n}_{\tau_{\delta,m}}(i)$.
\end{thm}
\begin{proof}
We first explicitly compute $\mathbb{C}_{\tau_{\delta,m}}$. As in the proof of Theorem \ref{aopt} we let ${\tau}=\tau_{\delta,m}$ and drop most of the indexation (where there should be no risk of confusion) going forward. For the first term of $\mathbb{C}_{\tau_{\delta,m}}$,
\small
\begin{align*}
&\E\left[ \hat{\beta}_\gamma\hat{\beta}_\gamma^T\|\mathcal{X}_{\tau}\right]   =      \E\left[\left(\hat{G}_++\gamma\hat{G}_\Delta \right)^{-1} \left(\hat{Z}_{+}-\gamma\hat{Z}_{w}\right)\left(\hat{Z}_{+}^T+\gamma\hat{Z}_{\Delta}^T\right)\left(\left(\hat{G}_++\gamma\hat{G}_\Delta \right)^{-1}\right)^T|\mathcal{X}_{\tau}\right]    
\\
=&\left(\hat{G}_++\gamma\hat{G}_\Delta \right)^{-1}\E\left[ \sum_{i=0}^k\frac{2}{\textbf{n}_\tau(i)}(1+\gamma)w_i(\mathbb{X}^{A_i})^T\mathbb{Y}^{A_i}	\sum_{i=0}^k\frac{2}{\textbf{n}_\tau(i)}(1+\gamma)w_i(\mathbb{Y}^{A_i})^T\mathbb{X}^{A_i}\|\mathcal{X}_{\tau}\right]\left(\left(\hat{G}_++\gamma\hat{G}_\Delta \right)^{-1}\right)^T  
\\
=&\left(\hat{G}_++\gamma\hat{G}_\Delta \right)^{-1}\sum_{i,j=1}^k\frac{4}{\textbf{n}_\tau(i)\textbf{n}_\tau(j)}(1+\gamma)w_i(1+\gamma)w_j(\mathbb{X}^{A_i})^T\E\left[\mathbb{Y}^{A_i} (\mathbb{Y}^{A_j})^T \|\mathcal{X}_{\tau} \right]\mathbb{X}^{A_j}\left(\left(\hat{G}_++\gamma\hat{G}_\Delta \right)^{-1}\right)^T  
\end{align*}
\normalsize
Using the same notation is in the proof of Theorem \ref{aopt}, then for $i\not=j$, $\E\left[\mathbb{Y}^{A_i} (\mathbb{Y}^{A_j})^T \|\mathcal{X}_{\tau} \right]=\tilde{Y}_i(\tilde{Y}_j)^T$ while if $i=j$ then analogous calculations as in that proof yields,
$$\E\left[\mathbb{Y}^{A_i} (\mathbb{Y}^{A_i})^T \|\mathcal{X}_{\tau} \right]=\tilde{Y}_i(\tilde{Y}_i)^T+D^i.$$
\\
For the second term of $\mathbb{C}_{\tau}$,
\small
\begin{align*}
&\E\left[ \hat{\beta}_\gamma\|\mathcal{X}_{\tau}\right]^T\E\left[ \hat{\beta}_\gamma\|\mathcal{X}_{\tau}\right]=\E\left[ \left(\hat{G}_++\gamma\hat{G}_\Delta \right)^{-1}\left(\hat{Z}_{+}+\gamma\hat{Z}_{\Delta}\right)\|\mathcal{X}_{\tau}\right]\E\left[ \left(\hat{Z}_{+}^T+\gamma\hat{Z}_{\Delta}^T\right)\left(\left(\hat{G}_++\gamma\hat{G}_\Delta \right)^{-1}\right)^T\|\mathcal{X}_{\tau}\right] 		
\\
&=\sum_{i,j=1}^k\frac{4(1+\gamma)w_i(1+\gamma)w_j}{\textbf{n}_\tau(i)\textbf{n}_\tau(j)}\E\left[ \left(\hat{G}_++\gamma\hat{G}_\Delta \right)^{-1}(\mathbb{X}^{A_i})^T\mathbb{Y}^{A_i}\|\mathcal{X}_{\tau}\right]
\E\left[ \left(\mathbb{Y}^{A_j}\right)^T\mathbb{X}^{A_j} \left(\left(\hat{G}_++\gamma\hat{G}_\Delta \right)^{-1}\right)^T\|\mathcal{X}_{\tau}\right] 		
\\
&=\sum_{i,j=1}^k\frac{4(1+\gamma)w_i(1+\gamma)w_j}{\textbf{n}_\tau(i)\textbf{n}_\tau(j)}\left(\hat{G}_++\gamma\hat{G}_\Delta \right)^{-1}(\mathbb{X}^{A_i})^T\tilde{Y}_i\tilde{Y}_j^T\mathbb{X}^{A_j}\left(\left(\hat{G}_++\gamma\hat{G}_\Delta \right)^{-1}\right)^T.
\end{align*}
\normalsize
We then readily see that this leads to
\begin{align*}
&\mathbb{C}_{\tau}=\sum_{i=0}^k\frac{4(1+\gamma)w_i^2}{\textbf{n}_\tau(i)^2}\left(\hat{G}_++\gamma\hat{G}_\Delta \right)^{-1}(\mathbb{X}^{A_i})^TD^i\mathbb{X}^{A_i}\left(\left(\hat{G}_++\gamma\hat{G}_\Delta \right)^{-1}\right)^T.
\end{align*}
By the Hadamard inequality
$$\left|\det\left(\mathbb{C}_{\tau}\right)\right|\le \prod_{c=1}^p \n (\mathbb{C}_{\tau})_c \n_{\mathit{l}^2}, $$
where $(\mathbb{C}_{\tau})_c$ denotes column number $c$ of $\mathbb{C}_{\tau}$. Since $\n . \n_{\mathit{l}^2} \le \n .\n_{\mathit{l}^1}$,
$$ \prod_{c=1}^p \n (\mathbb{C}_{\tau})_c \n_{\mathit{l}^2}\le \prod_{c=1}^p \n(\mathbb{C}_{\tau})_c \n_{\mathit{l}^1} \le\left(\max_{1\le c\le p}\n (\mathbb{C}_{\tau})_c \n_{\mathit{l}^1}\right)^p
=
\n \mathbb{C}_{\tau}\n_{1}^p, $$
where $\n . \n_{1}$ denotes the usual maximal column sum-norm for matrices. Next, using just the triangle inequality and sub-multiplicative property of the matrix norm,
\begin{align*}
&\n \mathbb{C}_{\tau} \n_{1}\le  
\sum_{i=0}^k\frac{4(1+\gamma)w_i^2}{\textbf{n}_\tau(i)^2} \n \left(\hat{G}_++\gamma\hat{G}_\Delta \right)^{-1}(\mathbb{X}^{A_i})^TD^i\mathbb{X}^{A_i}\left(\left(\hat{G}_++\gamma\hat{G}_\Delta \right)^{-1}\right)^T \n_{1}
\\
&\le\sum_{i=0}^k\frac{4(1+\gamma)w_i^2}{\textbf{n}_\tau(i)^2} \n \left(\hat{G}_++\gamma\hat{G}_\Delta \right)^{-1}\n_{1}^2 \n (\mathbb{X}^{A_i})^TD^i\mathbb{X}^{A_i} \n_1
\end{align*}
The entries of the matrix $M^i(\textbf{n}_{\tau})=\frac{1}{\textbf{n}_{\tau}(i)}(\mathbb{X}^{A_i}(\textbf{n}_{\tau}))^TD^i(\textbf{n}_{\tau})\mathbb{X}^{A_i}(\textbf{n}_{\tau})$ are given by 
$$M^i_{(r,s)}(\textbf{n}_{\tau})=\frac{1}{\textbf{n}_{\tau}(i)}\sum_{u=1}^{\textbf{n}_{\tau}}X^{A_i}_u(r)X^{A_i}_u(s)Var\left(Y^{A_i}\| X^{A_i}=X^{A_i}_u\right)$$
and since $\{X^{A_i}_u(r)X^{A_i}_u(s)D_{(u,u)}\}_u$ are i.i.d. for every $r$ and $s$, it follows from the law of large numbers (recall that $\tau$ is short-hand notation for $\tau$) that
$$\lim_{m\to\infty}M^i_{(r,s)}(\textbf{n}_{\tau})=\E\left[X^{A_i}(r)X^{A_i}(s)Var\left(Y^{A_i}\| X^{A_i}\right)\right] \hspace{3mm} a.s..$$
This implies that $\{\n M^i (\textbf{n}_{\tau})\n_1\}_m$ is stochastically bounded. We have,
\begin{align*}
&\left|\det\left(\mathbb{C}_{\tau}\right)\right| 
\le
\left(\sum_{i=0}^k\frac{4(1+\gamma)w_i^2}{\textbf{n}_\tau(i)} \n \left(\hat{G}_++\gamma\hat{G}_\Delta \right)^{-1}\n_{1}^2 \n M^i(\textbf{n}_{\tau}) \n_1\right)^p
\le
\\
&\left(4(1+\gamma)\right)^p\left(\sum_{i=0}^k\frac{1}{\textbf{n}_\tau(i)} \left(\n \left(\hat{G}_++\gamma\hat{G}_\Delta -\left(G_++\gamma G_\Delta\right)\right)^{-1}\n_{1}+\n \left(G_++\gamma G_\Delta \right)^{-1}\n_{1} \right)^2 \n M^i(\textbf{n}_{\tau}) \n_1\right)^p.
\end{align*}
By the law of large numbers $\n \left(\hat{G}_++\gamma\hat{G}_\Delta -\left(G_++\gamma G_\Delta\right)\right)^{-1}\n_{1}$ is bounded in probability so by letting
$$C_m=\left(4(1+\gamma)\right)^p\max_{1\le i\le k} \left(\n \left(\hat{G}_++\gamma\hat{G}_\Delta -\left(G_++\gamma G_\Delta\right)\right)^{-1}\n_{1}+\n \left(G_++\gamma G_\Delta \right)^{-1}\n_{1} \right)^{2p} \n M^i(\textbf{n}_{\tau}) \n_1^p,$$
the result follows.
\end{proof}

\section{Discussion}
\label{sec:conclusion} 
We have seen how in a multiple environment setting the SEM of \eqref{SEMA} and \eqref{SEMO} is connected to the worst risk decomposition seen in Proposition \ref{suppropfixedw} (analogously to the case when there is only one shifted environment and an observational environment and a linear SEM). As seen this worst risk decomposition is quite universal. We then defined the risk minimizer as the unique $\arg\min$ solution for this worst risk. We proved several types of bounds for the plug-in estimator of this minimizer and the plug-in estimator for the worst risk. Further study of the implications for practical applications of the quite general SEM proposed here would certainly be interesting.

\section{Additional proofs}
\subsection{Proof of part 2 and 3 of Theorem 1}
We state the following elementary facts as a lemma to ease our exposition for the proof of Theorem1.
\begin{lemma}
Given two square matrices $C_1,C_2$ of the same dimension and any matrix norm $\n .\n$ we have that if $C_1$ is invertible and $\n C_1-C_2 \n< \frac{1}{\n C_1^{-1}\n}$ then $C_2$ is also invertible. Moreover, in this case
\begin{align}\label{invineq}
\n C_1^{-1}-C_2^{-1} \n\le \n C_1^{-1} \n \frac{\n I-C_1^{-1}C_2 \n}{1-\n I-C_1^{-1}C_2 \n}
\end{align}
\end{lemma}
\begin{proof} The proof easily follows by noting that $C_1^{-1}-C_2^{-1}=(I-C_2^{-1}C_1)C_1^{-1}$ and 
$$I-C_2^{-1}C_1=(I-C_1^{-1}C_2)(I-C_2^{-1}C_1)-(I-C_1^{-1}C_2).$$
\end{proof}
\begin{proof}[Proof of the second and third claim of Theorem 1]
We start with the second claim of the theorem.
\\
\textbf{Step 1}: Locally linearise the inverse in the estimator\\
Note first that 
\begin{align*}
&\n I-\left(G_++\gamma G_\Delta\right)^{-1}\left(\mathbb{G}_+(\textbf{n})+\gamma\mathbb{G}_\Delta(\textbf{n})\right)\n_{\mathit{l}^2}= \n \left(G_++\gamma G_\Delta\right)^{-1}\left( G_++\gamma G_\Delta-\left(\mathbb{G}_+(\textbf{n})+\gamma\mathbb{G}_\Delta(\textbf{n})\right)\right)\n_{\mathit{l}^2}\le 
\\
&\n \left(G_++\gamma G_\Delta\right)^{-1} \n_{\mathit{l}^2}\n G_++\gamma G_\Delta-\left(\mathbb{G}_+(\textbf{n})+\gamma\mathbb{G}_\Delta(\textbf{n})\right)\n_{\mathit{l}^2},
\end{align*}
where we used the sub-multiplicative property of the (induced) matrix norm in the last step.
By combining \eqref{invineq}, the above inequality as well as the inequality $\frac{x}{1-x}\le (1+\delta)x$ valid for $0\le x\le\frac{\delta}{1+\delta}$ and any $1>\delta>0$ we get that
\begin{align}\label{Ginvineq}
&\n \left(G_++\gamma G_\Delta\right)^{-1}-\left(\mathbb{G}_+(\textbf{n})+\gamma\mathbb{G}_\Delta(\textbf{n})\right)^{-1} \n_{\mathit{l}^2}
\le\nonumber
\\
&(1+\delta)\n \left(G_++\gamma G_\Delta\right)^{-1} \n_{\mathit{l}^2}\n I-\left(G_++\gamma G_\Delta\right)^{-1}\left(\hat{G}_+(\textbf{n})+\gamma\hat{G}_\Delta(\textbf{n})\right)\n_{\mathit{l}^2}
\le\nonumber
\\
&(1+\delta) \n \left(G_++\gamma G_\Delta\right)^{-1} \n_{\mathit{l}^2}^2\n G_++\gamma G_\Delta-\left(\hat{G}_+(\textbf{n})+\gamma\hat{G}_\Delta(\textbf{n})\right)\n_{\mathit{l}^2},
\end{align}
whenever $\n G_++\gamma G_\Delta-\left(\hat{G}_+(\textbf{n})+\gamma\hat{G}_\Delta(\textbf{n})\right)\n_{\mathit{l}^2}< \delta\n \left(G_++\gamma G_\Delta\right)^{-1}\n_{\mathit{l}^2}$ where we again used the sub-multiplicative property of the matrix norm in the second inequality. Let 
$$D_{\textbf{n}}(\delta)=\left\{\n G_++\gamma G_\Delta-\left(\hat{G}_+(\textbf{n})+\gamma\hat{G}_\Delta(\textbf{n})\right)\n_{\mathit{l}^2}
<
\delta\n \left(G_++\gamma G_\Delta\right)^{-1}\n_{\mathit{l}^2} \right\},$$
then as we noted earlier $\hat{\beta}_\gamma(\textbf{n})=\tilde{\beta}_\gamma(\textbf{n})$ on $D_{\textbf{n}}(\delta)$. Moreover, utilizing inequality \eqref{Ginvineq}, on the subset $D_\textbf{n}(\delta)$ we obtain,

\begin{align}\label{localexpansion}
&\n \tilde{\beta}_\gamma(\textbf{n}) -  \beta_\gamma\n_{\mathit{l}^2}
=
\n \hat{\beta}_\gamma(\textbf{n}) -  \left(G_++\gamma G_\Delta\right)^{-1}\left( Z_++\gamma Z_\Delta \right)\n_{\mathit{l}^2}
 \nonumber     
\\
&\le\n \left(\hat{G}_+(\textbf{n})+\gamma\hat{G}_\Delta(\textbf{n})\right)^{-1}\left( \left( Z_++\gamma Z_\Delta \right)-\left(\hat{Z}_+(\textbf{n})+\gamma\hat{Z}_\Delta(\textbf{n})\right)\right)  \n_{\mathit{l}^2} \nonumber
\\
&+\n \left(\left(\hat{G}_+(\textbf{n})+\gamma\hat{G}_\Delta(\textbf{n})\right)^{-1}- \left(G_++\gamma G_\Delta\right)^{-1}\right) \left( Z_++\gamma Z_\Delta \right)\n_{\mathit{l}^2}\nonumber
\\
&\le\n \left(\hat{G}_+(\textbf{n})+\gamma\hat{G}_\Delta(\textbf{n})\right)^{-1} \n_{\mathit{l}^2} \n \left( Z_++\gamma Z_\Delta \right)-\left(\hat{Z}_+(\textbf{n})+\gamma\hat{Z}_\Delta(\textbf{n})\right)\n_{\mathit{l}^2}\nonumber
\\
&+\n \left( Z_++\gamma Z_\Delta \right)\n_{\mathit{l}^2} \n \left(\hat{G}_+(\textbf{n})+\gamma\hat{G}_\Delta(\textbf{n})\right)^{-1}- \left(G_++\gamma G_\Delta\right)^{-1}\n_{\mathit{l}^2} \nonumber
\\
&\le\left(\n \left(\hat{G}_+(\textbf{n})+\gamma\hat{G}_\Delta(\textbf{n})\right)^{-1} 
-
\left(G_++\gamma G_\Delta\right)^{-1} \n_{\mathit{l}^2}+\n\left(G_++\gamma G_\Delta\right)^{-1} \n_{\mathit{l}^2} \right)
\n \left( Z_++\gamma Z_\Delta \right)-\left(\hat{Z}_+(\textbf{n})+\gamma\hat{Z}_\Delta(\textbf{n})\right)\n_{\mathit{l}^2}    
\nonumber
\\
&+(1+\delta)\n \left( Z_++\gamma Z_\Delta \right)\n_{\mathit{l}^2}
\n \left(G_++\gamma G_\Delta\right)^{-1} \n_{\mathit{l}^2}^2
\n G_++\gamma G_\Delta-\left(\hat{G}_+(\textbf{n})+\gamma\hat{G}_\Delta(\textbf{n})\right)\n_{\mathit{l}^2}\nonumber
\\
&\le\n\left(G_++\gamma G_\Delta\right)^{-1} \n_{\mathit{l}^2}
\n \left( Z_++\gamma Z_\Delta \right)-\left(\hat{Z}_+(\textbf{n})+\gamma\hat{Z}_\Delta(\textbf{n})\right)\n_{\mathit{l}^2}\nonumber
\\
&+(1+\delta)
\n \left(G_++\gamma G_\Delta\right)^{-1} \n_{\mathit{l}^2}^2
\n G_++\gamma G_\Delta-\left(\hat{G}_+(\textbf{n})+\gamma\hat{G}_\Delta(\textbf{n})\right)\n_{\mathit{l}^2}
\big(\n\left( Z_++\gamma Z_\Delta \right)\n_{\mathit{l}^2}
\nonumber
\\
&+
\n \left( Z_++\gamma Z_\Delta \right)-\left(\hat{Z}_+(\textbf{n})+\gamma\hat{Z}_\Delta(\textbf{n})\right)\n_{\mathit{l}^2}\big).
\end{align}
Thus, by taking the measure of the left-most and right-most sides above,
\begin{align}\label{used}
&\P\left(\n \tilde{\beta}_\gamma(\textbf{n}) - \beta_\gamma\n_{\mathit{l}^2}\ge c\right)
\le\P\left(\left\{\n \tilde{\beta}_\gamma(\textbf{n}) -  \left(G_++\gamma G_\Delta\right)^{-1}\left( Z_++\gamma Z_\Delta \right)\n_{\mathit{l}^2}\ge c\right\}\cap D_{\textbf{n}}(\delta) \right)
+
\P\left(D_{\textbf{n}}(\delta)^c\right)\nonumber
\\
&\le\P\left(D_{\textbf{n}}(\delta)^c\right)+\P\left(\n\left(G_++\gamma G_\Delta\right)^{-1} \n_{\mathit{l}^2}
\n \left( Z_++\gamma Z_\Delta \right)-\left(\hat{Z}_+(\textbf{n})+\gamma\hat{Z}_\Delta(\textbf{n})\right)\n_{\mathit{l}^2}\ge c/3 \right)\nonumber
\\
&+\P\left((1+\delta)
\n \left(G_++\gamma G_\Delta\right)^{-1} \n_{\mathit{l}^2}^2
\n G_++\gamma G_\Delta-\left(\hat{G}_+(\textbf{n})+\gamma\hat{G}_\Delta(\textbf{n})\right)\n_{\mathit{l}^2}
\n\left( Z_++\gamma Z_\Delta \right)\n_{\mathit{l}^2}\ge c/3\right)\nonumber
\\
&+\P\left((1+\delta)
\left\{\n \left(G_++\gamma G_\Delta\right)^{-1} \n_{\mathit{l}^2}^2
\n G_++\gamma G_\Delta-\left(\hat{G}_+(\textbf{n})+\gamma\hat{G}_\Delta(\textbf{n})\right)\n_{\mathit{l}^2}\n \left( Z_++\gamma Z_\Delta \right)-\left(\hat{Z}_+(\textbf{n})+\gamma\hat{Z}_\Delta(\textbf{n})\right)\n_{\mathit{l}^2}
\ge c/3\right\}
\right.\nonumber
\\
&\left.\bigcap D_\textbf{n}(\delta)\right).
\end{align}
\textbf{Step 2}: Split the Grammian matrices into independent increments with zero mean\\
For each $i\in \{1,...,k\}$ and $j=1,...,n_{A_i}$ we define the following matrices entry-wise for $u,v\in\{1,...,p\}$ by
$$W(j)_{u,v}^{n_{A_i},A_i}=\frac{(1+\gamma) w_i^2}{n_{A_i}}\mathbb{X}^{n_{A_i},A_i}_{u,j}\mathbb{X}^{n_{A_i},A_i}_{v,j}-\frac{1}{n_{A_i}}\E\left[(1+\gamma) w_i^2(X^{A_i})^TX^{A_i}\right]_{u,v},$$
and for $i=0$,
$$W(j)_{u,v}^{n_{A_i},A_i}=\frac{(1-\gamma)}{n_{A_i}}\mathbb{X}^{n_{A_i},A_i}_{u,j}\mathbb{X}^{n_{A_i},A_i}_{v,j}-\frac{1}{n_{A_i}}\E\left[(1-\gamma) (X^{A_i})^TX^{A_i}\right]_{u,v}.$$
Then $\{W(j)^{n_{A_i},A_i}\}_{j=1}^{n_{A_i}}$ are i.i.d., $\E\left[W(j)^{n_{A_i}}\right]=0$ and 
$$\sum_{i=0}^k\sum_{j=1}^{n_{A_i}}W(j)^{n_{A_i},A_i}=\hat{G}_+(\textbf{n})+\gamma\hat{G}_\Delta(\textbf{n})- G_++\gamma G_\Delta.$$ 
Let $\textbf{R}=(R_1,...,R_k)$ be a vector of length $k$ with positive real entries. 
\\
\textbf{Step 3}: Split the above increments into bounded parts and tail parts
\\
If we now also define 
$$W^{A_i}_{R_i}(j)_{u,v}=W(j)_{u,v}^{n_{A_i},A_i}1_{|\mathbb{X}^{n_{A_i},A_i}_{u,j}|\vee|\mathbb{X}^{n_{A_i},A_i}_{v,j}|\le R_i}$$ and 
$$E_\textbf{R}=\bigcap_{i=0}^k\bigcap_{j=1}^{n_{A_i}}\left\{W^{A_i}_{R_i}(j)_{u,v}=W(j)_{u,v}^{n_{A_i},A_i},\forall u,v\in \{1,...,p\}\right\}$$
then $E_\textbf{R}=\bigcap_{i=0}^k\bigcap_{u=1}^p\bigcap_{j=1}^{n_{A_i}}\left\{|\mathbb{X}^{n_{A_i},A_i}_{u,j}|\le R_i,\right\}$ and $\P\left(E_\textbf{R}\right)=\prod_{i=0}^k\left(F_{\left|X^{A_i}(1)\right|,...,\left|X^{A_i}(p)\right|}(R_i,...,R_i)\right)^{n_{A_i}}$. We also define the set $E_i'(R_i)=\left\{|X^{A_i}(u)|\le R_i,1\le u \le p\right\}$. We shall use the following Jensen-type inequality,
\begin{align}\label{jensen}
&\n \E\left[ X \right] \n_{\mathit{l}^2} 
=
\sup_{\n x \n_{\mathit{l}^2}=1} \n \E\left[ X x\right] \n_{\mathit{l}^2}
=
\sup_{\n x \n_{\mathit{l}^2}=1}\sup_{\n y \n_{\mathit{l}^2}=1}\left| <\E\left[Xx\right],y> \right|=\nonumber
\\
&\sup_{\n x \n_{\mathit{l}^2}=1}\sup_{\n y \n_{\mathit{l}^2}=1}\left| \E\left[<Xx,y>\right] \right|
\le
\sup_{\n x \n_{\mathit{l}^2}=1}\sup_{\n y \n_{\mathit{l}^2}=1}\E\left[\left|<Xx,y>\right|\right]
\le 
\E\left[\sup_{\n x \n_{\mathit{l}^2}=1}\sup_{\n y \n_{\mathit{l}^2}=1}\left|<Xx,y>\right|\right]=\nonumber
\\
&\E\left[\sup_{\n x \n_{\mathit{l}^2}=1} \n Xx\n_{\mathit{l}^2} \right]
=
\E\left[\n X\n_{\mathit{l}^2}\right].
\end{align}
On the set $E_{\textbf{R}}$,
\begin{align*}
&\n \hat{G}_+(\textbf{n})+\gamma\hat{G}_\Delta(\textbf{n})- G_++\gamma G_\Delta \n_{\mathit{l}^2} =
\n \sum_{i=0}^k\sum_{j=1}^n W^{A_i}_{R_i}(j)\n_{\mathit{l}^2}\le 
\sum_{i=0}^k \n \sum_{j=1}^{n_{A_i}} W^{A_i}_{R_i}(j) \n_{\mathit{l}^2} \le
\\
&\n \sum_{j=1}^{n_{A_0}} W^{A_0}_R(j)\n_{\mathit{l}^2}     +|1-\gamma|\n\E\left[(X^{A_0})^TX^{A_0}1_{(E_0')^c}\right]\n_{\mathit{l}^2}+
\sum_{i=1}^k \n \sum_{j=1}^{n_{A_i}} W^{A_i}_R(j)\n_{\mathit{l}^2} +\sum_{i=1}^k(1+\gamma )w_i^2\n\E\left[(X^{A_i})^TX^{A_i}1_{(E_i')^c}\right]\n_{\mathit{l}^2}
\le
\\
&\sum_{i=0}^k \n \sum_{j=1}^{n_{A_i}} W^{A_i}_R(j)\n_{\mathit{l}^2} +
\sum_{i=0}^k(1+\gamma )\n\E\left[(X^{A_i})^TX^{A_i}1_{(E_i')^c}\right]\n_{\mathit{l}^2}
\le
\sum_{i=0}^k \n \sum_{j=1}^{n_{A_i}} W^{A_i}_R(j)\n_{\mathit{l}^2} +
\sum_{i=0}^k(1+\gamma )\E\left[\n(X^{A_i})^TX^{A_i}1_{(E_i')^c}\n_{\mathit{l}^2} \right],
\end{align*}
where we used the \eqref{jensen} in the final step. 
\\
\textbf{Step 4}: Apply the matrix Hoeffding inequality to the bounded parts of the increments\\
Recall that the maximum eigenvalue for a symmetric matrix is equivalent to the induced $\mathit{l}^2$-norm of that matrix. Moreover, an upper bound for the maximum eigenvalue of $W^{A_i}_R(j)$ is given by its maximum absolute row sum, which is bounded above by $\frac{pR_i^2}{n_{A_i}}$. If we let $M_i=\frac{pR_i^2}{n_{A_i}}I_{p\times p}$ then obviously its minimal eigenvalue is an upper bound for $W^{A_i}_{R_i}(j)$ for all $j$ and thus $W^{A_i}_{R_i}(j)\preccurlyeq M_i$ which implies $W^{A_i}_{R_i}(j)^2\preccurlyeq A_i^2$. Since $ \n M_i^2 \n=\frac{p^2R_i^4}{n_{A_i}^2}$ and therefore $\n \sum_{j=1}^{n_{A_i}} M_i^2\n =\sum_{j=1}^{n_{A_i}} \n M_i^2\n=\frac{p^2R_i^4}{n_{A_i}}$ (regardless of the chosen matrix norm), it follows from the matrix Hoeffding inequality that
\begin{align*}
&\P\left( \n \sum_{j=1}^{n_{A_i}} W^{A_i}_R(j)\n_{\mathit{l}^2} \ge t\right)\le 
\P\left( \lambda_{\max}\left(\sum_{j=1}^n W^{A_i}_R(j)\right) \ge t\right)\le pe^{-\frac{t^2}{ 2p^2R_i^4}n_{A_i}}.
\end{align*}
\textbf{Step 5}: Sum up with the tail parts of the increments\\
Let $h_i(R_i)=\left(\sum_{u=1}^p\sqrt{\E\left[ |X^{A_i}(u)|^21_{(E_i'(R_i))^c}\right]}\right)^2$. Note that by the Cauchy-Schwartz inequality,
\begin{align*}
&\E\left[\n(X^{A_i})^TX^{A_i}1_{(E_i')^c}\n_{\mathit{l}^2}\right]\le \E\left[\n(X^{A_i})^TX^{A_i}1_{(E_i')^c}\n_{\mathit{l}^1}\right]\le
\sum_{u=1}^p\sum_{v=1}^p\E\left[ |X^{A_i}(u)||X^{A_i}(v)|1_{(E_i')^c}\right]
\\
&\le\sum_{u=1}^p\sum_{v=1}^p\sqrt{\E\left[ |X^{A_i}(u)|^21_{(E_i')^c}\right]}\sqrt{\E\left[ |X^{A_i}(v)|^21_{(E_i')^c}\right]}
=\left(\sum_{u=1}^p\sqrt{\E\left[ |X^{A_i}(u)|^21_{(E_i')^c}\right]}\right)^2=h_i(R_i).
\end{align*}
Using the above notation,
\begin{align}\label{Gdiff}
&\P\left( \n \hat{G}_+(\textbf{n})+\gamma\hat{G}_\Delta(\textbf{n})- G_++\gamma G_\Delta \n_{\mathit{l}^2}\ge c' \right)\le \P\left( \left\{\n \hat{G}_+(\textbf{n})+\gamma\hat{G}_\Delta(\textbf{n})- G_++\gamma G_\Delta \n_{\mathit{l}^2}\ge c'\right\}\cap E_{\textbf{R}} \right)+\P\left(E_{\textbf{R}}^c\right)\le\nonumber
\\
&\P\left( \sum_{i=0}^k \left(\n \sum_{j=1}^{n_{A_i}} W^{A_i}_R(j)\n_{\mathit{l}^2}
+
(1+\gamma )\E\left[\n(X^{A_i})^TX^{A_i}1_{(E_i')^c}\n_{\mathit{l}^2}\right]\right)
\ge c'  \right)+\P\left(E_{\textbf{R}}^c\right)
\le\nonumber
\\
&\P\left(\sum_{i=0}^k \n \sum_{j=1}^{n_{A_i}} W^{A_i}_R(j)\n_{\mathit{l}^2}
\ge c'-(1+\gamma)h_i(R_i)  \right)+\P\left(E_{\textbf{R}}^c\right)
\le\nonumber
\\
&\sum_{i=0}^k\P\left(  \n \sum_{j=1}^{n_{A_i}} W^{A_i}_R(j)\n_{\mathit{l}^2} 
\ge 
\frac{c'-(1+\gamma )h_i(R_i)}{k+1} \right)+\P\left(E_{\textbf{R}}^c\right)\le
\P\left( \gamma_{\max}\left(\sum_{j=1}^{n_{A_0}} W^{A_0}_R(j)\right) \ge \frac{c'-(1+\gamma)h_i(R_i)}{(k+1)}\right)+\nonumber
\\
&\sum_{i=1}^k\P\left( \lambda_{\max}\left(\sum_{j=1}^{n_{A_i}} W^{A_i}_R(j)\right) \ge \frac{c'-(1+\gamma)h_i(R_i)}{k+1}\right)
+
\P\left(E_{\textbf{R}}^c\right)
\le\nonumber
\\ 
&p\sum_{i=0}^ke^{-\frac{\left(\left(c'-(1+\gamma)h_i(\textbf{R})\right)_+\right)^2}{2p^2(k+1)^2 R_i^4}n_{A_i}}+1-\prod_{i=0}^k\left(F_{\left|X^{A_i}(1)\right|,...,\left|X^{A_i}(p)\right|}(R_i,...,R_i)\right)^{n_{A_i}}.
\end{align}
If we now choose $R_i=n_{A_i}^\alpha$ for some $0<\alpha<\frac 14$ then
\begin{align}\label{PDnc}
&\P\left( \n \hat{G}_+(\textbf{n})+\gamma\hat{G}_\Delta(\textbf{n})- G_++\gamma G_\Delta \n_{\mathit{l}^2}\ge c' \right) 
\le  
p\sum_{i=0}^ke^{-\frac{\left(\left(c'-(1+\gamma)h_i(n_{A_i}^\alpha)\right)_+\right)^2}{ 2p^2(k+1)^2}n_{A_i}^{1-4\alpha}}\nonumber
\\
&+1-\prod_{i=0}^k\left(F_{\left|X^{A_i}(1)\right|,...,\left|X^{A_i}(p)\right|}(n_{A_i}^\alpha,...,n_{A_i}^\alpha)\right)^{n_{A_i}}.
\end{align}
\textbf{Step 6}: Compute the corresponding bounds for the $Z$-part of the estimator\\
Similarly to $W$, for each $i\in \{1,...,k\}$, $j=1,...,n$ and $u\in\{1,...,p\}$ we define the vectors with entries
$$V(j)_{u}^{n_{A_i},A_i}=\frac{(1+\gamma )w_i^2}{n_{A_i}}\mathbb{X}^{n_{A_i},A_i}_{u,j}\mathbb{Y}^{n_{A_i},A_i}_{j}-\frac{1}{n_{A_i}}\E\left[(1+\gamma )w_i^2(X^{A_i})^TY^{A_i}\right]_{u},$$
and
$$V(j)_{u}^{n_{A_0},A_0}=\frac{(1-\gamma)}{n_{A_0}}\mathbb{X}^{n_{A_0},A_0}_{u,j}\mathbb{Y}^{n_{A_0},A_0}_{j}-\frac{1}{n_{A_0}}\E\left[(1-\gamma)(X^{A_0})^TY^{A_0}\right]_{u},$$
for $1\le u\le p$. As before $\{V(j)^{n_{A_i},A_i}\}_{j=1}^n$ are i.i.d., $\E\left[V(j)^{n_{A_i}}\right]=0$ and 
$$\sum_{i=0}^k\sum_{j=1}^nV(j)^{n_{A_i},A_i}=\hat{Z}_+(\textbf{n})+\gamma\hat{Z}_\Delta(\textbf{n})- \left( Z_++\gamma Z_\Delta \right).$$ 
We define $V^{A_i}_{R_i}(j)_{u}=V(j)_{u}^{n_{A_i},A_i}1_{|\mathbb{X}^{n_{A_i},A_i}_{u,j}|\vee|\mathbb{Y}^{n_{A_i},A_i}_{j}|\le R_i}=V(j)_{u}^{n_{A_i},A_i}1_{|\mathbb{X}^{n_{A_i},A_i}_{u,j}|\le R_i}1_{|\mathbb{Y}^{n_{A_i},A_i}_{j}|\le R_i} $ and the set 
$$H_\textbf{R}=\bigcap_{i=0}^k\bigcap_{j=1}^{n_{A_i}}\left\{V^{A_i}_{R_i}(j)_{u}=V(j)_{u}^{n_{A_i},A_i},\forall u\in \{1,...,p\}\right\},$$
and 
$$\P\left(H_\textbf{R}\right)=\prod_{i=0}^k\left(F_{\left|X^{A_i}(1)\right|,...,\left|X^{A_i}(p)\right|,|Y^{A_i}|}(R_i,...,R_i)\right)^{n_{A_i}},$$
with $p+1$ entries of $R_i$. Let $H_i'=\left\{|X^{A_i}(u)|\le R_i,1\le u \le p,|Y^{A_i}|\le R_i\right\}$ On the set $H_\textbf{R}$,
\begin{align*}
&\n \hat{Z}_+(\textbf{n})+\gamma\hat{Z}_\Delta(\textbf{n})- \left( Z_++\gamma Z_\Delta \right) \n_{\mathit{l}^2}\le 
\sum_{i=0}^k \n \sum_{j=1}^{n_{A_i}} V^{A_i}_{R_i}(j)\n_{\mathit{l}^2} +
\sum_{i=0}^k(1+\gamma )w_i^2\E\left[\n X^{A_i}Y^{A_i}1_{(H_i')^c}\n_{\mathit{l}^2}\right].
\end{align*}
Since $\n .\n_{\mathit{l}^2} \le \n .\n_{\mathit{l}^1}$
\begin{align*}
&
\sum_{i=0}^k \n \sum_{j=1}^{n_{A_i}} V^{A_i}_{R_i}(j)\n_{\mathit{l}^2} 
\le
\sum_{i=0}^k\sum_{u=1}^p \left| \sum_{j=1}^{n_{A_i}} V^{A_i}_{R_i}(j)_u\right |.
\end{align*}
By the (scalar) Hoeffding inequality,
\begin{align*}
&\P\left(
\sum_{i=0}^k \n \sum_{j=1}^{n_{A_i}} V^{A_i}_{R_i}(j)\n_{\mathit{l}^2} \ge c' \right)
\le
\sum_{i=0}^k\sum_{u=1}^p \P\left(\left| \sum_{j=1}^{n_{A_i}} V^{A_i}_{R_i}(j)_u\right| \ge \frac{c'}{(k+1)p}\right)\le
\sum_{i=0}^k 2p\exp\left( -\frac{(c')^2}{2R_i^4(k+1)^2p^2}n_{A_i}\right).
\end{align*}
Let $g_i(R_i)=\sum_{u=1}^p\E\left[ |X^{A_i}(u)||Y^{A_i}|1_{(H_i')^c}\right]$, then
\begin{align*}
&\E\left[\n X^{A_i}(u)Y^{A_i}1_{(H_i')^c}\n_{\mathit{l}^2}\right]
\le
\E\left[\n X^{A_i}Y^{A_i}1_{(H_i')^c}\n_{\mathit{l}^1}\right]=
\sum_{u=1}^p\E\left[ |X^{A_i}(u)||Y^{A_i}|1_{(H_i')^c}\right]=g_i(R_i).
\end{align*}
Now
\begin{align*}
&\P\left( \n \hat{Z}_+(\textbf{n})+\gamma\hat{Z}_\Delta(\textbf{n})- \left( Z_++\gamma Z_\Delta \right) \n_{\mathit{l}^2}\ge c' \right)
\le
\P\left( \left\{\n \hat{Z}_+(\textbf{n})+\gamma\hat{Z}_\Delta(\textbf{n})- \left( Z_++\gamma Z_\Delta \right) \n_{\mathit{l}^2}\ge c'\right\}\cap H_{\textbf{R}} \right)+\P\left(H_{\textbf{R}}^c\right)
\\
&\le\P\left(
\sum_{i=0}^k \left(\n \sum_{j=1}^{n_{A_i}} V^{A_i}_{R_i}(j)\n_{\mathit{l}^2}  + (1+\gamma )\E\left[\n X^{A_i}Y^{A_i}1_{(H_i')^c}\n_{\mathit{l}^2}\right]\right)\ge c'
 \right)
+\P\left(H_{\textbf{R}}^c\right)
\\
&\le\P\left(\sum_{i=0}^k \n\sum_{j=1}^{n_{A_i}} V^{A_i}_{R_i}(j)\n_{\mathit{l}^2} \ge c'-(1+\gamma )g_i(R_i)\right)
+
\P\left(H_{\textbf{R}}^c\right)
\\ 
&\le 2p\sum_{i=0}^ke^{-\frac{\left(\left(c'-(1+\gamma )g_i(R_i)\right)_+\right)^2}{2p^2(k+1)^2 R_i^4}n_{A_i}}
+1-\prod_{i=0}^k\left(F_{\left|X^{A_i}(1)\right|,...,\left|X^{A_i}(p)\right|,|Y^{A_i}|}(R_i,...,R_i)\right)^{n_{A_i}}.
\end{align*}
and with $R_i$ chosen as before, then
\begin{align*}
&\P\left( \n \hat{Z}_+(\textbf{n})+\gamma\hat{Z}_\Delta(\textbf{n})- \left( Z_++\gamma Z_\Delta \right) \n_{\mathit{l}^2}\ge c' \right)
\le
2p\sum_{i=0}^ke^{-\frac{\left(\left(c'-(1+\gamma )g_i(n_{A_i}^\alpha\right)_+\right)^2}{2p^2(k+1)^2}n_{A_i}^{1-4\alpha}}
+
\\
&1-\prod_{i=0}^k\left(F_{\left|X^{A_i}(1)\right|,...,\left|X^{A_i}(p)\right|,|Y^{A_i}|}(n_{A_i}^\alpha,...,n_{A_i}^\alpha)\right)^{n_{A_i}}.
\end{align*}
\\
\textbf{Step 7}: Sum up all the terms\\
We can now estimate
\begin{align*}
&\P(D_{\textbf{n}}(\delta)^c)
=\P\left(\n G_++\gamma G_\Delta-( \hat{G}_+(\textbf{n})+\gamma\hat{G}_\Delta(\textbf{n}))\n_{\mathit{l}^2}\ge \delta\left(G_++\gamma G_\Delta\right)^{-1}\right)
\\
&\le
p\sum_{i=0}^ke^{-\frac{\left(\left(\delta\n \left(G_++\gamma G_\Delta\right)^{-1}\n_{\mathit{l}^2} -(1+\gamma )h_i(n_{A_i}^\alpha)\right)_+\right)^2}{ 2p^2(k+1)^2}n_{A_i}^{1-4\alpha}}
+1-\prod_{i=0}^k\left(F_{\left|X^{A_i}(1)\right|,...,\left|X^{A_i}(p)\right|}(n_{A_i}^\alpha,...,n_{A_i}^\alpha)\right)^{n_{A_i}}
\end{align*}
For the second term
\begin{align}\label{term2}
&\P\left(\n\left(G_++\gamma G_\Delta\right)^{-1} \n_{\mathit{l}^2}
\n \left( Z_++\gamma Z_\Delta \right)-\left(\hat{Z}_+(\textbf{n})+\gamma\hat{Z}_\Delta(\textbf{n})\right)\n_{\mathit{l}^2}\ge c/3 \right)
\le \nonumber
\\
&2p\sum_{i=0}^ke^{-\frac{\left(\frac{c}{3\n\left(G_++\gamma G_\Delta\right)^{-1} \n_{\mathit{l}^2}}- (1+\gamma )g_i(n_{A_i}^\alpha) \right)_+^2}{2p^2(k+1)^2}n_{A_i}^{1-4\alpha}}
+
1-\prod_{i=0}^k\left(F_{\left|X^{A_i}(1)\right|,...,\left|X^{A_i}(p)\right|,|Y^{A_i}|}(n_{A_i}^\alpha,...,n_{A_i}^\alpha)\right)^{n_{A_i}},
\end{align}
For the third term,
\begin{align}\label{term3}
&\P\left((1+\delta)
\n \left(G_++\gamma G_\Delta\right)^{-1} \n_{\mathit{l}^2}^2
\n G_++\gamma G_\Delta-\left(\hat{G}_+(\textbf{n})+\gamma\hat{G}_\Delta(\textbf{n})\right)\n_{\mathit{l}^2}
\n\left( Z_++\gamma Z_\Delta \right)\n_{\mathit{l}^2}\ge c/3\right)\le\nonumber
\\
&    p\sum_{i=0}^ke^{-\frac{\left( \frac{c}{3(1+\delta)
\n \left(G_++\gamma G_\Delta\right)^{-1} \n_{\mathit{l}^2}^2\n\left( Z_++\gamma Z_\Delta \right)\n_{\mathit{l}^2}}  -(1+\gamma )h_i(n_{A_i}^\alpha)\right)_+^2}{(k+1)^2 2p^2}n_{A_i}^{1-4\alpha}}
+1-\prod_{i=0}^k\left(F_{\left|X^{A_i}(1)\right|,...,\left|X^{A_i}(p)\right|}(n_{A_i}^\alpha,...,n_{A_i}^\alpha)\right)^{n_{A_i}}
\end{align}
For the final and fourth term
\begin{align}\label{term4}
&\P\left(\left\{(1+\delta)
\n \left(G_++\gamma G_\Delta\right)^{-1} \n_{\mathit{l}^2}^2
\n G_++\gamma G_\Delta-\left(\hat{G}_+(\textbf{n})+\gamma\hat{G}_\Delta(\textbf{n})\right)\n_{\mathit{l}^2}\n \left( Z_++\gamma Z_\Delta \right)-\left(\hat{Z}_+(\textbf{n})+\gamma\hat{Z}_\Delta(\textbf{n})\right)\n_{\mathit{l}^2}
\ge c/3\right\}\cap D_\textbf{n}(\delta)\right)\nonumber
\\
&\le
\P\left((1+\delta)
\n \left(G_++\gamma G_\Delta\right)^{-1} \n_{\mathit{l}^2}^3\delta
\n \left( Z_++\gamma Z_\Delta \right)-\left(\hat{Z}_+(\textbf{n})+\gamma\hat{Z}_\Delta(\textbf{n})\right)\n_{\mathit{l}^2}
\ge c/3\right)\nonumber
\\
&\le
2p\sum_{i=0}^ke^{-\frac{\left(\frac{c}{3(1+\delta)\n \left(G_++\gamma G_\Delta\right)^{-1}\n_{\mathit{l}^2}^3\delta}   -(1+\gamma )g_i(n_{A_i}^\alpha)\right)_+^2}
{(k+1)^22p^2}n_{A_i}^{1-4\alpha}}
+1-\prod_{i=0}^k\left(F_{\left|X^{A_i}(1)\right|,...,\left|X^{A_i}(p)\right|,|Y^{A_i}|}(n_{A_i}^\alpha,...,n_{A_i}^\alpha)\right)^{n_{A_i}}.
\end{align}
We now collect all the terms and use the fact that $F_{\left|X^{A_i}(1)\right|,...,\left|X^{A_i}(p)\right|}(n_{A_i}^\alpha,...,n_{A_i}^\alpha) 
\ge 
F_{\left|X^{A_i}(1)\right|,...,\left|X^{A_i}(p)\right|,|Y^{A_i}|}(n_{A_i}^\alpha,...,n_{A_i}^\alpha)$ to find that when $n_{A_i}\ge N_{A_i}$, 
\begin{align*}
&\P\left(\n \tilde{\beta}_\gamma(\textbf{n}) -  \beta_\gamma\n_{\mathit{l}^2}\ge c \right)
\\
&\le 6p\sum_{i=0}^ke^{-\frac{\left(\tilde{r}(c)-(1+\gamma)\left(g_i(n_{A_i}^\alpha)\vee h_i(n_{A_i}^\alpha)\right) \right)_+^2}{L} n_{A_i}^{1-4\alpha}}
+4\left(1-\prod_{i=0}^k\left(F_{\left|X^{A_i}(1)\right|,...,\left|X^{A_i}(p)\right|,|Y^{A_i}|}(n_{A_i}^\alpha,...,n_{A_i}^\alpha)\right)^{n_{A_i}}\right)
\end{align*}
where
$$\tilde{r}(c)= \left(\left(\delta\n \left(G_++\gamma G_\Delta\right)^{-1}\n_{\mathit{l}^2}\right)\wedge c\right)\left(\frac{1}{3\n \left(G_++\gamma G_\Delta\right)^{-1}\n_{\mathit{l}^2}(1+\delta)\left(1\vee
\n \left(G_++\gamma G_\Delta\right)^{-1} \n_{\mathit{l}^2}\vee\n\left( Z_++\gamma Z_\Delta \right)\n_{\mathit{l}^2}\right)}\right)$$
and
$$L=2p^2(k+1)^2.$$
We can now let 
\begin{align*}
&N_{A_i}=\min \left\{n\in\N:  (1+\gamma)g_i(n_{A_i}^\alpha)\vee h_i(n_{A_i}^\alpha)<\frac{\tilde{r}(c)}{2}\right\}
\end{align*}
and
note that since both $g$ and $h$ are non-increasing, $(1+\gamma)g_i(n_{A_i}^\alpha)\vee h_i(n_{A_i}^\alpha)<\frac{\tilde{r}(c)}{2}$ for all $n\ge N_{A_i}$. We can now let $r(c)=\frac{\tilde{r}(c)}{\sqrt{2L}}$ and with this definition
\begin{align*}
&\P\left(\n \tilde{\beta}_\gamma(\textbf{n}) -  \beta_\gamma\n_{\mathit{l}^2}\ge c \right)
\le
(4p+2)\sum_{i=0}^ke^{-r(c)^2 n_{A_i}^{1-4\alpha}}
+
5\left(1-\prod_{i=0}^k\left(F_{\left|X^{A_i}(1)\right|,...,\left|X^{A_i}(p)\right|,|Y^{A_i}|}(n_{A_i}^\alpha,...,n_{A_i}^\alpha)\right)^{n_{A_i}}\right),
\end{align*}
for all $n\ge N_{A_i}$.
\end{proof}
\begin{proof}[Proof of the third claim]
We now assume $M(\zeta)<\infty$ for some $\zeta>2$ and let $\sigma(\zeta):= (p+1)M(\zeta)$. This new assumption will make it possible to obtain an explicit bound for all $n_{A_i}\ge 1$. By the Markov inequality we have
\begin{align}\label{Markovg}
\P\left(|X^{A_i}(1)|\vee...\vee|X^{A_i}(p)|\vee|Y^{A_i}|\ge x\right)
&=
\P\left(|X^{A_i}(1)|^\zeta\vee...\vee|X^{A_i}(p)|^\zeta\vee|Y^{A_i}|^\zeta\ge x^\zeta\right)
\nonumber\\
&\le
\P\left(\sum_{u=1}^p|X^{A_i}(u)|^\zeta+|Y^{A_i}|^\zeta\ge x^\zeta\right) 
\nonumber\\
&\le
\frac{1}{x^\zeta}\left(\sum_{u=1}^p\E\left[|X^{A_i}(u)|^\zeta\right]+\E\left[|Y^{A_i}|^\zeta\right]\right)
\le
\frac{(p+1)M(\zeta)}{x^\zeta}.
\end{align}
\eqref{Markovg} now implies,
$$1-F_{|X^{A_i}(1)|,...,|X^{A_i}(p)|,|Y^{A_i}|}(x,...,x)\le \frac{\sigma(\zeta)}{x^\zeta},$$
and trivially also
$$1-F_{|X^{A_i}(1)|,...,|X^{A_i}(p)|}(x,...,x)\le 1-F_{|X^{A_i}(1)|,...,|X^{A_i}(p)|,|Y^{A_i}|}(x,...,x)\le \frac{\sigma(\zeta)}{x^\zeta}.$$
This leads to the following bound of $g_i(n_{A_i}^\alpha)$,
\begin{align}\label{gbound}
g_i(n_{A_i}^\alpha)
&=
\sum_{u=1}^p\E\left[\left|X^{A_i}(u)Y^{A_i}\right|1_{(H_i'(n_{A_i}^\alpha))^c}\right]
\le
\sum_{u=1}^p\E\left[\left|X^{A_i}(u)Y^{A_i}\right|^{\frac{\zeta}{2}}\right]^{\frac{2}{\zeta}}\P\left((H_i'(n_{A_i}^\alpha))^c\right)^{1-\frac{2}{\zeta}}
\nonumber
\\
&\le
\sum_{u=1}^p\E\left[\left|X^{A_i}(u)\right|^{\zeta}\right]^{\frac{1}{\zeta}}\E\left[\left|Y^{A_i}\right|^{\zeta}\right]^{\frac{1}{\zeta}}\P\left((H_i'(n_{A_i}^\alpha))^c\right)^{1-\frac{2}{\zeta}}\nonumber
\\
&\le
pM(\zeta)^{\frac{2}{\zeta}}\left(1-F_{\left|X^{A_i}(1)\right|,...,\left|X^{A_i}(1)\right|,\left|Y^{A_i}\right|}\left(n_{A_i}^\alpha,...,n_{A_i}^\alpha\right)\right)^{\frac{\zeta-2}{\zeta}}
\le
pM(\zeta)^{\frac{2}{\zeta}}\left(\frac{\sigma(\zeta)}{n_{A_i}^{\alpha\zeta}}\right)^{\frac{\zeta-2}{\zeta}},
\end{align}
where we used the Hölder inequality. We have a similar bound for $h_i(n_{A_i}^\alpha)$
\begin{align}\label{hboundHolder}
h_i(n_{A_i}^\alpha)
&=
\left(\sum_{u=1}^p\sqrt{\E\left[ |X^{A_i}(u)|^21_{(E_i'(n_{A_i}^\alpha))^c}\right]}\right)^2
\le
\left(\sum_{u=1}^p\sqrt{\E\left[ |X^{A_i}(u)|^\zeta\right]^{\frac{2}{\zeta}}\P\left((E_i'(n_{A_i}^\alpha))^c\right)^{1-\frac{2}{\zeta}}}\right)^2\nonumber
\\
&\le
\left(p\sqrt{M(\zeta)^{\frac{2}{\zeta}}\left(1-F_{\left|X^{A_i}(1)\right|,...,\left|X^{A_i}(1)\right|,\left|Y^{A_i}\right|}\left(n_{A_i}^\alpha,...,n_{A_i}^\alpha\right)\right)^{1-\frac{2}{\zeta}}}\right)^2
\le
p^2M(\zeta)^{\frac{2}{\zeta}}\left(\frac{\sigma(\zeta)}{n_{A_i}^{\alpha\zeta}}\right)^{1-\frac{2}{\zeta}}.
\end{align}
Define $N'_{A_i}=\min\left\{n_{A_i}\in\N: (1+\gamma)h_i(n_{A_i}^\alpha)<\frac{\delta\n \left(G_++\gamma G_\Delta\right)^{-1}\n_{\mathit{l}^2}}{2} \right\}$, then \eqref{hboundHolder} implies 
\begin{align}
N_{A_i}'\le \left(\frac{2(1+\gamma)p^2M(\zeta)^{\frac{2}{\zeta}}}{\delta\n \left(G_++\gamma G_\Delta\right)^{-1}\n_{\mathit{l}^2}}\right)^{\frac{1}{\alpha(\zeta-2)}}\sigma(\zeta)^{\frac{\zeta-2}{\zeta}}
\end{align}
With our new assumption we may now compute  new bounds for the four terms in \eqref{used}. For the first term,
\begin{align*}
&\P(D_{\textbf{n}}(\delta)^c)
=\P\left(\n G_++\gamma G_\Delta-( \hat{G}_+(\textbf{n})+\gamma\hat{G}_\Delta(\textbf{n}))\n_{\mathit{l}^2}\ge \delta\left(G_++\gamma G_\Delta\right)^{-1}\right)
\\
&\le
p\sum_{i=0}^ke^{-\frac{\left(\left(\delta\n \left(G_++\gamma G_\Delta\right)^{-1}\n_{\mathit{l}^2} -(1+\gamma )h_i(n_{A_i}^\alpha)\right)_+\right)^2}{ 2p^2(k+1)^2}n_{A_i}^{1-4\alpha}}
+1-\prod_{i=0}^k\left(F_{\left|X^{A_i}(1)\right|,...,\left|X^{A_i}(p)\right|}(n_{A_i}^\alpha,...,n_{A_i}^\alpha)\right)^{n_{A_i}}
\end{align*}
The first term on the right-hand side can be bounded as follows, using the inequality $(x-y)^2\ge \frac{x^2}{2}$, valid for $2y\le x$,
\begin{align}\label{Concpexp01}
&p\sum_{i=0}^ke^{-\frac{\left(\left(\delta\n \left(G_++\gamma G_\Delta\right)^{-1}\n_{\mathit{l}^2} -(1+\gamma )h_i(n_{A_i}^\alpha)\right)_+\right)^2}{2p^2(k+1)^2}n_{A_i}^{1-4\alpha}}
\le
p\sum_{i=0}^k1_{n_{A_i}<N_{A_i}'}+p\sum_{i=0}^k1_{n_{A_i}\ge N_{A_i}'}e^{-\frac{\delta^2\n \left(G_++\gamma G_\Delta\right)^{-1}\n_{\mathit{l}^2}^2}{4p^2(k+1)^2}n_{A_i}^{1-4\alpha}}
\nonumber
\\
&\le
2p\sum_{i=0}^ke^{-\frac{\delta^2\n \left(G_++\gamma G_\Delta\right)^{-1}\n_{\mathit{l}^2}^2}{4p^2(k+1)^2}\left(n_{A_i}^{1-4\alpha}-(N_{A_i}')^{1-4\alpha}\right)},
\end{align}
Recall that for the second term of \eqref{used} we have
\begin{align}\label{2term2}
&\P\left(\n\left(G_++\gamma G_\Delta\right)^{-1} \n_{\mathit{l}^2}
\n \left( Z_++\gamma Z_\Delta \right)-\left(\hat{Z}_+(\textbf{n})+\gamma\hat{Z}_\Delta(\textbf{n})\right)\n_{\mathit{l}^2}\ge c/3 \right)
\le \nonumber
\\
&2p\sum_{i=0}^ke^{-\frac{\left(\frac{c}{3\n\left(G_++\gamma G_\Delta\right)^{-1} \n_{\mathit{l}^2}}- (1+\gamma )g_i(n_{A_i}^\alpha) \right)_+^2}{2p^2(k+1)^2}n_{A_i}^{1-4\alpha}}
+
1-\prod_{i=0}^k\left(F_{\left|X^{A_i}(1)\right|,...,\left|X^{A_i}(p)\right|,|Y^{A_i}|}(n_{A_i}^\alpha,...,n_{A_i}^\alpha)\right)^{n_{A_i}}.
\end{align}
Letting $N''_{A_i}=\min\left\{n_{A_i}\in\N: (1+\gamma)g_i(n_{A_i}^\alpha)<\frac{c}{6\n\left(G_++\gamma G_\Delta\right)^{-1} \n_{\mathit{l}^2}} \right\}$ then due to \eqref{gbound} we get the following bound for $N_{A_i}''$,
\begin{align*}
N''_{A_i}\le \left(\frac{6(1+\gamma)M(\zeta)^{\frac{2}{\zeta}}\n \left(G_++\gamma G_\Delta\right)^{-1}\n_{\mathit{l}^2}}{c}\right)^{\frac{1}{\alpha(\zeta-2)}}\sigma(\zeta)^{\frac{1}{\alpha\zeta}}.
\end{align*}
We get the following bound for the first term on the right-hand side of \eqref{2term2},
\begin{align}\label{second2}
&2p\sum_{i=0}^ke^{-\frac{\left(\frac{c}{3\n\left(G_++\gamma G_\Delta\right)^{-1} \n_{\mathit{l}^2}}- (1+\gamma )g_i(n_{A_i}^\alpha) \right)_+^2}{2p^2(k+1)^2}n_{A_i}^{1-4\alpha}}
\le
4p\sum_{i=0}^ke^{-\frac{\frac{c^2}{9\n\left(G_++\gamma G_\Delta\right)^{-1} \n_{\mathit{l}^2}^2}}{4p^2(k+1)^2}\left(n_{A_i}^{1-4\alpha}-(N_{A_i}'')^{1-4\alpha}\right)},
\end{align}
For the third term of \eqref{used},
\begin{align}\label{term3}
&\P\left((1+\delta)
\n \left(G_++\gamma G_\Delta\right)^{-1} \n_{\mathit{l}^2}^2
\n G_++\gamma G_\Delta-\left(\hat{G}_+(\textbf{n})+\gamma\hat{G}_\Delta(\textbf{n})\right)\n_{\mathit{l}^2}
\n\left( Z_++\gamma Z_\Delta \right)\n_{\mathit{l}^2}\ge c/3\right)\le\nonumber
\\
&    p\sum_{i=0}^ke^{-\frac{\left( \frac{c}{3(1+\delta)
\n \left(G_++\gamma G_\Delta\right)^{-1} \n_{\mathit{l}^2}^2\n\left( Z_++\gamma Z_\Delta \right)\n_{\mathit{l}^2}}  -(1+\gamma )h_i(n_{A_i}^\alpha)\right)_+^2}{(k+1)^2 2p^2}n_{A_i}^{1-4\alpha}}
+1-\prod_{i=0}^k\left(F_{\left|X^{A_i}(1)\right|,...,\left|X^{A_i}(p)\right|}(n_{A_i}^\alpha,...,n_{A_i}^\alpha)\right)^{n_{A_i}}
\end{align}
we may bound the first term on the right-hand side by
\begin{align}\label{second3}
&p\sum_{i=0}^ke^{-\frac{\left( \frac{c}{3(1+\delta)
\n \left(G_++\gamma G_\Delta\right)^{-1} \n_{\mathit{l}^2}^2\n Z_++\gamma Z_\Delta \n_{\mathit{l}^2}}  -(1+\gamma )h_i(n_{A_i}^\alpha)\right)_+^2}{(k+1)^2 2p^2}n_{A_i}^{1-4\alpha}}
\le
2p\sum_{i=0}^ke^{-\frac{\frac{c^2}{9(1+\delta)^2
\n \left(G_++\gamma G_\Delta\right)^{-1} \n_{\mathit{l}^2}^4\n Z_++\gamma Z_\Delta \n_{\mathit{l}^2}^2}}{4p^2(k+1)^2}\left(n_{A_i}^{1-4\alpha}-(N_{A_i}''')^{1-4\alpha}\right)},
\end{align}
where
\begin{align*}
N'''_{A_i}\le \left(\frac{6(1+\gamma)(1+\delta)M(\zeta)^{\frac{2}{\zeta}}\n \left(G_++\gamma G_\Delta\right)^{-1}\n_{\mathit{l}^2}^2\n Z_++\gamma Z_\Delta \n_{\mathit{l}^2}}{c}\right)^{\frac{1}{\alpha(\zeta-2)}}\sigma(\zeta)^{\frac{1}{\alpha\zeta}}.
\end{align*}
For the final and fourth term
\small
\begin{align}\label{term4}
&\P\left(\left\{(1+\delta)
\n \left(G_++\gamma G_\Delta\right)^{-1} \n_{\mathit{l}^2}^2
\n G_++\gamma G_\Delta-\left(\hat{G}_+(\textbf{n})+\gamma\hat{G}_\Delta(\textbf{n})\right)\n_{\mathit{l}^2}\n \left( Z_++\gamma Z_\Delta \right)-\left(\hat{Z}_+(\textbf{n})+\gamma\hat{Z}_\Delta(\textbf{n})\right)\n_{\mathit{l}^2}
\ge c/3\right\}\cap D_\textbf{n}(\delta)\right)\nonumber
\\
&\le
2p\sum_{i=0}^ke^{-\frac{\left(\frac{c}{3(1+\delta)\n \left(G_++\gamma G_\Delta\right)^{-1}\n_{\mathit{l}^2}^3\delta}   -(1+\gamma )g_i(n_{A_i}^\alpha)\right)_+^2}
{(k+1)^22p^2}n_{A_i}^{1-4\alpha}}
+1-\prod_{i=0}^k\left(F_{\left|X^{A_i}(1)\right|,...,\left|X^{A_i}(p)\right|,|Y^{A_i}|}(n_{A_i}^\alpha,...,n_{A_i}^\alpha)\right)^{n_{A_i}}.
\end{align}
\normalsize
we may bound the first term on the right-hand side by
\begin{align}\label{second4}
&2p\sum_{i=0}^ke^{-\frac{\left(\frac{c}{3(1+\delta)\n \left(G_++\gamma G_\Delta\right)^{-1}\n_{\mathit{l}^2}^3\delta}   -(1+\gamma )g_i(n_{A_i}^\alpha)\right)_+^2}
{(k+1)^22p^2}n_{A_i}^{1-4\alpha}}
\le
2p\sum_{i=0}^ke^{-\frac{\frac{c^2}{9(1+\delta)^2
\n \left(G_++\gamma G_\Delta\right)^{-1} \n_{\mathit{l}^2}^6}}{4p^2(k+1)^2}\left(n_{A_i}^{1-4\alpha}-(N_{A_i}''')^{1-4\alpha}\right)},
\end{align}
where 
\begin{align*}
N''''_{A_i}\le \left(\frac{6(1+\gamma)(1+\delta)M(\zeta)^{\frac{2}{\zeta}}\delta\n \left(G_++\gamma G_\Delta\right)^{-1}\n_{\mathit{l}^2}^3}{c}\right)^{\frac{1}{\alpha(\zeta-2)}}\sigma(\zeta)^{\frac{1}{\alpha\zeta}}.
\end{align*}
We now collect all the terms and use the fact that $F_{\left|X^{A_i}(1)\right|,...,\left|X^{A_i}(p)\right|}(n_{A_i}^\alpha,...,n_{A_i}^\alpha) 
\ge 
F_{\left|X^{A_i}(1)\right|,...,\left|X^{A_i}(p)\right|,|Y^{A_i}|}(n_{A_i}^\alpha,...,n_{A_i}^\alpha)$ to find that when $n_{A_i}\ge N_{A_i}$, 
\begin{align}\label{sista}
&\P\left(\n \tilde{\beta}_\gamma(\textbf{n}) -  \beta_\gamma\n_{\mathit{l}^2}\ge c \right)\nonumber
\\
&\le 10p\sum_{i=0}^ke^{\frac{(c\vee \delta)^2\left(\n \left(G_++\gamma G_\Delta\right)^{-1}\n_{\mathit{l}^2}\vee \n G_++\gamma G_\Delta\n_{\mathit{l}^2}\right)^6}{6p^2(k+1)^2\left(1\wedge\n\left( Z_++\gamma Z_\Delta \right)\n_{\mathit{l}^2}\right)^2} (N_{A_i}'\vee N_{A_i}''\vee N_{A_i}'''\vee N_{A_i}'''')^{1-4\alpha}}
\nonumber
\\
&\times e^{-\frac{(c\wedge \delta)^2\left(\n \left(G_++\gamma G_\Delta\right)^{-1}\n_{\mathit{l}^2}\wedge \n G_++\gamma G_\Delta\n_{\mathit{l}^2}\right)^6}{6p^2(k+1)^2(1+\delta)\left(1\vee
\n\left( Z_++\gamma Z_\Delta \right)\n_{\mathit{l}^2}\right)^2} n_{A_i}^{1-4\alpha}}
\nonumber
\\
&+4\left(1-\prod_{i=0}^k\left(F_{\left|X^{A_i}(1)\right|,...,\left|X^{A_i}(p)\right|,|Y^{A_i}|}(n_{A_i}^\alpha,...,n_{A_i}^\alpha)\right)^{n_{A_i}}\right).
\end{align}
By noting that $\n \left(G_++\gamma G_\Delta\right)^{-1}\n_{\mathit{l}^2}^{-1}\le \n G_++\gamma G_\Delta\n_{\mathit{l}^2}$ we arrive at the following bound
\begin{align*}
&N_{A_i}'\vee N_{A_i}''\vee N_{A_i}'''\vee N_{A_i}''''
\\
&\le
  M(\zeta)^{\frac{1}{\alpha(\zeta-2)}}(p+1)^{\frac{1}{\alpha\zeta}}\left((1\vee c^{-1}\vee\delta^{-1})6(1+\gamma)(1+\delta)\left(1\vee \n \left(G_++\gamma G_\Delta\right)^{-1}\n_{\mathit{l}^2}\right)^3\left(1\vee \n Z_++\gamma Z_\Delta \n_{\mathit{l}^2}\right)\right)^{\frac{1}{\alpha(\zeta-2)}},
\end{align*}
which gives us the desired result when plugged back into \eqref{sista}.
\end{proof}
\newpage
\subsection{Proof of Theorem 2}
\begin{proof}
We first split the difference across the different environments,
\begin{align*}
&\P\left( \left|\hat{R}(\textbf{n}) -R\right|\ge c \right)
\le
\sum_{i=0}^k\P\left( \left|\hat{R}_{A_i}(\hat{\beta}_\gamma(\textbf{n})) -R_{A_i}(\beta_\gamma)\right|\ge \frac{c}{2(\gamma+1)(k+1)} \right).
\end{align*}
For each environment we now split across the three main terms for each of the $p$ covariates,
\begin{align}\label{envisplit}
&\P\left( \left|\hat{R}_{A_i}(\hat{\beta}_\gamma(\textbf{n})) -R_{A_i}(\beta_\gamma)\right|\ge \frac{c}{2(\gamma+1)(k+1)} \right)
\le
\P\left( \left|\E\left[(Y^{A_i})^2\right] -\frac{1}{n_{A_i}}\sum_{u=1}^{n_{A_i}}(Y^{A_i}_u)^2\right|\ge \frac{c}{6(\gamma+1)(k+1)} \right)
\nonumber\\
&+
\sum_{l=1}^p\P\left( \left|\beta_\gamma(l)\E\left[Y^{A_i}X^{A_i}(l)\right] -\hat{\beta}_\gamma(l)\frac{1}{n_{A_i}}\sum_{u=1}^{n_{A_i}}Y^{A_i}_uX^{A_i}_u(l)\right|\ge \frac{c}{12(\gamma+1)(k+1)p} \right)
\nonumber\\
&+
\sum_{l=1}^p\P\left( \left|\beta_\gamma(l)^2\E\left[(X^{A_i}(l))^2\right] -\hat{\beta}_\gamma(l)^2\frac{1}{n_{A_i}}\sum_{u=1}^{n_{A_i}}(X^{A_i}_u(l))^2\right|\ge \frac{c}{12(\gamma+1)(k+1)p} \right).
\end{align}
We not handle each of the three terms above separately. Let us first introduce a bit of notation. Let 
$g_{i,l}(R)=\E\left[\left|X^{A_i}(l)Y^{A_i}\right|1_{\left|X^{A_i}(l)Y^{A_i}\right|> R}\right],$ $h_{i,l}(R)=\E\left[\left|X^{A_i}(l)\right|^21_{\left|X^{A_i}(l)\right|^2> R}\right]$ 
and 
$f_{i}(R)=\E\left[\left|Y^{A_i}\right|^21_{\left|Y^{A_i}\right|>R}\right]$.
We split the differences between the tail and bulk part of each term in \eqref{envisplit} and apply the Hoeffding inequality on the bulk part.
\begin{align*}
&\P\left( \left|\E\left[(Y^{A_i})^2\right] -\frac{1}{n_{A_i}}\sum_{u=1}^{n_{A_i}}(Y^{A_i}_u)^2\right|\ge \frac{c}{6(\gamma+1)(k+1)} \right)
\\
&\le
\P\left( \left|\E\left[(Y^{A_i})^21_{(Y^{A_i})^2\le n_{A_i}^{2\alpha}}\right] -\frac{1}{n_{A_i}}\sum_{u=1}^{n_{A_i}}(Y^{A_i}_u)^21_{(Y^{A_i}_u)^2\le n_{A_i}^{2\alpha}}\right|\ge \frac{c}{12(\gamma+1)(k+1)}-\E\left[(Y^{A_i})^21_{(Y^{A_i})^2> n_{A_i}^{2\alpha}}\right] \right)
\\
&+
\P\left( \left|1_{\bigcup_{v=1}^{n_{A_i}}\left\{(Y^{A_i}_v)^2> n_{A_i}^{2\alpha}\right\}}\frac{1}{n_{A_i}}\sum_{u=1}^{n_{A_i}}(Y^{A_i}_u)^2\right|\ge \frac{c}{12(\gamma+1)(k+1)} \right)
\\
&\le
e^{-\frac{2\left(\frac{c}{12(\gamma+1)(k+1)}-\E\left[(Y^{A_i})^21_{(Y^{A_i})^2> n_{A_i}^{2\alpha}}\right]\right)^2}{n_{A_i}^{2\alpha}}n_{A_i}}
+
\P\left(\bigcup_{v=1}^{n_{A_i}}\left\{(Y^{A_i}_v)^2> n_{A_i}^{2\alpha}\right\}\right)
\\
&\le e^{-2\left(\frac{c}{12(\gamma+1)(k+1)}-\E\left[(Y^{A_i})^21_{(Y^{A_i})^2> n_{A_i}^{2\alpha}}\right]\right)^2n_{A_i}^{1-4\alpha}}
+
1-\left(F_{\left|Y^{A_i}\right|}(n_{A_i}^{\alpha})\right)^{n_{A_i}}
\end{align*}
With $n_{A_i}$ so large that $\frac{c}{\n\beta_\gamma\n_{\mathit{l}^2} 48(\gamma+1)(k+1)p}>\E\left[\left|X^{A_i}(l)Y^{A_i}\right|1_{\left|X^{A_i}(l)Y^{A_i}\right|> n_{A_i}^{2\alpha}}\right]$ (recall that we use the convention $1/0=+\infty$) we have that
\footnotesize
\begin{align*}
&\P\left( \left|\beta_\gamma(l)\E\left[Y^{A_i}X^{A_i}(l)\right] -\hat{\beta}_\gamma(l)\frac{1}{n_{A_i}}\sum_{u=1}^{n_{A_i}}Y^{A_i}_uX^{A_i}_u(l)\right|\ge \frac{c}{12(\gamma+1)(k+1)p} \right)
\\
&\le
\P\left( \left|\beta_\gamma(l)-\hat{\beta}_\gamma(l)\right|\left|\frac{1}{n_{A_i}}\sum_{u=1}^{n_{A_i}}Y^{A_i}_uX^{A_i}_u(l)\right|\ge \frac{c}{24(\gamma+1)(k+1)p} \right)
+
\P\left( \left|\beta_\gamma(l)\right|\left|\E\left[Y^{A_i}X^{A_i}(l)\right] -\frac{1}{n_{A_i}}\sum_{u=1}^{n_{A_i}}Y^{A_i}_uX^{A_i}_u(l)\right|\ge \frac{c}{24(\gamma+1)(k+1)p} \right)
\\
&\le
\P\left( \bigcup_{v=1}^{n_{A_i}}\left\{\left|X_v^{A_i}(l)Y^{A_i}_v\right|> n_{A_i}^{2\alpha}\right\} \right)
+
\P\left( \left|\beta_\gamma(l)-\hat{\beta}_\gamma(l)\right|\left|\frac{1}{n_{A_i}}\sum_{u=1}^{n_{A_i}}Y^{A_i}_uX^{A_i}_u(l)1_{\left|X_u^{A_i}(l)Y_u^{A_i}\right|\le n_{A_i}^{2\alpha}}\right|\ge \frac{c}{48(\gamma+1)(k+1)p} \right)
\\
&+
\P\left( \left|\E\left[Y^{A_i}X^{A_i}(l)\right] -\frac{1}{n_{A_i}}\sum_{u=1}^{n_{A_i}}Y^{A_i}_uX^{A_i}_u(l)\right|\ge \frac{c}{\n\beta_\gamma\n_{\mathit{l}^2} 24(\gamma+1)(k+1)p} \right)
\\
&\le
1-\left(F_{\left|X^{A_i}(l)\right|,\left|Y^{A_i}\right|}(n_{A_i}^{\alpha},n_{A_i}^{\alpha})\right)^{n_{A_i}}
+
\P\left( \n\beta_\gamma-\hat{\beta}_\gamma\n_{\mathit{l}^2}\ge \frac{c}{48(\gamma+1)(k+1)pn_{A_i}^{2\alpha}} \right)
\\
&+
\P\left( \left|\E\left[Y^{A_i}X^{A_i}(l)1_{\left|Y^{A_i}X^{A_i}(l)\right|\le n_{A_i}^{2\alpha}}\right] -\frac{1}{n_{A_i}}\sum_{u=1}^{n_{A_i}}Y^{A_i}_uX^{A_i}_u(l)1_{\left|Y^{A_i}X^{A_i}(l)\right|\le n_{A_i}^{2\alpha}}\right|
 \ge\frac{c}{\n\beta_\gamma\n_{\mathit{l}^2} 48(\gamma+1)(k+1)p}-g_i(n_{A_i}^{2\alpha}) \right)
\\
&+
\P\left( \bigcup_{v=1}^{n_{A_i}}\left\{\left|X_v^{A_i}(l)Y^{A_i}_v\right|> n_{A_i}^{2\alpha}\right\} \right)
\le
\P\left( \n\beta_\gamma(l)-\hat{\beta}_\gamma(l)\n_{\mathit{l}^2}\ge\frac{c}{48n_{A_i}^{2\alpha}(\gamma+1)(k+1)p} \right)
+
2\left(1-\left(F_{\left|X^{A_i}(l)\right|,\left|Y^{A_i}\right|}(n_{A_i}^{\alpha},n_{A_i}^{\alpha})\right)^{n_{A_i}}\right)
\\
&+
e^{-\frac{2\left(\frac{c}{\n\beta_\gamma\n_{\mathit{l}^2} 48(\gamma+1)(k+1)p}-g_i(n_{A_i}^{2\alpha})\right)^2}{n_{A_i}^{4\alpha}}n_{A_i}}.
\end{align*}
\normalsize
The final term of \eqref{envisplit} is handled analogously,
\small
\begin{align*}
&\P\left( \left|\beta_\gamma(l)\E\left[(X^{A_i}(l))^2\right] -\hat{\beta}_\gamma(l)^2\frac{1}{n_{A_i}}\sum_{u=1}^{n_{A_i}}(X^{A_i}_u(l))^2\right|\ge \frac{c}{12(\gamma+1)(k+1)p} \right)
\le
\P\left( \n\beta_\gamma(l)-\hat{\beta}_\gamma(l)\n_{\mathit{l}^2}\ge\frac{c}{48n_{A_i}^{2\alpha}(\gamma+1)(k+1)p} \right)
\\
&+
2\left(1-\left(F_{\left|X^{A_i}(l)\right|}(n_{A_i}^{\alpha})\right)^{n_{A_i}}\right)
+
e^{-\frac{2\left(\frac{c}{\n\beta_\gamma\n_{\mathit{l}^2} 48(\gamma+1)(k+1)p}-h_i(n_{A_i}^{2\alpha})\right)^2}{n_{A_i}^{4\alpha}}n_{A_i}}
\end{align*}
\normalsize
This leads to
\small
\begin{align*}
&\P\left( \left|\hat{R}_{A_i}(\hat{\beta}_\gamma(\textbf{n})) -R_{A_i}(\beta_\gamma)\right|\ge \frac{c}{2(\gamma+1)(k+1)} \right)
\le
e^{-\frac{2\left(\frac{c}{12(\gamma+1)(k+1)}-f_i(n_{A_i}^{2\alpha})\right)^2}{n_{A_i}^{4\alpha}}n_{A_i}}
+
1-\left(F_{\left|Y^{A_i}\right|}(n_{A_i}^{\alpha})\right)^{n_{A_i}}
\\
&+
2p\P\left( \n\beta_\gamma-\hat{\beta}_\gamma\n_{\mathit{l}^2}\ge\frac{c}{48n_{A_i}^{2\alpha}(\gamma+1)(k+1)p} \right)
+
2\sum_{l=1}^p\left(1-\left(F_{\left|X^{A_i}(l)\right|,\left|Y^{A_i}\right|}(n_{A_i}^{\alpha},n_{A_i}^{\alpha})\right)^{n_{A_i}}\right)
+
pe^{-\frac{2\left(\frac{c}{\n\beta_\gamma\n_{\mathit{l}^2} 48(\gamma+1)(k+1)p}-g_i(n_{A_i}^{2\alpha})\right)^2}{n_{A_i}^{4\alpha}}n_{A_i}}
\\
&+
pe^{-\frac{2\left(\frac{c}{\n\beta_\gamma\n_{\mathit{l}^2} 48(\gamma+1)(k+1)p}-h_i(n_{A_i}^{2\alpha})\right)^2}{n_{A_i}^{4\alpha}}n_{A_i}}
+
\sum_{l=1}^p\left(1-\left(F_{\left|X^{A_i}(l)\right|}(n_{A_i}^{\alpha})\right)^{n_{A_i}}\right).
\end{align*}
\normalsize
Therefore
\begin{align*}
&\P\left( \left|\hat{R}(\textbf{n}) -R\right|\ge c \right)
\le
\sum_{i=0}^k\P\left( \left|\hat{R}_{A_i}(\hat{\beta}_\gamma(\textbf{n})) -R_{A_i}(\beta_\gamma)\right|\ge \frac{c}{2(\gamma+1)(k+1)} \right)
\\
&\le
(2p+1)(k+1)e^{-\frac{2\left(\frac{c}{\n\beta_\gamma\n_{\mathit{l}^2} 48(\gamma+1)(k+1)p}-h_i(n_{A_i}^{2\alpha})\vee f_i(n_{A_i}^{2\alpha})\vee g_i(n_{A_i}^{2\alpha})\right)^2}{n_{A_i}^{4\alpha}}n_{A_i}}
+
2p(k+1)\P\left( \n\beta_\gamma-\hat{\beta}_\gamma\n_{\mathit{l}^2}\ge\frac{c}{48n_{A_i}^{2\alpha}(\gamma+1)(k+1)p} \right)
\\
&+4(k+1)\sum_{l=1}^p\left(1-\left(F_{\left|X^{A_i}(l)\right|,\left|Y^{A_i}\right|}(n_{A_i}^{\alpha},n_{A_i}^{\alpha})\right)^{n_{A_i}}\right)
\end{align*}
We now let 
$$N_{A_i}'=\inf\{n\in\N: \max_l\left(h_{i,l}(n_{A_i}^\alpha)\vee f_i(n_{A_i}^\alpha)\vee g_{i,l}(n_{A_i}^\alpha)\right)\le \frac{c}{2\n\beta_\gamma\n_{\mathit{l}^2} 48(\gamma+1)(k+1)p} \}\vee N_{A_i},$$
where $N_{A_i}$ is defined as in Theorem 1.1). Then $\max_l\left(h_{i,l}(n_{A_i}^\alpha)\vee f_i(n_{A_i}^\alpha)\vee g_{i,l}(n_{A_i}^\alpha)\right)\le \frac{c}{2\n\beta_\gamma\n_{\mathit{l}^2} 48(\gamma+1)(k+1)p} $ for all $n_{A_i}\ge N_{A_i}'$ due to the monotonicity of $h_{i,l}$, $f_i$ and $g_{i,l}$. We now plug in the (first) concentration inequality from Theorem 1 and we get
\begin{align*}
&\P\left( \left|\hat{R}(\textbf{n}) -R\right|\ge c \right)
\le
(4k+9)\left(1-\prod_{i=0}^k\left(F_{\left|X^{A_i}(1)\right|,...,\left|X^{A_i}(p)\right|,|Y^{A_i}|}(n_{A_i}^\alpha,...,n_{A_i}^\alpha)\right)^{n_{A_i}}\right)
\\
&+2p(k+1)(3+4p)e^{-(\delta\wedge c)\left(E\wedge \frac{c}{\n\beta_\gamma\n_{\mathit{l}^2} 48(\gamma+1)(k+1)p}\right)^2},
\end{align*}
for $n_{A_i}\ge N_{A_i}'$.
\end{proof}
\newpage
\subsection{Proof of Theorem 3}
\begin{proof}[Proof of (2)]
\textbf{Step 1}: Split the estimator into a local part and a tail part.
\\
Let 
\begin{align}\label{delta1}
\delta=\frac 12\wedge\left(\frac{C-(\n \left( G_+ +\gamma G_\Delta\right)^{-1} \n_{\mathit{l}^2}+ \n \left( G_+ +\gamma G_\Delta\right)^{-1} \n_{\mathit{l}^2}^3)}{2\n \left( G_+ +\gamma G_\Delta\right)^{-1} \n_{\mathit{l}^2}^3}\right),
\end{align}
then by this definition and the assumption on $C$, $C>\n \left( G_+ +\gamma G_\Delta\right)^{-1} \n+ (1+\delta)\n \left( G_+ +\gamma G_\Delta\right)^{-1} \n^3$ and $0<\delta<1$.
With $\delta$ defined as above, let $D_{\textbf{n}}(\delta)$ be defined as in the proof of Theorem 1,
$$D_{\textbf{n}}(\delta)=\left\{\n G_++\gamma G_\Delta-\left(\hat{G}_+(\textbf{n})+\gamma\hat{G}_\Delta(\textbf{n})\right)\n_{\mathit{l}^2}
<
\delta\n \left(G_++\gamma G_\Delta\right)^{-1}\n_{\mathit{l}^2} \right\}.$$
For $\omega\in D_{\textbf{n}}(\delta)$, we know that
\begin{align*}
&\n \left(\hat{G}_+(\textbf{n})+\gamma\hat{G}_\Delta(\textbf{n}) \right)^{-1} \n_{\mathit{l}^2}
\le 
\n \left(\hat{G}_+(\textbf{n})+\gamma\hat{G}_\Delta(\textbf{n}) \right)^{-1} - \left(G_++\gamma G_\Delta \right)^{-1} \n_{\mathit{l}^2}
+
\n\left(G_++\gamma G_\Delta \right)^{-1} \n_{\mathit{l}^2}
\le
\\
&(1+\delta)\n\left(G_++\gamma G_\Delta \right)^{-1} \n^2\n \hat{G}_+(\textbf{n})+\gamma\hat{G}_\Delta(\textbf{n}) - \left(G_++\gamma G_\Delta \right) \n+\n\left(G_++\gamma G_\Delta \right)^{-1} \n
\le
\\
&\n \left( G_+ +\gamma G_\Delta\right)^{-1} \n+ (1+\delta)\delta\n \left( G_+ +\gamma G_\Delta\right)^{-1} \n^3< C, 
\end{align*}
where we used the inequality \eqref{Ginvineq} from the proof (in the appendix) of Theorem 1, which is valid on the set $D_\textbf{n}(\delta)$. So if we let $C_\textbf{n}=\left\{ \n \left(\hat{G}_+(\textbf{n})+\gamma\hat{G}_\Delta(\textbf{n}) \right)^{-1} \n \le C  \right\}$ it follows that $D_{\textbf{n}}(\delta) \subseteq C_\textbf{n}$. Due to the Radon inequality 
\begin{align}\label{varbnd11}
&\E\left[\n\hat{\beta}_\gamma(\textbf{n})1_{C_\textbf{n}} -\E\left[\hat{\beta}_\gamma(\textbf{n})1_{C_\textbf{n}} \right] \n_{\mathit{l}^2}^q\right] 
\le
\E\left[\left(\n\hat{\beta}_\gamma(\textbf{n})1_{C_\textbf{n}} -\beta_\gamma \n_{\mathit{l}^2}+ \n \beta_\gamma-\E\left[\hat{\beta}_\gamma(\textbf{n})1_{C_\textbf{n}} \right] \n_{\mathit{l}^2}\right)^q\right] 
\nonumber
\\
&\le
2^{q-1}\E\left[\n\hat{\beta}_\gamma(\textbf{n})1_{C_\textbf{n}} -\beta_\gamma \n_{\mathit{l}^2}^q \right]
+
2^{q-1}\E\left[\n \beta_\gamma-\E\left[\hat{\beta}_\gamma(\textbf{n})1_{C_\textbf{n}} \right] \n_{\mathit{l}^2}^q\right]
\le 
2^q\E\left[\n\hat{\beta}_\gamma(\textbf{n})1_{C_\textbf{n}} -\beta_\gamma \n_{\mathit{l}^2}^q \right]
\nonumber
\\
&\le 
2^q\E\left[\n\hat{\beta}_\gamma(\textbf{n}) -\beta_\gamma \n_{\mathit{l}^2}^q1_{C_\textbf{n}} \right]+2^q\n\beta_\gamma \n_{\mathit{l}^2}^q\P\left(C_\textbf{n}^c\right)
\le 
2^q\E\left[\n\hat{\beta}_\gamma(\textbf{n}) -\beta_\gamma \n_{\mathit{l}^2}^q1_{C_\textbf{n}} \right]+2^q\n\beta_\gamma \n_{\mathit{l}^2}^q\P\left(D_\textbf{n}(\delta)^c\right)
\end{align}
where we used the inequality $\n\E\left[ \textbf{X} \right] \n_{\mathit{l}^2}^q \le \E\left[ \n\textbf{X} \n_{\mathit{l}^2}^q\right],$
for any vector $\textbf{X}$ of $L^q(\P)$ variables, which is just an application of the Jensen inequality, first to the norm $\n.\n_{\mathit{l}^2}$ and then to the map $x^q$. We now note that
\begin{align}\label{cnx11}
&\E\left[\n\hat{\beta}_\gamma(\textbf{n}) -\beta_\gamma \n_{\mathit{l}^2}^q1_{C_\textbf{n}} \right]
=
\E\left[\n\hat{\beta}_\gamma(\textbf{n}) -\beta_\gamma \n_{\mathit{l}^2}^q1_{C_\textbf{n}\cap D_\textbf{n}(\delta)} \right]
+
\E\left[\n\hat{\beta}_\gamma(\textbf{n}) -\beta_\gamma \n_{\mathit{l}^2}^q1_{C_\textbf{n}\cap D_\textbf{n}(\delta)^c} \right]
\nonumber
\\
&\le 
\E\left[\n\hat{\beta}_\gamma(\textbf{n}) -\beta_\gamma \n_{\mathit{l}^2}^q1_{C_\textbf{n}\cap D_\textbf{n}(\delta)} \right]
+
2^{q-1}\E\left[\n\hat{\beta}_\gamma(\textbf{n})\n_{\mathit{l}^2}^q1_{C_\textbf{n}\cap D_\textbf{n}(\delta)^c} \right]
+
2^{q-1}\n\beta_\gamma \n_{\mathit{l}^2}^q \P\left(C_\textbf{n}\cap D_\textbf{n}(\delta)^c\right)
\nonumber
\\
&\le
2^{q-1}\E\left[\n\hat{\beta}_\gamma(\textbf{n}) -\beta_\gamma \n_{\mathit{l}^2}^q1_{C_\textbf{n}\cap D_\textbf{n}(\delta)} \right]
+
2^{q-1}\E\left[\n\hat{\beta}_\gamma(\textbf{n})\n_{\mathit{l}^2}^q1_{C_\textbf{n}\cap D_\textbf{n}(\delta)^c} \right]
+
2^{q-1}\n\beta_\gamma \n_{\mathit{l}^2}^q \P\left(D_\textbf{n}(\delta)^c\right) ,
\end{align}
\textbf{Step 2}: Linearise the local part.
\\
Using the local Lipschitz property of the matrix inverse we have (see inequality \eqref{localexpansion} in the Appendix and note that this calculation is independent of the chosen norm) on $D_\textbf{n}(\delta)$
\begin{align}\label{localexpansion1}
\n \tilde{\beta}_\gamma(\textbf{n}) -  \beta_\gamma\n_{\mathit{l}^2}
\le &    
\n\left(G_++\gamma G_\Delta\right)^{-1} \n_{\mathit{l}^2}
\n \left( Z_++\gamma Z_\Delta \right)-\left(\hat{Z}_+(\textbf{n})+\gamma\hat{Z}_\Delta(\textbf{n})\right)\n_{\mathit{l}^2}\nonumber
\\
&+(1+\delta)
\n \left(G_++\gamma G_\Delta\right)^{-1} \n_{\mathit{l}^2}^2
\n G_++\gamma G_\Delta-\left(\hat{G}_+(\textbf{n})+\gamma\hat{G}_\Delta(\textbf{n})\right)\n_{\mathit{l}^2}
 \nonumber \\
&\times \big(\n\left( Z_++\gamma Z_\Delta \right)\n_{\mathit{l}^2}+\n\left( Z_++\gamma Z_\Delta \right)-\left(\hat{Z}_+(\textbf{n})+\gamma\hat{Z}_\Delta(\textbf{n})\right)\n_{\mathit{l}^2}\big).
\end{align}
Due to \eqref{localexpansion1},
\begin{align}\label{expbound}
&\E\left[\n\hat{\beta}_\gamma(\textbf{n}) -\beta_\gamma  \n_{\mathit{l}^2}^q1_{ D_\textbf{n}(\delta)\cap C_\textbf{n}}\right]
\le
\E\left[\n\hat{\beta}_\gamma(\textbf{n}) -\beta_\gamma  \n_{\mathit{l}^2}^q1_{ D_\textbf{n}(\delta)}\right]
\nonumber
\\
&\le 
2^{q-1}\n\left(G_++\gamma G_\Delta\right)^{-1} \n_{\mathit{l}^2}^{q}
\E\left[\n \left( Z_++\gamma Z_\Delta \right)-\left(\hat{Z}_+(\textbf{n})+\gamma\hat{Z}_\Delta(\textbf{n})\right)\n_{\mathit{l}^2}^{q}
\right]\nonumber
\\
&+2^{q-1}(1+\delta)^{2(q-1)}\n \left(G_++\gamma G_\Delta\right)^{-1} \n_{\mathit{l}^2}^{2q}\E\left[
\n G_++\gamma G_\Delta-\left(\hat{G}_+(\textbf{n})+\gamma\hat{G}_\Delta(\textbf{n})\right)\n_{\mathit{l}^2}^q
\right.\nonumber
\\
&\left. \times \big(\n\left( Z_++\gamma Z_\Delta \right)\n_{\mathit{l}^2}+
\n \left( Z_++\gamma Z_\Delta \right)-\left(\hat{Z}_+(\textbf{n})+\gamma\hat{Z}_\Delta(\textbf{n})\right)\n_{\mathit{l}^2}\big)^q1_{D_\textbf{n}(\delta)}\right]
\nonumber
\\
&\le 2^{q-1}\n\left(G_++\gamma G_\Delta\right)^{-1} \n_{\mathit{l}^2}^q\E\left[
\n \left( Z_++\gamma Z_\Delta \right)-\left(\hat{Z}_+(\textbf{n})+\gamma\hat{Z}_\Delta(\textbf{n})\right)\n_{\mathit{l}^2}^q\right]\nonumber
\\
&+2^{q-1}(1+\delta)^q
\n \left(G_++\gamma G_\Delta\right)^{-1} \n_{\mathit{l}^2}^{2q}\left(\n\left( Z_++\gamma Z_\Delta \right)\n_{\mathit{l}^2}^q\E\left[\n G_++\gamma G_\Delta-\left(\hat{G}_+(\textbf{n})+\gamma\hat{G}_\Delta(\textbf{n})\right)\n_{\mathit{l}^2}^q1_{ D_\textbf{n}(\delta)}\right]
\right.\nonumber
\\
&\left.
+\E\left[\n \left( Z_++\gamma Z_\Delta \right)-\left(\hat{Z}_+(\textbf{n})+\gamma\hat{Z}_\Delta(\textbf{n})\right)\n_{\mathit{l}^2}^q\n G_++\gamma G_\Delta-\left(\hat{G}_+(\textbf{n})+\gamma\hat{G}_\Delta(\textbf{n})\right)\n_{\mathit{l}^2}^q1_{D_\textbf{n}}(\delta)\right]
\right).
\end{align}
\textbf{Step 3}: Bound the tail part in terms of the probability of the set $D_\textbf{n}(\delta)^c$.
\\
By the definition of the set $D_\textbf{n}(\delta)$,
\begin{align}
&\E\left[\n \left( Z_++\gamma Z_\Delta \right)-\left(\hat{Z}_+(\textbf{n})+\gamma\hat{Z}_\Delta(\textbf{n})\right)\n_{\mathit{l}^2}^q\n G_++\gamma G_\Delta-\left(\hat{G}_+(\textbf{n})+\gamma\hat{G}_\Delta(\textbf{n})\right)\n_{\mathit{l}^2}^q1_{D_\textbf{n}}(\delta)\right]
\nonumber
\\
&\le
\delta^q\n\left(G_++\gamma G_\Delta\right)^{-1}\n^q \E\left[\n \left( Z_++\gamma Z_\Delta \right)-\left(\hat{Z}_+(\textbf{n})+\gamma\hat{Z}_\Delta(\textbf{n})\right)\n_{\mathit{l}^2}^q\right].
\end{align}
For the second term on the right-most side of \eqref{cnx11}, using the Radon and Hölder inequalities (with conjugate exponents $\frac{\zeta}{q}$ and $\frac{\zeta}{\zeta-q}$) we have
\begin{align}\label{lastTermInVar}
\E\left[\n\hat{\beta}_\gamma(\textbf{n})\n_{\mathit{l}^2}^q1_{C_\textbf{n}\cap D_\textbf{n}(\delta)^c}\right]
&\le 
\E\left[\n \hat{Z}_+(\textbf{n})+\gamma\hat{Z}_\Delta(\textbf{n}) \n_{\mathit{l}^2}^q \n\left(G_++\gamma G_\Delta\right)^{-1}\n_{\mathit{l}^2}^q1_{C_\textbf{n}\cap D_\textbf{n}(\delta)^c}\right]
\nonumber
\\
&\le
C^q\E\left[\n \hat{Z}_+(\textbf{n})+\gamma\hat{Z}_\Delta(\textbf{n}) \n_{\mathit{l}^2}^q 1_{ D_\textbf{n}(\delta)^c}\right]
\le
C^q\E\left[\n \hat{Z}_+(\textbf{n})+\gamma\hat{Z}_\Delta(\textbf{n}) \n_{\mathit{l}^2}^{\zeta}\right]^{\frac{q}{\zeta}}\P\left(D_\textbf{n}(\delta)^c\right)^{\frac{\zeta-q}{\zeta}}.
\end{align}
Let us now bound the term $\E\left[\n \hat{Z}_+(\textbf{n})+\gamma\hat{Z}_\Delta(\textbf{n}) \n_{\mathit{l}^2}^{\zeta/2}\right]$ above. By the Radon inequality it follows that 
\begin{align*}
&\n \hat{Z}_+(\textbf{n})+\gamma\hat{Z}_\Delta(\textbf{n}) \n_{\mathit{l}^2}^{\zeta}
\le 
2^{\zeta-1}\left(\n \hat{Z}_+(\textbf{n})\n_{\mathit{l}^2}^{\zeta}+\gamma^{\zeta}\n \hat{Z}_\Delta(\textbf{n})\n_{\mathit{l}^2}^{\zeta} \right)
\le
2^{\zeta-1}(1+\gamma^{\zeta})\left( \sum_{i=0}^k\n \frac{1}{n_{A_i}} (\mathbb{X}^{A_i}(\textbf{n}))^T\mathbb{Y}^{A_i}(\textbf{n}) \n_{\mathit{l}^2}\right)^{\zeta} 
\\
&\le 2^{\zeta-1}(1+\gamma^{\zeta})(k+1)^{\zeta-1}\sum_{i=0}^k\n \frac{1}{n_{A_i}} (\mathbb{X}^{A_i}(\textbf{n}))^T\mathbb{Y}^{A_i}(\textbf{n}) \n_{\mathit{l}^2}^{\zeta}.
\end{align*}
Further applications of the Radon inequality yields
\begin{align}
&\n \frac{1}{n_{A_i}} (\mathbb{X}^{A_i}(\textbf{n}))^T\mathbb{Y}^{A_i}(\textbf{n}) \n_{\mathit{l}^2}^{\zeta}
=
\left(\sum_{l=1}^{p} \left| \frac{1}{n_{A_i}} \sum_{u=1}^{n_{A_i}}X^{A_i}_u(l)Y^{A_i}_u \right|\right)^{\zeta}
\le
\sum_{l=1}^{p}\left| \frac{1}{n_{A_i}} \sum_{u=1}^{n_{A_i}}X^{A_i}_u(l)Y^{A_i}_u \right|^{\zeta/2}p^{\zeta-1}
\nonumber
\\
&\le
\sum_{l=1}^{p}\frac{1}{n_{A_i}}\sum_{u=1}^{n_{A_i}}\left|  X^{A_i}_u(l)Y^{A_i}_u \right|^{\zeta}p^{\zeta}.
\end{align}
Returning to \eqref{expbound} we now have that
\begin{align}\label{Zxibnd}
\E\left[\n \hat{Z}_+(\textbf{n})+\gamma\hat{Z}_\Delta(\textbf{n}) \n_{\mathit{l}^2}^{\zeta}\right]
&\le  
2^{\zeta-1}(1+\gamma^{\zeta})(k+1)^{\zeta}\sum_{l=1}^{p}\frac{1}{n_{A_i}}\sum_{u=1}^{n_{A_i}}\E\left[\left|  X^{A_i}_u(l)Y^{A_i}_u \right|^{\zeta}\right]p^{\zeta-1}
\nonumber
\\
&\le 2^{\zeta-1}(1+\gamma^{\zeta})(k+1)^{\zeta}p^{\zeta}\tilde{M}(\zeta).
\end{align}
Plugging this into \eqref{lastTermInVar} we get
\begin{align}\label{Zxibnd1}
&\E\left[\n\hat{\beta}_\gamma(\textbf{n})  \n_{\mathit{l}^2}^q1_{ D_\textbf{n}(\delta)^c\cap C_\textbf{n}}\right]
\le  
\P\left(D_\textbf{n}(\delta)^c\right)^{\frac{\zeta-q}{\zeta}}C^q\left(2^{\zeta-1}(1+\gamma^{\zeta/2})(k+1)^{\zeta}p^{\zeta-1}\tilde{M}(\zeta)\right)^{\frac{q}{\zeta}}.
\end{align}
\\
\textbf{Step 4}: Split the $G$-part into independent increments\\
For each $i\in \{1,...,k\}$, define for $j=1,...,n_{A_i}$ and $u,v\in\{1,...,p\}$, the matrices with entries
$$W(j)_{u,v}^{n_{A_i},A_i}=(1+\gamma) w_i^2\mathbb{X}^{n_{A_i},A_i}_{u,j}\mathbb{X}^{n_{A_i},A_i}_{v,j}-(1+\gamma) w_i^2\E\left[(X^{A_i})^TX^{A_i}\right]_{u,v},$$
and for $i=0$,
$$ W(j)_{u,v}^{n_{A_i},A_i}=(1-\gamma)\mathbb{X}^{n_{A_i},A_i}_{u,j}\mathbb{X}^{n_{A_i},A_i}_{v,j}-(1-\gamma)\E\left[ (X^{A_i})^TX^{A_i}\right]_{u,v}.$$
Then $\{W(j)_{u,v}^{n_{A_i},A_i}\}_{j=1}^{n_{A_i}}$ are i.i.d. and $\E\left[W(j)_{u,v}^{n_{A_i}}\right]=0$, for every $u,v\in\{1,...,p\}$. Moreover 
$$\frac{1}{n_{A_i}}\sum_{i=0}^k\sum_{j=1}^{n_{A_i}}W(j)^{n_{A_i},A_i}=\hat{G}_+(\textbf{n})+\gamma\hat{G}_\Delta(\textbf{n})- G_++\gamma G_\Delta.$$ 
By the Radon inequality we have for $a>1$
\begin{align}\label{Gnormbound}
&\n \hat{G}_+(\textbf{n})+\gamma\hat{G}_\Delta(\textbf{n})- G_++\gamma G_\Delta\n_{\mathit{l}^2}^a 
=
\n \sum_{i=0}^k\frac{1}{n_{A_i}}\sum_{j=1}^{n_{A_i}}W(j)^{n_{A_i},A_i} \n_{\mathit{l}^2}^a
\le
(k+1)^{a-1}\sum_{i=0}^k\n \frac{1}{n_{A_i}}\sum_{j=1}^{n_{A_i}}W(j)^{n_{A_i},A_i} \n_{\mathit{l}^2}^2
\nonumber
\\
&=
(k+1)^{a-1}\sum_{i=0}^k\n \frac{1}{n_{A_i}}\sum_{j=1}^{n_{A_i}}W(j)^{n_{A_i},A_i} \n_{1,1}^a
=(k+1)^{a-1}\sum_{i=0}^k\left(\sum_{u=1}^p\sum_{v=1}^p  \left| \sum_{j=1}^{n_{A_i}}\frac{1}{n_{A_i}} W(j)_{u,v}^{n_{A_i},A_i} \right|\right)^a
\nonumber
\\
&\le
(k+1)^{a-1}p^{2(a-1)} \sum_{i=0}^k\sum_{u=1}^p\sum_{v=1}^p\left| \sum_{j=1}^{n_{A_i}}\frac{1}{n_{A_i}}W(j)_{u,v}^{n_{A_i},A_i} \right|^a.
\end{align}
\\
\textbf{Step 5}: Apply a Marcinkiewicz-Zygmund type inequality to the increments\\
The following bound found is a special case of the main result in \cite{rio2009moment} which applies to any independent zero mean collection of variables $U_1,...,U_n$ with finite a:th moment,
$$\E\left[\left|\sum_{j=1}^n U_j\right|^a\right]^{2/a}\le (a-1)\sum_{j=1}^n\E\left[\left|U_j\right|^a\right]^{\frac{2}{a}}.$$ 
Raising both sides above to the power $a/2$ gives
$$\E\left[\left|\sum_{j=1}^n U_j\right|^a\right]\le (a-1)^{a/2}\left(\sum_{j=1}^n\E\left[\left|U_j\right|^a\right]^{\frac{2}{a}}\right)^{\frac{a}{2}}.$$
Combining this with the Radon inequality yields
\begin{align}\label{MaZy}
\E\left[\left|\sum_{j=1}^n U_j\right|^a\right]\le n^{\frac{a}{2}-1}(a-1)^{a/2}\sum_{j=1}^n\E\left[\left|U_j\right|^a\right].
\end{align}
\\
\textbf{Step 6}: Finalize the the local $G$-bound\\
Since $L^{\zeta''}(\P)\subseteq L^{\zeta',w}(\P)$ for all $1\le \zeta''<\zeta'$, if we let $\psi(\zeta')=q\wedge(2(\alpha\zeta'-1))$ (noting that $4(\alpha\zeta'-1)<\zeta'$) it follows that $M(2\psi(\zeta'))<\infty$.
Applying \eqref{MaZy} to \eqref{Gnormbound}, and the fact that 
$$\E\left[\left| W(j)_{u,v}^{n_{A_i},A_i} \right|^a\right]
\le (1+\gamma)^a\E\left[\left|X^{A_i}(u)X^{A_i}(v)-\E\left[X^{A_i}(u)X^{A_i}(v)\right]\right|^a\right]
\le (1+\gamma)^a2^{a+1}\E\left[\left|X^{A_i}(u)X^{A_i}(v)\right|^a\right]$$ 
it then follows that
\begin{align}\label{Gexp1}
&\E\left[\n \hat{G}_+(\textbf{n})+\gamma\hat{G}_\Delta(\textbf{n})- G_++\gamma G_\Delta\n_{\mathit{l}^2}^q 1_{ D_\textbf{n}(\delta)}\right]
\nonumber
\\
&\le
\left(\delta\n \left(G_++\gamma G_\Delta\right)^{-1}\n_{\mathit{l}^2}\right)^{q-(q\wedge(\alpha\zeta'-1))}\E\left[\n \hat{G}_+(\textbf{n})+\gamma\hat{G}_\Delta(\textbf{n})- G_++\gamma G_\Delta\n_{\mathit{l}^2}^{\psi(\zeta')} \right]
\nonumber
\\
&\le  
\left(\delta\n \left(G_++\gamma G_\Delta\right)^{-1}\n_{\mathit{l}^2}\right)^{q-\psi(\zeta')}(k+1)^{\psi(\zeta')-1}p^{2(\psi(\zeta')-1)} \sum_{i=0}^k\sum_{u=1}^p\sum_{v=1}^p\E\left[\left| \sum_{j=1}^{n_{A_i}}\frac{1}{n_{A_i}}W(j)_{u,v}^{n_{A_i},A_i} \right|^{\psi(\zeta')}\right]
\nonumber
\\
&\le 
\left(\delta\n \left(G_++\gamma G_\Delta\right)^{-1}\n_{\mathit{l}^2}\right)^{q-\psi(\zeta')}(k+1)^{\psi(\zeta')-1}p^{2(\psi(\zeta')-1)} (\psi(\zeta')-1)^{\left(\psi(\zeta')\right)/2}\frac{2^{q+1}}{n_{A_i}^{\frac{\psi(\zeta')}{2}+1}}\sum_{i=0}^k\sum_{u=1}^p\sum_{v=1}^p\sum_{j=1}^{n_{A_i}}\E\left[\left|W(j)_{u,v}^{n_{A_i},A_i}\right|^{\psi(\zeta')}\right]\nonumber
\\
&\le
\left(\delta\n \left(G_++\gamma G_\Delta\right)^{-1}\n_{\mathit{l}^2}\right)^{q-\psi(\zeta')}(k+1)^{\psi(\zeta')-1}p^{2(\psi(\zeta')-1)} (\psi(\zeta')-1)^{\left(\psi(\zeta')\right)/2}(1+\gamma)^{\psi(\zeta')}2^{q+1}
\nonumber
\\
&\times
\sum_{i=0}^k\sum_{u=1}^p\sum_{v=1}^p\sum_{j=1}^{n_{A_i}}\frac{1}{n_{A_i}^{\frac{\psi
(\zeta')}{2}+1}}\E\left[\left|X^{A_i}(u)X^{A_i}(v)\right|^{\psi(\zeta')}\right].
\end{align}
Bounding the sum above with the Cauchy-Schwartz inequality yields,
\begin{align}\label{Wexpand}
\sum_{i=0}^k\sum_{u=1}^p\sum_{v=1}^p\sum_{j=1}^{n_{A_i}}\frac{1}{n_{A_i}^{\frac{\psi
(\zeta')}{2}+1}}\E\left[\left|X^{A_i}(u)X^{A_i}(v)\right|^{\psi(\zeta')}\right]
&\le
\sum_{i=0}^k\sum_{u=1}^p\sum_{v=1}^p\frac{1}{n_{A_i}^{\frac{\psi
(\zeta')}{2}}}\E\left[\left|X^{A_i}(u)\right|^{2q\wedge(4(\alpha\zeta'-1))}\right]^{\frac 12}\E\left[\left|X^{A_i}(v)\right|^{2q\wedge(4(\alpha\zeta'-1))}\right]^{\frac 12} 
\nonumber
\\
&\le 
p^2M(2q\wedge(4(\alpha\zeta'-1)))\sum_{i=0}^k\frac{1}{n_{A_i}^{\frac{\psi
(\zeta')}{2}}}
\nonumber
\\
&\le
p^2M(2q\wedge(4(\alpha\zeta'-1)))\frac{k}{\left(n_{A_0}\wedge...\wedge n_{A_k}\right)^{\frac{\psi
(\zeta')}{2}}}.
\end{align}
Plugging \eqref{Wexpand} back into \eqref{Gexp1} we get
\begin{align}\label{Gexp}
&\E\left[\n \hat{G}_+(\textbf{n})+\gamma\hat{G}_\Delta(\textbf{n})- G_++\gamma G_\Delta\n_{\mathit{l}^2}^q 1_{ D_\textbf{n}(\delta)}\right]
\nonumber
\\
&\le
\frac{\delta^{q-\psi(\zeta')}(k+1)^{\psi(\zeta')}p^{2\psi(\zeta')} (\psi(\zeta')-1)^{\psi(\zeta')/2}(1+\gamma)^{\psi(\zeta')}M(2q\wedge(4(\alpha\zeta'-1)))}{n_{A_i}^{\frac{\psi
(\zeta')}{2}}}.
\end{align}
\\
\textbf{Step 7}: Correspondingly bound the $Z$-part\\
Similarly to $W$, for each $i\in \{1,...,k\}$, $j=1,...,n$ and $u\in\{1,...,p\}$ we define the vectors with entries
$$V(j)_{u}^{n_{A_i},A_i}=(1+\gamma )w_i^2\mathbb{X}^{n_{A_i},A_i}_{u,j}\mathbb{Y}^{n_{A_i},A_i}_{j}-\E\left[(1+\gamma )w_i^2(X^{A_i})^TY^{A_i}\right]_{u},$$
and
$$V(j)_{u}^{n_{A_0},A_0}=(1-\gamma)\mathbb{X}^{n_{A_0},A_0}_{u,j}\mathbb{Y}^{n_{A_0},A_0}_{j}-\E\left[(1-\gamma)(X^{A_0})^TY^{A_0}\right]_{u}.$$
Then $\{V(j)_{u}^{n_{A_i},A_i}\}_{j=1}^{n_{A_i}}$ are independent and $\E\left[V(j)_{u}^{n_{A_i}}\right]=0$, for every $u\in\{1,...,p\}$.
We also have
$$\frac{1}{n_{A_i}}\sum_{i=0}^k\sum_{j=1}^nV(j)^{n_{A_i},A_i}=\hat{Z}_+(\textbf{n})+\gamma\hat{Z}_\Delta(\textbf{n})- \left( Z_++\gamma Z_\Delta \right).$$ 
Since $\n.\n_{\mathit{l}^2}\le \n.\n_{\mathit{l}^1}$, analogously to \eqref{Gnormbound},
\begin{align}\label{Znormbound}
&\n \hat{Z}_+(\textbf{n})+\gamma\hat{Z}_\Delta(\textbf{n})- Z_++\gamma Z_\Delta\n_{\mathit{l}^2}^q 
=
\n \sum_{i=0}^k\sum_{j=1}^{n_{A_i}}V(j)^{n_{A_i},A_i} \n_{\mathit{l}^2}^q
\le
(k+1)^{q-1}\sum_{i=0}^k\n \sum_{j=1}^{n_{A_i}}V(j)^{n_{A_i},A_i} \n_{\mathit{l}^2}^q
\nonumber
\\
&\le
(k+1)^{q-1}\sum_{i=0}^k\n \sum_{j=1}^{n_{A_i}}V(j)^{n_{A_i},A_i} \n_{\mathit{l}^1}^q
=(k+1)^{q-1}\sum_{i=0}^k\left(\sum_{u=1}^p  \left| \sum_{j=1}^{n_{A_i}}\frac{1}{n_{A_i}}V(j)_{u}^{n_{A_i},A_i} \right|\right)^q
\le
(k+1)^{q-1}p^{q-1} \left| \sum_{j=1}^{n_{A_i}}\frac{1}{n_{A_i}}V(j)_{u}^{n_{A_i},A_i} \right|^q.
\end{align}
Since
$$\E\left[\left|\sum_{j=1}^{n_{A_i}}\frac{1}{n_{A_i}}V(j)_{u}^{n_{A_i},A_i} \right|^q\right]
\le 
\sum_{j=1}^{n_{A_i}}\frac{(q-1)^{q/2}2^{q+1}(1+\gamma)^q}{n_{A_i}^{q/2+1}}\E\left[\left|V(j)_{u}^{n_{A_i},A_i}\right|^q\right] $$
Analogously to \eqref{Gexp} we get,
\begin{align}\label{Zexp}
\E\left[\n \hat{Z}_+(\textbf{n})+\gamma\hat{Z}_\Delta(\textbf{n})- Z_++\gamma Z_\Delta\n_{\mathit{l}^2}^q\right]
\le 
(k+1)^{q}p^{q}(q-1)^{q/2}(1+\gamma)^q\frac{1}{\left(n_{A_0}\wedge...\wedge n_{A_k}\right)^{q/2}}\tilde{M}(q).
\end{align}
\\
\textbf{Step 8}: Plug-in our $Z$ and local $G$-bound\\
Returning to \eqref{expbound}, we plug in \eqref{Gexp} and \eqref{Zexp}, to find
\small
\begin{align}\label{penul}
&\E\left[\n\hat{\beta}_\gamma(\textbf{n}) -\beta_\gamma  \n_{\mathit{l}^2}^q1_{ D_\textbf{n}(\delta)\cap C_\textbf{n}}\right]
\le 
2^{q-1}\n\left(G_++\gamma G_\Delta\right)^{-1} \n_{\mathit{l}^2}^q \frac{1}{n_{A_i}^{q/2}}(k+1)^{q}p^q(1+\gamma)^q(q-1)^{q/2}\tilde{M}(q)
\nonumber
\\
&+2^{q-1}(1+\delta)^q\n \left(G_++\gamma G_\Delta\right)^{-1} \n_{\mathit{l}^2}^{2q} 
\left( 
\frac{\n\left( Z_++\gamma Z_\Delta \right)\n_{\mathit{l}^2}^q\delta^{q-\psi(\zeta')}(k+1)^{\psi(\zeta')}p^{2\psi(\zeta')} (\psi(\zeta')-1)^{\psi(\zeta')/2}(1+\gamma)^{\psi(\zeta')}M(2q\wedge(4(\alpha\zeta'-1)))}{n_{A_i}^{\frac{\psi
(\zeta')}{2}}}
\right.
\nonumber
\\
&\left.+
\delta^q \n\left(G_++\gamma G_\Delta\right)^{-1}\n_{\mathit{l}^2}^q
\frac{\delta^{q-\psi(\zeta')}(k+1)^{\psi(\zeta')}p^{2\psi(\zeta')} (\psi(\zeta')-1)^{\psi(\zeta')/2}(1+\gamma)^{\psi(\zeta')}M(2q\wedge(4(\alpha\zeta'-1)))}{n_{A_i}^{\frac{\psi
(\zeta')}{2}}}\right)
\nonumber
\\
&\le
\frac{2^{q+1}}{n_{A_i}^{\frac{\psi
(\zeta')}{2}}}(k+1)^{q}p^{q\vee(2\psi(\zeta'))}(1+\gamma)^q(1+\delta)^{3q}
(M(2\psi(\zeta'))\vee\tilde{M}(q)) (\psi(\zeta')-1)^{\psi(\zeta')/2}
\times\left(1\vee\n \left(G_++\gamma G_\Delta\right)^{-1} \n_{\mathit{l}^2}\right)^{3q}\left(1\vee\n Z_++\gamma Z_\Delta \n_{\mathit{l}^2}\right)^{q}.
\end{align}
\normalsize
\\
\textbf{Step 8}: Bound the measure of the set $D_{\textbf{n}}(\delta)^c$ and bound the number of necessary samples $N$\\
We may recall from the proof of Theorem 1 (set $c'=\delta\n\left(G_+ +\gamma G_\Delta\right)^{-1}\n_{\mathit{l}^2}$ in \eqref{PDnc}) that
\begin{align}\label{Dnc1}
\P\left( C_\textbf{n}^c\right)\le\P\left( D_{\textbf{n}}(\delta)^c\right)\le p\sum_{i=0}^ke^{-\frac{\left(\left(\delta\n \left(G_++\gamma G_\Delta\right)^{-1}\n_{\mathit{l}^2} -(1+\gamma )h_i(n_{A_i}^\alpha)\right)_+\right)^2}{2p^2(k+1)^2}n_{A_i}^{1-4\alpha}}+1-\prod_{i=1}^k\left(F_{\left|X^{A_i}(1)\right|,...,\left|X^{A_i}(p)\right|}(n_{A_i}^\alpha,...,n_{A_i}^\alpha)\right)^{n_{A_i}},
\end{align}
where $h_i(x)=\left(\sum_{u=1}^p\sqrt{\E\left[ |X^{A_i}(u)|^21_{(E_i'(x))^c}\right]}\right)^2$ and $E_i'(x)=\left\{|X^{A_i}(u)|\le x,1\le u \le p\right\}$. By the vector space property of $L^{\zeta,w}(\P)$ it follows that $\sum_{l=1}^p|X^{A_i}(l)|\in L^{\zeta,w}(\P)$ which in turn implies $|X^{A_i}(1)|\vee...\vee |X^{A_i}(p)|\in L^{\zeta,w}(\P)$, i.e.
\begin{align*}
\P\left(|X^{A_i}(1)|\vee...\vee |X^{A_i}(p)|>x\right)&=
\P\left(\bigcup_{l=1}^p\left\{|X^{A_i}(l)|>x\right\}\right)
\\
&\le
\sum_{l=1}^p\P\left( \left|X^{A_i}(l)\right|>x\right)
\le \frac{pM_w(\zeta')}{x^{\zeta'}}.
\end{align*}
If we let $\sigma_w(\zeta'):= pM_w(\zeta')$ then the above inequality is equivalent to
\begin{align}\label{weak}
F_{|X^{A_i}(1)|,...,|X^{A_i}(p)|}(x,...,x)\ge 1-\frac{\sigma_w(\zeta')}{x^{\zeta'}}.
\end{align}
This leads to
\begin{align}\label{cdfprod}
&\prod_{i=0}^k\left(F_{\left|X^{A_i}(1)\right|,...,\left|X^{A_i}(p)\right|}(n_{A_i}^\alpha,...,n_{A_i}^\alpha)\right)^{n_{A_i}}
\ge  
\prod_{i=0}^k\left(1-\frac{\sigma_w(\zeta')}{n_{A_i}^{\alpha\zeta'}}\right)^{n_{A_i}} 1_{n_{A_i}\ge \sigma_w(\zeta')^{\frac{1}{\alpha\zeta'}}}, 
\end{align}
which converges to 1 (as $\bigwedge_{i=0}^kn_{A_i}\to\infty$) for all $n_{A_i}\ge \sigma_w(\zeta')^{\frac{1}{\alpha\zeta'}}$. For $n_{A_i}\ge \sigma_w(\zeta')^{\frac{1}{\alpha\zeta'}}$,
\begin{align}\label{cdfrate}
&\left| \prod_{i=0}^k\left(F_{\left|X^{A_i}(1)\right|,...,\left|X^{A_i}(p)\right|}(n_{A_i}^\alpha,...,n_{A_i}^\alpha)\right)^{n_{A_i}}
-1 \right|
=1-e^{\sum_{i=0}^k n_{A_i} \log\left(F_{\left|X^{A_i}(1)\right|,...,\left|X^{A_i}(p)\right|}(n_{A_i}^\alpha,...,n_{A_i}^\alpha)\right)}
\nonumber
\\
&\le 1-e^{\sum_{i=0}^k n_{A_i} \log\left(1-\frac{\sigma_w(\zeta')}{n_{A_i}^{\alpha\zeta'}}\right)}
\le
1-e^{\sum_{i=0}^k n_{A_i} \left(-\frac{\sigma_w(\zeta')}{n_{A_i}^{\alpha\zeta'}}
\right)}
\le
\sum_{i=0}^k\frac{\sigma_w(\zeta')}{n_{A_i}^{\alpha\zeta'-1}}
\le \frac{(k+1)\sigma_w(\zeta')}{\left(n_{A_0}\wedge...\wedge n_{A_k}\right)^{\alpha\zeta'-1}}, 
\end{align}
where we did a first order Taylor expansion of the logarithm and the exponential function. For the purpose of dealing with the first term in \eqref{Dnc1}, define $N'_{A_i}=\min\left\{n_{A_i}\in\N: (1+\gamma)h_i(n_{A_i}^\alpha)<\frac{\delta\n \left(G_++\gamma G_\Delta\right)^{-1}\n_{\mathit{l}^2}}{2} \right\}$. We have for $n_{A_i}\ge \sigma_w(\zeta')^{\frac{1}{\alpha\zeta'}}$ and $1<\eta'<\zeta'/2$,

\begin{align}\label{hbound}
h_i(n_{A_i}^\alpha)
&=\left(\sum_{u=1}^p\sqrt{\E\left[ |X^{A_i}(u)|^21_{(E_i'(n_{A_i}^\alpha))^c}\right]}\right)^2
=\left(\sum_{u=1}^p\sqrt{\int_0^\infty\P\left(|X^{A_i}(u)|^21_{(E_i'(n_{A_i}^\alpha))^c}\ge x\right)dx}\right)^2
\nonumber
\\
&\le
\left(\sum_{u=1}^p\sqrt{\int_0^\infty\P\left(|X^{A_i}(u)|\ge \sqrt{x}\right)^{\frac{1}{\eta'}}\P\left((E_i'(n_{A_i}^\alpha))^c\right)^{\frac{\eta'-1}{\eta'}} dx}\right)^2
\nonumber
\\
&\le 
\left(\sum_{u=1}^p\sqrt{\int_1^\infty\P\left(|X^{A_i}(u)|\ge \sqrt{x}\right)^{\frac{1}{\eta'}}\P\left((E_i'(n_{A_i}^\alpha))^c\right)^{\frac{\eta'-1}{\eta'}} dx+\P\left((E_i'(n_{A_i}^\alpha))^c\right)}\right)^2
\nonumber
\\
&\le
\left(\sum_{u=1}^p\sqrt{\int_1^\infty \frac{M_w(\zeta')^{\frac{1}{\eta'}}}{x^{\frac{\zeta'}{2\eta'}}} \P\left((E_i'(n_{A_i}^\alpha))^c\right)^{\frac{\eta'-1}{\eta'}} dx+\P\left((E_i'(n_{A_i}^\alpha))^c\right)}\right)^2 \le \left( p\frac{M_w(\zeta')^{\frac{1}{\eta'}}}{\frac{\zeta'}{2\eta'}-1}+1\right)\P\left((E_i'(n_{A_i}^\alpha))^c\right)^{\frac{\eta'-1}{\eta'}}
\nonumber
\\
&\le
\left( p\frac{M_w(\zeta')^{\frac{1}{\eta'}}}{\frac{\zeta'}{2\eta'}-1}+1\right) \left(1-F_{\left|X^{A_i}(1)\right|,...,\left|X^{A_i}(p)\right|}\left(n_{A_i}^\alpha,...,n_{A_i}^\alpha\right)\right)^{\frac{\eta'-1}{\eta'}}
\le  \left( p\frac{M_w(\zeta')^{\frac{1}{\eta'}}}{\frac{\zeta'}{2\eta'}-1}+1\right)\left(\frac{\sigma_w(\zeta')}{n_{A_i}^{\alpha\zeta'}}\right)^{\frac{\eta'-1}{\eta'}}.
\end{align}
Solving $(1+\gamma)h_i(n_{A_i}^\alpha)<\frac{\delta\n \left(G_++\gamma G_\Delta\right)^{-1}\n_{\mathit{l}^2}}{2}$ for $n_{A_i}$, while also taking into account that $n_{A_i}$ leads to
\begin{align}\label{NPrime1}
N'_{A_i}\le 1\vee\left(\left(\frac{2(1+\gamma)}{\delta\n \left(G_++\gamma G_\Delta\right)^{-1}\n_{\mathit{l}^2}}\right)^{\frac{1}{\alpha\zeta'\frac{\eta'-1}{\eta'}}}\left(\frac{pM_w(\zeta')^{\frac{1}{\eta'}}}{\frac{\zeta'}{2\eta'-1}}+1\right)^{\frac{1}{\alpha\zeta'\frac{\eta'-1}{\eta'}}}\right)\sigma_w(\zeta')^{\frac{1}{\alpha\zeta'}}.
\end{align}
Plugging \eqref{cdfrate} into \eqref{Dnc1} yields
\begin{align}\label{sumrate11}
&\P\left(D_\textbf{n}(\delta)^c\right)\le p\sum_{i=0}^ke^{-\frac{\left(\left(\delta\n \left(G_++\gamma G_\Delta\right)^{-1}\n_{\mathit{l}^2} -(1+\gamma )h_i(n_{A_i}^\alpha)\right)_+\right)^2}{2p^2(k+1)^2}n_{A_i}^{1-4\alpha}}+\frac{\sigma_w(\zeta')}{\left(n_{A_0}\wedge...\wedge n_{A_k}\right)^{\alpha\zeta'-1}}.
\end{align}
For $a,b,c\in\R^+$ we have the following inequality, $e^{-ax^b}b^{-\frac{c}{b}}\le x^{-c}$ for all $x\ge \left(\frac{c}{ab}\right)^{\frac{1}{b}}$ which implies that $e^{-ax^b}b^{-\frac{c}{b}}\left(\frac{c}{ab}\right)^{\frac{-c}{b}}=e^{-ax^b}\left(\frac{c}{a}\right)^{\frac{-c}{b}}\le x^{-c}$ for all $x>0$ which is equivalent to $e^{-ax^b}\le\left(\frac{c}{a}\right)^{\frac{c}{b}}x^{-c}$ for all $x>0$. We utilize this fact in the following bound (setting $a=\frac{\delta^2\n \left(G_++\gamma G_\Delta\right)^{-1}\n_{\mathit{l}^2}^2}{4p^2(k+1)^2}$, $b=1-4\alpha$ and $c=\alpha\zeta'-1$)

\begin{align}\label{pexp01}
&p\sum_{i=0}^ke^{-\frac{\left(\left(\delta\n \left(G_++\gamma G_\Delta\right)^{-1}\n_{\mathit{l}^2} -(1+\gamma )h_i(n_{A_i}^\alpha)\right)_+\right)^2}{2p^2(k+1)^2}n_{A_i}^{1-4\alpha}}
<
p\sum_{i=0}^k1_{n_{A_i}<N'_{A_i}}+ p\sum_{i=0}^ke^{-\frac{\delta^2\n \left(G_++\gamma G_\Delta\right)^{-1}\n_{\mathit{l}^2}^2}{4p^2(k+1)^2}n_{A_i}^{1-4\alpha}}1_{n_{A_i}\ge N'_{A_i}}\nonumber
\\
&\le 
p\sum_{i=0}^k\left(\frac{N'_{A_i}}{n_{A_0}\wedge...\wedge n_{A_k}}\right)^{\alpha\zeta'-1}+p\sum_{i=0}^ke^{-\frac{\delta^2\n \left(G_++\gamma G_\Delta\right)^{-1}\n_{\mathit{l}^2}^2}{4p^2(k+1)^2}n_{A_i}^{1-4\alpha}}
\le 
p(k+1)\frac{(N'_{A_i})^{\alpha\zeta'-1} +\left(\frac{\alpha\zeta'-1}{\frac{\delta^2\n \left(G_++\gamma G_\Delta\right)^{-1}\n_{\mathit{l}^2}^2}{4p^2(k+1)^2}}\right)^{\frac{\alpha\zeta'-1}{1-4\alpha}}}{\left(n_{A_0}\wedge...\wedge n_{A_k}\right)^{\alpha\zeta'-1}}
\nonumber
\\
&=p(k+1)\frac{(N'_{A_i})^{\alpha\zeta'-1} +\left(\frac{(\alpha\zeta'-1)4p^2(k+1)^2}{\delta^2\n \left(G_++\gamma G_\Delta\right)^{-1}\n_{\mathit{l}^2}^2}\right)^{\frac{\alpha\zeta'-1}{1-4\alpha}}}{\left(n_{A_0}\wedge...\wedge n_{A_k}\right)^{\alpha\zeta'-1}}.
\end{align}
If we assume that 
\begin{align}\label{Nw1}
&N\ge \left(\frac{(p+1)(M_w(\zeta')\vee 1)(k+1)^{\zeta'}}{1-\left(\frac{3}{4}\right)^{\frac{1}{k+1}}}\right)^{\frac{1}{\alpha\zeta'-1}}
\end{align}
(which is larger than $\sigma_w(\zeta')^{\frac{1}{\alpha\zeta'}}$ so \eqref{cdfrate} applies) it will imply that $\left(1-\frac{\sigma_w(\zeta')}{\left(n_{A_0}\wedge...\wedge n_{A_k}\right)^{\alpha\zeta'}}\right)^{k+1}\ge \frac 34,$
while if we also assume 
\tiny
\begin{align}\label{Nw2}
N\ge \left(4p(k+1)\left( 1\vee\left(\left(\frac{2(1+\gamma)}{\delta\n \left(G_++\gamma G_\Delta\right)^{-1}\n_{\mathit{l}^2}}\right)^{\frac{1}{\alpha\zeta'\frac{\eta'-1}{\eta'}}}\left(\frac{pM_w(\zeta')^{\frac{1}{\eta'}}}{\frac{\zeta'}{2\eta'-1}}+1\right)^{\frac{1}{\alpha\zeta'\frac{\eta'-1}{\eta'}}}\right)\sigma_w(\zeta')^{\frac{1}{\alpha\zeta'}}\right)^{\alpha\zeta'-1} 
+
4p(k+1)\left(\frac{(\alpha\zeta'-1)4p^2(k+1)^2}{\delta^2\n \left(G_++\gamma G_\Delta\right)^{-1}\n_{\mathit{l}^2}^2}\right)^{\frac{\alpha\zeta'-1}{1-4\alpha}}\right)^{\frac{1}{\alpha\zeta'-1}}
\end{align} 
\normalsize
this will imply that 
\small
$$
\frac{p(k+1)\left( 1\vee\left(\left(\frac{2(1+\gamma)}{\delta\n \left(G_++\gamma G_\Delta\right)^{-1}\n_{\mathit{l}^2}}\right)^{\frac{1}{\alpha\zeta'\frac{\eta'-1}{\eta'}}}\left(\frac{pM_w(\zeta')^{\frac{1}{\eta'}}}{\frac{\zeta'}{2\eta'-1}}+1\right)^{\frac{1}{\alpha\zeta'\frac{\eta'-1}{\eta'}}}\right)\sigma_w(\zeta')^{\frac{1}{\alpha\zeta'}}\right)^{\alpha\zeta'-1} 
+
p(k+1)\left(\frac{(\alpha\zeta'-1)4p^2(k+1)^2}{\delta^2\n \left(G_++\gamma G_\Delta\right)^{-1}\n_{\mathit{l}^2}^2}\right)^{\frac{\alpha\zeta'-1}{1-4\alpha}}}{\left(n_{A_0}\wedge...\wedge n_{A_k}\right)^{\alpha\zeta'-1}}\le \frac 14,$$
\normalsize
for $n_{A_i}\ge N$. Hence, if $n_{A_i}\ge N$
\small
\begin{align}\label{PCn1}
&\P\left( C_\textbf{n}\right)\ge\P\left( D_{\textbf{n}}(\delta)\right)=1-\P\left( D_{\textbf{n}}(\delta)^c\right)\ge 
\prod_{i=0}^k\left(1-\frac{\sigma_w(\zeta')}{n_{A_i}^{\alpha\zeta'}}\right)
-p\sum_{i=0}^ke^{-\frac{\left(\left(\delta\n \left(G_++\gamma G_\Delta\right)^{-1}\n_{\mathit{l}^2} -(1+\gamma )h_i(n_{A_i}^\alpha)\right)_+\right)^2}{2p^2(k+1)^2}n_{A_i}^{1-4\alpha}}\nonumber
\\
&\ge \left(1-\frac{\sigma_w(\zeta')}{\left(n_{A_0}\wedge...\wedge n_{A_k}\right)^{\alpha\zeta'}}\right)^{k+1}\nonumber
\\
&-
\frac{p(k+1)\left( 1\vee\left(\left(\frac{2(1+\gamma)}{\delta\n \left(G_++\gamma G_\Delta\right)^{-1}\n_{\mathit{l}^2}}\right)^{\frac{1}{\alpha\zeta'\frac{\eta'-1}{\eta'}}}\left(\frac{pM_w(\zeta')^{\frac{1}{\eta'}}}{\frac{\zeta'}{2\eta'-1}}+1\right)^{\frac{1}{\alpha\zeta'\frac{\eta'-1}{\eta'}}}\right)\sigma_w(\zeta')^{\frac{1}{\alpha\zeta'}}\right)^{\alpha\zeta'-1} 
+
p(k+1)\left(\frac{(\alpha\zeta'-1)4p^2(k+1)^2}{\delta^2\n \left(G_++\gamma G_\Delta\right)^{-1}\n_{\mathit{l}^2}^2}\right)^{\frac{\alpha\zeta'-1}{1-4\alpha}}
}{\left(n_{A_0}\wedge...\wedge n_{A_k}\right)^{\alpha\zeta'-1}}
\nonumber
\\
&\ge \frac 34-\frac 14= \frac 12.
\end{align}
\normalsize
Plugging in \eqref{pexp01} and \eqref{NPrime1} into \eqref{sumrate11}, which is then plugged into \eqref{sumrate11} and \eqref{Dnc1} successively gives us
\scriptsize
\begin{align}
&\P(D_\textbf{n}(\delta)^c)\le
\nonumber
\\
&\frac{p(k+1)\left( 1\vee\left(\left(\frac{2(1+\gamma)}{\delta\n \left(G_++\gamma G_\Delta\right)^{-1}\n_{\mathit{l}^2}}\right)^{\frac{1}{\alpha\zeta'\frac{\eta'-1}{\eta'}}}\left(\frac{pM_w(\zeta')^{\frac{1}{\eta'}}}{\frac{\zeta'}{2\eta'-1}}+1\right)^{\frac{1}{\alpha\zeta'\frac{\eta'-1}{\eta'}}}\right)\sigma_w(\zeta')^{\frac{1}{\alpha\zeta'}}\right)^{\alpha\zeta'-1}
+
p(k+1)\left(\frac{(\alpha\zeta'-1)4p^2(k+1)^2}{\delta^2\n \left(G_++\gamma G_\Delta\right)^{-1}\n_{\mathit{l}^2}^2}\right)^{\frac{\alpha\zeta'-1}{1-4\alpha}}
+
\sigma_w(\zeta')}{\left(n_{A_0}\wedge...\wedge n_{A_k}\right)^{\alpha\zeta'-1}}
\end{align}
\normalsize
\\
\textbf{Step 9}: Plug in all of our bounds and sum all the terms\\
Then we plug this back into \eqref{Zxibnd1},
\tiny
\begin{align*}
&\E\left[\n\hat{\beta}_\gamma(\textbf{n})   \n_{\mathit{l}^2}^q1_{C_\textbf{n}\cap D_\textbf{n}(\delta)^c}\right]
\le
(1\vee C)^q\left(2^{\zeta-1}(1+\gamma^{\zeta})(k+1)^{\zeta}p^{\zeta-1}\tilde{M}(\zeta)\right)^{\frac{q}{\zeta}}
\nonumber
\\
&\times \left(\frac{p(k+1)\left( 1\vee\left(\left(\frac{2(1+\gamma)}{\delta\n \left(G_++\gamma G_\Delta\right)^{-1}\n_{\mathit{l}^2}}\right)^{\frac{1}{\alpha\zeta'\frac{\eta'-1}{\eta'}}}\left(\frac{pM_w(\zeta')^{\frac{1}{\eta'}}}{\frac{\zeta'}{2\eta'-1}}+1\right)^{\frac{1}{\alpha\zeta'\frac{\eta'-1}{\eta'}}}\right)\sigma_w(\zeta')^{\frac{1}{\alpha\zeta'}}\right)^{\alpha\zeta'-1}
+
p(k+1)\left(\frac{(\alpha\zeta'-1)4p^2(k+1)^2}{\delta^2\n \left(G_++\gamma G_\Delta\right)^{-1}\n_{\mathit{l}^2}^2}\right)^{\frac{\alpha\zeta'-1}{1-4\alpha}}
+
\sigma_w(\zeta')}{\left(n_{A_0}\wedge...\wedge n_{A_k}\right)^{\alpha\zeta'-1}} \right)^{1-\frac{q}{\zeta}}
\end{align*}
\normalsize
We thus have that for $n_{A_i}\ge N$,
\small
\begin{align}\label{hejsan_}
&Var_q\left(\hat{\beta}_\gamma(\textbf{n})\| C_\textbf{n}\right)\le 2\E\left[\n\hat{\beta}_\gamma(\textbf{n})1_{C_\textbf{n}} -\E\left[\hat{\beta}_\gamma(\textbf{n})1_{C_\textbf{n}} \right] \n_{\mathit{l}^2}^q\right] 
\le 
2^{q}\E\left[\n\hat{\beta}_\gamma(\textbf{n}) -\beta_\gamma \n_{\mathit{l}^2}^q1_{C_\textbf{n}} \right]+
2^{q}\n\beta_\gamma \n_{\mathit{l}^2}^q\P\left(D_\textbf{n}(\delta)^c\right)
\nonumber
\\
&\le
2^{q}\E\left[\n\hat{\beta}_\gamma(\textbf{n}) -\beta_\gamma \n_{\mathit{l}^2}^q1_{C_\textbf{n}\cap D_\textbf{n}(\delta)} \right]
+
2^{q}\E\left[\n\hat{\beta}_\gamma(\textbf{n})-\beta_\gamma\n_{\mathit{l}^2}^q1_{C_\textbf{n}\cap D_\textbf{n}(\delta)^c} \right]
+
2^{q}\n\beta_\gamma \n_{\mathit{l}^2}^q\P\left(D_\textbf{n}(\delta)^c\right)
\nonumber
\\
&\le
\frac{2^{2q+1}}{n_{A_i}^{\frac{\psi
(\zeta')}{2}}}(k+1)^{q}p^{q\vee(2\psi(\zeta'))}(1+\gamma)^q(1+\delta)^{3q}
(M(2\psi(\zeta'))\vee\tilde{M}(q)) (\psi(\zeta')-1)^{\psi(\zeta')/2}
\times\left(1\vee\n \left(G_++\gamma G_\Delta\right)^{-1} \n_{\mathit{l}^2}\right)^{3q}\left(1\vee\n Z_++\gamma Z_\Delta \n_{\mathit{l}^2}\right)^{q}
\nonumber
\\
&+
2^{2q-1}(1\vee C)^q\left(2^{\zeta-1}(1+\gamma^{\zeta})(k+1)^{\zeta}p^{\zeta}\tilde{M}(\zeta)\right)^{\frac{q}{\zeta}}
\nonumber
\\
&\times \left(p(k+1)\left( 1\vee\left(\left(\frac{2(1+\gamma)}{\delta\n \left(G_++\gamma G_\Delta\right)^{-1}\n_{\mathit{l}^2}}\right)^{\frac{1}{\alpha\zeta'\frac{\eta'-1}{\eta'}}}\left(\frac{pM_w(\zeta')^{\frac{1}{\eta'}}}{\frac{\zeta'}{2\eta'-1}}+1\right)^{\frac{1}{\alpha\zeta'\frac{\eta'-1}{\eta'}}}\right)\sigma_w(\zeta')^{\frac{1}{\alpha\zeta'}}\right)^{\alpha\zeta'-1}
\right.\nonumber
\\
&\left.+
p(k+1)\left(\frac{(\alpha\zeta'-1)4p^2(k+1)^2}{\delta^2\n \left(G_++\gamma G_\Delta\right)^{-1}\n_{\mathit{l}^2}^2}\right)^{\frac{\alpha\zeta'-1}{1-4\alpha}}
+
\sigma_w(\zeta') \right)^{1-\frac{q}{\zeta}}
\times \left(n_{A_0}\wedge...\wedge n_{A_k}\right)^{\left(1-\frac{q}{\zeta}\right)(1-\alpha\zeta')}
\end{align}
\normalsize
By noting that $\n \left(G_++\gamma G_\Delta\right)^{-1}\n_{\mathit{l}^2}^{-1}\le \n G_++\gamma G_\Delta\n_{\mathit{l}^2}$, we can then trivially dominate the right-most side of \eqref{hejsan_},
\small
\begin{align}\label{hejsan}
&Var_q\left(\hat{\beta}_\gamma(\textbf{n})\| C_\textbf{n}\right)
\nonumber
\\
&\le
2^{2q+1}(1+\gamma)^{q+\frac{\eta'}{\eta'-1}\left(1-\frac{q}{\zeta}\right)}(k+1)^{q+\left(1+\frac{2(\alpha\zeta'-1)}{1-4\alpha}\right)\left(1-\frac{q}{\zeta}\right)}p^{(2\psi(\zeta'))\vee\left(q+1+\left(\left(\frac{\eta'}{\eta'-1}\right)\vee\left(\frac{2(\alpha\zeta'-1)}{1-4\alpha}\right)\right)\left(1-\frac{q}{\zeta}\right)\right)}(1+\delta)^{3q}
\nonumber
\\
&\times(1\vee M(2\psi(\zeta'))\vee\tilde{M}(\zeta)\vee M_w(\zeta'))^{1+\frac{1}{\eta'-1}} (\psi(\zeta')-1)^{\psi(\zeta')/2}
\left(\n G_++\gamma G_\Delta\n_{\mathit{l}^2}\vee\n \left(G_++\gamma G_\Delta\right)^{-1}\n_{\mathit{l}^2}\right)^{\left(\left(\left(\frac{\eta'}{\eta'-1}\right)\vee\left(\frac{2\alpha\zeta'-1}{1-4\alpha}\right)\right)\left(1-\frac{q}{\zeta}\right)\right)\vee 3q}
\nonumber
\\
&\left(1\vee \n Z_++\gamma Z_\Delta\n_{\mathit{l}^2}\right)^q
\left(1\vee\frac{1}{\delta}\right)^{\left(\left(\frac{\eta'}{\eta'-1}\right)\vee\left(\frac{2\alpha\zeta'-1}{1-4\alpha}\right)\right)\left(1-\frac{q}{\zeta}\right)}\left(1\vee\frac{2\eta'-1}{\zeta'}\right)^{\frac{\eta'}{\eta'-1}}
\times\frac{1}{\left(n_{A_0}\wedge...\wedge n_{A_k}\right)^{\frac{q}{2}\wedge(\alpha\zeta'-1)}}
\end{align}
\normalsize
\normalsize
Letting $\eta'=\frac{\zeta'+2}{4}$, we can now set

\begin{align}\label{hejsan1}
D&= 2^{2q+1}(1+\gamma)^{q+\left(1+\frac{4}{\zeta'+2}\right)\left(1-\frac{q}{\zeta}\right)}(k+1)^{q+\left(1+\frac{2(\alpha\zeta'-1)}{1-4\alpha}\right)\left(1-\frac{q}{\zeta}\right)}p^{(2\psi(\zeta'))\vee\left(q+1+\left(\left(1+\frac{4}{\zeta'+2}\right)\vee\left(\frac{2(\alpha\zeta'-1)}{1-4\alpha}\right)\right)\left(1-\frac{q}{\zeta}\right)\right)}(1+\delta)^{3q}
\nonumber
\\
&\times(1\vee M(2\psi(\zeta'))\vee\tilde{M}(\zeta)\vee M_w(\zeta'))^{1+\frac{4}{\zeta'+2}} (\psi(\zeta')-1)^{\psi(\zeta')/2}
\nonumber
\\
&\times\left(\n G_++\gamma G_\Delta\n_{\mathit{l}^2}\vee\n 
\left(G_++\gamma G_\Delta\right)^{-1}\n_{\mathit{l}^2}\right)^{\left(\left(\left(1+\frac{4}{\zeta'+2}\right)\vee\left(\frac{2\alpha\zeta'-1}{1-4\alpha}\right)\right)\left(1-\frac{q}{\zeta}\right)\right)\vee 3q}
\nonumber
\\
&\times\left(1\vee \n Z_++\gamma Z_\Delta\n_{\mathit{l}^2}\right)^q\left(1\vee\frac{1}{\delta}\right)^{\left(\left(1+\frac{4}{\zeta'+2}\right)\vee\left(\frac{2\alpha\zeta'-1}{1-4\alpha}\right)\right)\left(1-\frac{q}{\zeta}\right)}
\end{align}
\normalsize
and
\tiny
\begin{align*}
&N= \left(\frac{(p+1)(M_w(\zeta')\vee 1)(k+1)^{\zeta'}}{1-\left(\frac{3}{4}\right)^{\frac{1}{k+1}}}\right)^{\frac{1}{\alpha\zeta'-1}}
\\
&\bigvee\left(4p(k+1)\left(1\vee\left(\left(\frac{2(1+\gamma)}{\delta\n \left(G_++\gamma G_\Delta\right)^{-1}\n_{\mathit{l}^2}}\right)^{\frac{1}{\alpha\zeta'\frac{\zeta'-1}{\zeta'+1}}}\left(pM(\zeta+1)^{\frac{1}{\eta'}}\right)\right)^{\frac{1}{\alpha\zeta'\frac{\zeta'-1}{\zeta'+1}}}\sigma_w(\zeta')^{\frac{1}{\alpha\zeta'}}\right)^{\alpha\zeta'-1} 
+
4p(k+1)\left(\frac{(\alpha\zeta'-1)4p^2(k+1)^2}{\delta^2\n \left(G_++\gamma G_\Delta\right)^{-1}\n_{\mathit{l}^2}^2}\right)^{\frac{\alpha\zeta'-1}{1-4\alpha}}\right)^{\frac{1}{\alpha\zeta'-1}}
\end{align*}
\normalsize
\end{proof}
\begin{proof}[Proof of (2)]
This readily follows from Cauchy-Schwartz inequality since
$$\tilde{M}(\zeta)=\E\left[\left|  X^{A_i}_u(l)Y^{A_i}_u \right|^{\zeta}\right]\le \sqrt{\max_{i,l}\E\left[\left|  X^{A_i}_u(l) \right|^{2\zeta}\right]}\sqrt{\max_{i}\E\left[\left|  Y^{A_i}_u \right|^{2\zeta}\right]}\le M(\zeta)\sqrt{\bigvee_{i}\E\left[\left|  Y^{A_i}_u \right|^{2\zeta}\right]}. $$
Replacing $\tilde{M}$ with the right-most expression above, gives us $\tilde{D}$.
\end{proof}

\bibliographystyle{imsart-nameyear} 


\end{document}